\documentclass[a4paper,11pt]{article}
\usepackage{amsmath,amsthm,amssymb,array}

{\theoremstyle{definition}\newtheorem{definition}{Definition}[section]
\newtheorem{notation}[definition]{Notation}

\newtheorem{terminology}[definition]{Terminology}
\newtheorem{remark}[definition]{Remark}
\newtheorem{example}[definition]{Example}
\newtheorem{examples}[definition]{Examples}
\newtheorem{property}[definition]{Property}}

\newtheorem{proposition}[definition]{Proposition}
\newtheorem{lemma}[definition]{Lemma}
\newtheorem{theorem}[definition]{Theorem}
\newtheorem{defprop}[definition]{Definition/Proposition}
\newtheorem{corollary}[definition]{Corollary}

\newtheorem{step}{Step}

\newcommand{\cK}{\mathcal{K}}
\newcommand{\Qh}{\widehat{\Q}}
\newcommand{\transp}{\! ^\top}
\newcommand{\cP}{\mathcal{P}}
\newcommand{\pitil}{\widetilde{\pi}}
\newcommand{\Deltatil}{\widetilde{\Delta}}
\newcommand{\rL}{\operatorname{L}}

\newcommand{\PGL}{\operatorname{PGL}}
\newcommand{\PSL}{\operatorname{PSL}}
\newcommand{\Ind}{\operatorname{Ind}}
\newcommand{\si}{\sigma}
\newcommand{\recht}{\rightarrow}
\newcommand{\Ker}{\operatorname{Ker}}
\newcommand{\Out}{\operatorname{Out}}
\newcommand{\Aut}{\operatorname{Aut}}
\newcommand{\Char}{\operatorname{Char}}
\newcommand{\cZ}{\mathcal{Z}}
\newcommand{\actson}{\curvearrowright}

\newcommand{\al}{\alpha}
\newcommand{\ox}{\overline{x}}
\newcommand{\oy}{\overline{y}}

\newcommand{\Ad}{\operatorname{Ad}}
\newcommand{\om}{\omega}
\newcommand{\cG}{\mathcal{G}}

\newcommand{\vphi}{\varphi}
\newcommand{\Stab}{\operatorname{Stab}}
\newcommand{\cD}{\mathcal{D}}
\newcommand{\GL}{\operatorname{GL}}
\newcommand{\Z}{\mathbb{Z}}
\newcommand{\cU}{\mathcal{U}}
\newcommand{\SL}{\operatorname{SL}}
\newcommand{\Q}{\mathbb{Q}}
\newcommand{\be}{\beta}

\newcommand{\id}{\operatorname{id}}

\newcommand{\C}{\mathbb{C}}
\newcommand{\cF}{\mathcal{F}}
\newcommand{\eps}{\varepsilon}
\newcommand{\cC}{\mathcal{C}}

\newcommand{\cL}{\mathcal{L}}

\newcommand{\M}{\operatorname{M}}
\newcommand{\ot}{\otimes}
\newcommand{\N}{\mathbb{N}}
\newcommand{\B}{\operatorname{B}}
\newcommand{\lspan}{\operatorname{span}}

\newcommand{\Tr}{\operatorname{Tr}}

\newcommand{\R}{\mathbb{R}}

\newcommand{\Comm}{\operatorname{Comm}}
\newcommand{\Fix}{\operatorname{Fix}}
\newcommand{\Norm}{\operatorname{Norm}}
\newcommand{\dpr}{^{\prime\prime}}
\newcommand{\QN}{\operatorname{QN}}
\newcommand{\vtil}{\widetilde{v}}
\newcommand{\psitil}{\widetilde{\psi}}

\newcommand{\Ntil}{\widetilde{N}}
\newcommand{\Hrep}{H_{\text{\rm rep}}}
\newcommand{\Hiso}{H_{\text{\rm iso}}}
\newcommand{\Hincl}{H_{\text{\rm incl}}}
\newcommand{\Hred}{H_{\text{\rm red}}}
\newcommand{\D}{\operatorname{D}}
\newcommand{\cN}{\mathcal{N}}
\newcommand{\pS}{\operatorname{S}}
\newcommand{\PcU}{\mathalpha{\operatorname{P}\!\mathcal{U}}}
\newcommand{\FAlg}{\operatorname{FAlg}}
\newcommand{\ibar}{\overline{i}}
\newcommand{\rP}{\operatorname{P}}
\newcommand{\Ytil}{\widetilde{Y}}
\newcommand{\Lambdatil}{\widetilde{\Lambda}}
\newcommand{\rig}[1]{{\Gamma \! _{#1}}}
\newcommand{\lef}[1]{ { _{#1}\! \Gamma}}
\newcommand{\diag}{\operatorname{diag}}
\newcommand{\Hecke}{\mathcal{H}}
\newcommand{\Perm}{\operatorname{Perm}}
\newcommand{\He}{\mathcal{H}_\text{\rm rep}}
\newcommand{\HeC}{\mathcal{H}_{\text{\rm rep}}^\C}
\newcommand{\Repfin}{\operatorname{Rep}_{\text{\rm fin}}}
\newcommand{\rtimesfull}{\rtimes_\text{\rm f}}
\newcommand{\cstarmax}{C^*_\text{\rm max}}
\newcommand{\cA}{\mathcal{A}}
\newcommand{\ovt}{\overline{\otimes}}
\newcommand{\cJ}{\mathcal{J}}

\newcommand{\notembed}[1]{\underset{#1}{\not\prec}}
\newcommand{\embed}[1]{\underset{#1}{\prec}}
\newcommand{\fembed}[1]{\overset{\text{\rm f}}{\underset{#1}{\prec}}}

\newcommand{\fembedr}[1]{\overset{\text{\rm f}}{\underset{#1}{\succ}}}

\begin{document}
\begin{center}
{\LARGE\bf Explicit computations of all finite index \vspace{2mm}\\ bimodules for a family of II$_1$ factors}

\bigskip

{\sc by Stefaan Vaes\footnote{Partially supported by Research Programme G.0231.07 of the Research Foundation -- Flanders (FWO) and the Marie Curie Research Training Network Non-Commutative Geometry MRTN-CT-2006-031962.}}
%
%
\end{center}

{Department of Mathematics, K.U.Leuven, Celestijnenlaan 200B, B-3001 Leuven, Belgium \\
E-mail: stefaan.vaes@wis.kuleuven.be}

\begin{abstract}
\noindent We study II$_1$ factors $M$ and $N$ associated with good generalized Bernoulli actions of groups having an infinite almost normal subgroup with the relative property~(T). We prove the following rigidity result~: every finite index $M$-$N$-bimodule (in particular, every isomorphism between $M$ and $N$) is described by a commensurability of the groups involved and a commensurability of their actions. The fusion algebra of finite index $M$-$M$-bimodules is identified with an extended Hecke fusion algebra, providing the first explicit computations of the fusion algebra of a II$_1$ factor. We obtain in particular explicit examples of II$_1$ factors with trivial fusion algebra, i.e.\ only having trivial finite index subfactors.
\end{abstract}

\section*{Introduction}

To every probability measure preserving action $\Gamma \actson (X,\mu)$ of a countable group, is associated a tracial von Neumann algebra $\rL^\infty(X) \rtimes \Gamma$, through the group measure space construction of Murray and von Neumann \cite{MvN}. In the passage from group actions to von Neumann algebras, a lot of information gets lost. Indeed, by the celebrated results of Connes, Feldman, Ornstein and Weiss \cite{CFW,OW}, all free ergodic actions of amenable groups (and even all amenable type II$_1$ equivalence relations) systematically yield the same von Neumann algebra, called the hyperfinite II$_1$ factor.

Recently, Sorin Popa, in his breakthrough articles \cite{P1,P2}, proved a completely opposite rigidity result~: for the first time, he was able to provide a family of group actions such that isomorphism of the crossed product von Neumann algebras, implies isomorphism of the groups involved and conjugacy of their actions. More precisely, Popa proves in \cite{P2} the following von Neumann strong rigidity theorem~: let $\Gamma \actson (X,\mu)$ be a free ergodic action of an ICC \emph{$w$-rigid group,} i.e.\ a group admitting an infinite normal subgroup with the relative property~(T), and let $\Lambda \actson (Y_0,\eta_0)^\Lambda$ be a Bernoulli action of an ICC group $\Lambda$. If the corresponding group measure space II$_1$ factors are isomorphic, then the groups $\Gamma$ and $\Lambda$ are isomorphic and their actions conjugate.

The crucial idea of Popa is the \emph{deformation/rigidity principle}. One studies von Neumann algebras that exhibit both a deformation property (e.g.\ a specific flow of automorphisms, or a sequence of completely positive unital maps tending to the identity) and a rigidity property (e.g.\ a subalgebra with the relative property~(T)). The tension between both properties determines in a sense the position of the rigid part and allows in certain cases to completely unravel the structure of the studied von Neumann algebra. The deformation/rigidity principle has been successfully applied in a lot of articles. Without being complete, we cite \cite{IPP,P0,P1,P2,PBetti,PV,FV,houdayer,Ioana} and we explain some aspects of these works below. We also refer to \cite{VBour} for a survey of some of these results.

The deformation/rigidity principle allows in particular to compute \emph{invariants of II$_1$ factors.} In \cite{PBetti}, Popa proved that the group von Neumann algebra $\cL(\SL(2,\Z) \ltimes \Z^2)$ has \emph{trivial fundamental group.} This was the first such example, answering a question of Kadison that remained open since 1967. Here, it should be noticed that Connes proved in \cite{C1} that the fundamental group of the group von Neumann algebra $\cL(\Gamma)$ is countable whenever $\Gamma$ is a group with property (T) and infinite conjugacy classes (ICC).

In \cite{P1,P2}, Popa made a thorough study of \emph{Bernoulli actions} $\Gamma \actson (X_0,\mu_0)^\Gamma$ and their non-commutative versions, called Connes-St{\o}rmer Bernoulli actions. In \cite{P1}, this lead to the first constructions of II$_1$ factors with a \emph{prescribed countable subgroup of $\R^*_+$ as fundamental group.} Alternative constructions have been given since then in \cite{IPP,houdayer}. In \cite{P2}, Popa proves the von Neumann strong rigidity theorem stated in the first paragraph. As an application, he gets the following description of the \emph{outer automorphism group} of the associated II$_1$ factors. Given the Bernoulli action $\Gamma \actson (X,\mu) = (X_0,\mu_0)^\Gamma$ of an ICC $w$-rigid group $\Gamma$, the outer automorphism group of the associated II$_1$ factor is the semidirect product of the group of characters of $\Gamma$ and the normalizer of $\Gamma$ inside $\Aut(X,\mu)$. Up to now, the actual computation of this normalizer remains an open problem though.

In \cite{IPP}, the deformation/rigidity principle was applied to study \emph{amalgamated free product II$_1$ factors.} From the many far-reaching results obtained in \cite{IPP}, we quote the existence theorem of II$_1$ factors $M$ with $\Out(M)$ a prescribed abelian compact group. In particular, it was shown that the outer automorphism group of a II$_1$ factor can be trivial, answering a question posed by Connes in 1973. Using the same techniques, it was shown in \cite{FV}, that there exist II$_1$ factors with $\Out(M)$ an arbitrary compact group.

Some of the results on Bernoulli actions obtained in \cite{P1,P2}, were extended by Popa and the author \cite{PV}, to include \emph{generalized Bernoulli actions} $\Gamma \actson (X_0,\mu_0)^I$, associated with an action of $\Gamma$ on a countable set $I$. As a result, the first \emph{explicit examples} of II$_1$ factors with \emph{trivial outer automorphism group} were given. It should be noticed that the shift from plain to generalized Bernoulli actions is not only technical in nature~: the former are \emph{mixing} and this is extensively used in \cite{P1,P2}, while the latter are only \emph{weakly mixing}.

\emph{In the current article,} we study \emph{bimodules (Connes' correspondences) of finite Jones index} between II$_1$ factors given by generalized Bernoulli actions. Bimodules between von Neumann algebras were studied by Connes (see V.Appendix~B of \cite{CBook}) and Popa \cite{PCorr}. An $M$-$N$-bimodule of finite Jones index (see \cite{Jones}) can be considered as a \emph{commensuration of $M$ and $N$}, i.e.\ an isomorphism modulo finite index. Using the Connes tensor product, the set $\FAlg(M)$ of (equivalence classes of) finite index $M$-$M$-bimodules carries the structure of a \emph{fusion algebra} and contains the outer automorphism group $\Out(M)$ as group-like elements.

As a natural follow-up of \cite{P1,P2,PV}, we provide a family of good generalized Bernoulli actions of groups admitting an infinite almost normal subgroup with the relative property~(T) and prove the following rigidity property~:
any finite index bimodule between the associated II$_1$ factors comes from a commensurability of the groups and a commensurability of the actions. This allows us to get the following results.
\begin{itemize}
\item We provide the first \emph{explicit computations of fusion algebras} for a family of II$_1$ factors. If the II$_1$ factor $M$ is defined by a good generalized Bernoulli action $\Gamma \actson (X,\mu) = (X_0,\mu_0)^I$, the fusion algebra $\FAlg(M)$ is identified with the \emph{extended Hecke fusion algebra} $\He(\Gamma < G)$ of the Hecke pair $\Gamma < G$, where $G$ denotes the commensurator of $\Gamma$ inside the group of permutations $\Perm(I)$. Loosely speaking, the extended Hecke algebra $\He(\Gamma < G)$ is an extension of the usual Hecke algebra $\Hecke(\Gamma < G)$ by the fusion algebra of finite dimensional unitary representations of $\Gamma$. In \ref{ex.hecke} below, we give several concrete examples yielding II$_1$ factors whose fusion algebras are the extended Hecke algebras of Hecke pairs appearing naturally in arithmetic.
\item We give the first \emph{explicit examples of II$_1$ factors $M$ with trivial fusion algebra,} associated with the generalized Bernoulli action $(\SL(2,\Q) \ltimes \Q^2) \actson (X_0,\mu_0)^{\Q^2}$ and a scalar $2$-cocycle. Equivalently, every finite index subfactor $N \subset M$ is trivial, i.e.\ isomorphic with $1 \ot N \subset \M_n(\C) \ot N$. Note that we proved in \cite{V1} the existence of such II$_1$ factors $M$, using the techniques of \cite{IPP}.
\item Compared to \cite{PV}, we impose less stringent conditions on the generalized Bernoulli actions involved and obtain more general results on outer automorphism groups. We prove that the actions $\PSL(n,\Z) \actson (X_0,\mu_0)^{\rP(\Q^n)}$ for $n$ odd and $n \geq 3$, provide II$_1$ factors with \emph{trivial outer automorphism group}. In fact, we provide a concrete construction procedure to obtain II$_1$ factors with a \emph{prescribed countable group} as an outer automorphism group. The case of groups of finite presentation was dealt with in \cite{PV}.
\end{itemize}

{\bf Acknowledgment.}$\;$ I would like to thank the referee for pointing out a gap in the original proof of Lemma \ref{lem.control}.

\section{Preliminaries and notations} \label{sec.prelim}

All von Neumann algebras in this article have separable predual and all Hilbert spaces are separable.

Let $M$ be a von Neumann algebra. One calls $H_M$ a \emph{(right) $M$-module} if $H$ is a Hilbert space equipped with a weakly continuous right module action of $M$. If $M$ is a II$_1$ factor and if we denote by $\rL^2(M)$ the Hilbert space obtained by the GNS construction with respect to the unique tracial state of $M$, every $M$-module $H_M$ is isomorphic with an $M$-module of the form $p (\ell^2(\N) \ot \rL^2(M))$, for some projection $p \in \B(\ell^2(M)) \ovt M$. The projection $p$ is uniquely determined up to equivalence of projections and one defines $\dim(H_M) := (\Tr \ot \tau)(p)$. All this was already known to Murray and von Neumann (Theorem~X in \cite{MvN}).

The \emph{Jones index} of a subfactor $N \subset M$ of a II$_1$ factor is defined as $[M : N] := \dim(\rL^2(M)_N)$, see \cite{Jones}. If $[M : N] < \infty$, we call $N \subset M$ a \emph{finite index subfactor} or a \emph{finite index inclusion.}

If $(M,\tau)$ is a tracial von Neumann algebra with possibly non-trivial center, the dimension $\dim(H_M)$ of a right $M$-module $H_M$ is defined similarly, but depends on the choice of trace $\tau$. In this article, there will always be an obvious choice of $\tau$, so that we freely use the notation $\dim(H_M)$.

For any von Neumann algebra $M$, we denote $M^n := \M_n(\C) \ot M$ and $M^\infty := \B(\ell^2(\N)) \ovt M$.

Let $N$ and $M$ be von Neumann algebras. An \emph{$N$-$M$-bimodule $_N H_M$} is a Hilbert space $H$ equipped with commuting, weakly continuous, left $N$-module and right $M$-module actions. An $N$-$M$-bimodule $_N H_M$ between tracial von Neumann algebras $(N,\tau_1)$ and $(M,\tau_2)$ is said to be \emph{of finite Jones index} if $\dim(_N H) < \infty$ and $\dim(H_M) < \infty$. Bimodules between von Neumann algebras were studied by Connes (see V.Appendix~B in \cite{CBook}) who called them \emph{correspondences}, and by Popa \cite{PCorr}.

If $M$ is a II$_1$ factor, $\FAlg(M)$ is defined as the set of equivalence classes of finite index $M$-$M$-bimodules. We call $\FAlg(M)$ the \emph{fusion algebra of the II$_1$ factor $M$.}

First recall that an \emph{abstract fusion algebra $\cA$} is a free $\N$-module $\N[\cG]$ equipped with the following additional structure~:
\begin{itemize}
\item an associative and distributive product operation, and a multiplicative unit element $e \in \cG$,
\item an additive, anti-multiplicative, involutive map $x \mapsto \overline{x}$, called \emph{conjugation},
\end{itemize}
satisfying Frobenius reciprocity: defining the numbers $m(x,y;z) \in \N$ for $x,y,z \in \cG$ through the formula
$$x y = \sum_z m(x,y;z) z$$
one has $m(x,y;z) = m(\ox,z;y) = m(z,\oy;x)$ for all $x,y,z \in \cG$. The base $\cG$ of the fusion algebra $\cA$ is canonically determined~: these are exactly the non-zero elements of $\cA$ that cannot be expressed as the sum of two non-zero elements. The elements of $\cG$ are called the \emph{irreducible elements} of the fusion algebra $\cA$.

If $M$ is a II$_1$ factor, the fusion algebra structure on $\FAlg(M)$ is given by the direct sum and the \emph{Connes tensor product} of $M$-$M$-bimodules. Whenever $\psi : M \recht p M^n p$ is a finite index inclusion, define the $M$-$M$-bimodule $H(\psi)$ on the Hilbert space $p (\M_{n,1}(\C) \ot \rL^2(M))$ with left and right module action given by $a \cdot \xi = \psi(a) \xi$ and $\xi \cdot a = \xi a$. Every element of $\FAlg(M)$ is of the form $H(\psi)$ for a finite index inclusion $\psi$ uniquely determined up to conjugacy. The Connes tensor product of $H(\psi)$ and $H(\rho)$ is given by $H(\psi) \ot_M H(\rho) \cong H((\id \ot \rho)\psi)$.

We say that $M$ and $N$ are \emph{commensurable II$_1$ factors} if there exists a non-zero finite index $N$-$M$-bimodule.

Throughout this article, $\Gamma \actson (X,\mu)$ denotes a \emph{probability measure preserving action of a countable group $\Gamma$ on the standard probability space $(X,\mu)$.} We will always write $\Gamma$ acting on the right on $X$. Associated to $\Gamma \actson (X,\mu)$ is the so-called group measure space, or crossed product, von Neumann algebra $M=\rL^\infty(X) \rtimes \Gamma$. As a tracial von Neumann algebra, $(M,\tau)$ is uniquely characterized by the following properties~:
\begin{itemize}
\item $M$ contains a copy of $\rL^\infty(X)$ and a copy of $\Gamma$ as unitaries $(u_g)_{g \in \Gamma}$ satisfying $u_g u_h = u_{gh}$ for all $g,h \in \Gamma$,
\item $u_g F(\cdot) u_g^* = F(\; \cdot \; g)$ for all $F \in \rL^\infty(X)$ and $g \in \Gamma$,
\item $\tau(F u_g) = 0$ when $g \neq e$ and $\tau(F) = \int F \; d\mu$ for all $F \in \rL^\infty(X)$.
\end{itemize}
If $\Gamma \actson (X,\mu)$ and if $\Omega$ is a \emph{scalar $2$-cocycle} on $\Gamma$, the \emph{cocycle crossed product} $\rL^\infty(X) \rtimes_\Omega \Gamma$ is defined entirely similarly, the only difference being the relation $u_g u_h = \Omega(g,h) u_{gh}$. Note that the $2$-cocycle relation that we impose on $\Omega$ is exactly the one that makes this last product associative.

Let $\Gamma \actson (X,\mu)$ and denote by $(\si_g)$ the associated group of automorphisms of $\rL^\infty(X)$ given by $\si_g(F(\cdot)) = F(\; \cdot \; g)$. Then, $\Gamma \actson (X,\mu)$ is called
\begin{itemize}
\item \emph{mixing}, if $\lim_{g \recht \infty} \tau(\si_g(a) b) = \tau(a) \tau(b)$ for all $a, b \in \rL^\infty(X)$~;
\item \emph{weakly mixing}, if there exists a sequence $(g_n) \in \Gamma$ such that $\lim_{n \recht \infty} \tau(\si_{g_n}(a) b) = \tau(a) \tau(b)$ for all $a, b \in \rL^\infty(X)$.
\end{itemize}
The following three properties of a probability measure preserving action $\Gamma \actson (X,\mu)$ are equivalent~: weak mixing; the Hilbert space $\rL^2(X) \ominus \C 1$ does not have finite dimensional $\Gamma$-invariant subspaces; the diagonal action $\Gamma \actson X \times X$ is ergodic. We refer e.g.\ to Appendix~D of \cite{VBour} for proofs.

A group $\Gamma$ is said to have \emph{infinite conjugacy classes (ICC)} if $\{g h g^{-1} \mid g \in \Gamma\}$ is infinite for every $h \neq e$. More generally, we say that a subgroup $\Gamma_0 < \Gamma$ has the \emph{relative ICC property} if $\{g h g^{-1} \mid g \in \Gamma_0\}$ is infinite for all $h \in \Gamma - \{e\}$.

If $M$ is a von Neumann algebra with von Neumann subalgebra $A \subset M$, we define the subset $\QN_M(A)$ as the set of $x \in M$ such that there exist $x_1,\ldots,x_n,y_1,\ldots,y_m \in M$ satisfying
$$x A \subset \sum_{i=1}^n A x_i \quad\text{and}\quad A x \subset \sum_{i=1}^m y_i A \; .$$
Note that $\QN_M(A)$ is a unital $^*$-subalgebra of $M$ containing $A$. The generated von Neumann algebra $\QN_M(A)\dpr$ is called the \emph{quasi-normalizer} of $A$ inside $M$. If $\QN_M(A)\dpr = M$, we say that the inclusion $A \subset M$ is \emph{quasi-regular.}

If $(X_0,\mu_0)$ is a probability space and if $I$ is a countable set, we denote by $(X_0,\mu_0)^I$ the \emph{infinite product probability space.} Whenever $J \subset I$ or $i \in I$, we consider the obvious von Neumann subalgebras $\rL^\infty(X_0^J)$ and $\rL^\infty(X_0^i)$ of $\rL^\infty(X_0^I)$.

\section{Main results} \label{sec.main-results}

We gather the main results of the article in this section. All the proofs, including a more detailed discussion of the given examples, are provided in Section \ref{sec.proofs}, based of course on the work done in Sections \ref{sec.intertwining}-\ref{sec.partB}.

We make a detailed study of the II$_1$ factors given by generalized Bernoulli actions. These actions are defined as follows.

\begin{definition}
Let $(X_0,\mu_0)$ be any standard probability space.
If the countable group $\Gamma$ acts on the countable set $I$, we call the action $\Gamma \actson (X_0,\mu_0)^I$ a \emph{generalized Bernoulli action}. We call $(X_0,\mu_0)$ the \emph{base space} of the generalized Bernoulli action and we note that it is allowed to be atomic. But we assume of course that $\mu_0$ is not concentrated on one atom.
\end{definition}

The following is the main theorem of the article~: we describe entirely explicitly all finite index bimodules between the II$_1$ factors $N$ and $M$ coming from \lq good\rq\ generalized Bernoulli actions of \lq good\rq\ groups. These kind of good actions are introduced in Definitions \ref{def.conditions} and \ref{def.goodgood} below.

\begin{theorem} \label{thm.main}
Let $\Lambda \actson J$ and $\Gamma \actson I$ be good actions of good groups (see Def.\ \ref{def.goodgood}). Take scalar $2$-cocycles $\omega \in Z^2(\Lambda,S^1)$ and $\Omega \in Z^2(\Gamma,S^1)$. Consider the generalized Bernoulli actions $\Lambda \actson (Y,\eta)=(Y_0,\eta_0)^J$ and $\Gamma \actson (X,\mu) = (X_0,\mu_0)^I$. Define the II$_1$ factors
$$N = \rL^\infty(Y) \rtimes_\omega \Lambda \quad\text{and}\quad \rL^\infty(X) \rtimes_\Omega \Gamma \; .$$
Suppose that $_N H_M$ is a finite index $N$-$M$-bimodule. Then, the following holds.
\begin{itemize}
\item The actions $\Lambda \actson J$ and $\Gamma \actson I$ are commensurable~: there exists a bijection $\Delta : J \recht I$ and an isomorphism $\delta : \Lambda_1 \recht \Gamma_1$ between finite index subgroups of $\Lambda$, resp.\ $\Gamma$, satisfying $\Delta(g \cdot j) = \delta(g) \cdot \Delta(j)$ for all $j \in J$, $g \in \Lambda_1$.
\item The probability spaces $(Y_0,\eta_0)$ and $(X_0,\mu_0)$ are isomorphic.
\item There exists a finite dimensional projective representation $\pi$ of $\Lambda_1$ such that $\Omega_\pi (\Omega \circ \delta) = \omega$ on $\Lambda_1$.
\end{itemize}
Moreover, the $N$-$M$-bimodule $_N H_M$ can be entirely described in terms of the above data. We refer to Thm.\ \ref{thm.cartan-preserving-bimodules} and Prop.\ \ref{prop.commensuration-bernoulli} for a precise statement.
\end{theorem}

The following consequences will be deduced from Theorem \ref{thm.main}.

\begin{itemize}
\item Corollary \ref{cor.example} provides the first explicit example of a II$_1$ factor $M$ without non-trivial finite index bimodules. Equivalently, every finite index subfactor $N \subset M$ is isomorphic with the trivial subfactor $1 \ot N \subset \M_n(\C) \ot N$. The existence of such II$_1$ factors had been shown before by the author in \cite{V1}.

\item When $M = \rL^\infty\bigl((X_0,\mu_0)^I\bigr) \rtimes \Gamma$ is as above and $(X_0,\mu_0)$ is atomic with unequal weights, the \emph{fusion algebra} of finite index $M$-$M$-bimodules, can be identified with the \emph{extended Hecke fusion algebra} of the Hecke pair given by $\Gamma < \Comm_{\Perm(I)}(\Gamma)$~: see Theorem \ref{thm.fusion-algebra}. Here $\Comm_{\Perm(I)}(\Gamma)$ denotes the commensurator of $\Gamma$ inside the group $\Perm(I)$ of permutations of $I$, see Definition \ref{def.conditions}. This provides the first explicit computations of the fusion algebra for a family of II$_1$ factors.

\item In \ref{ex.hecke}, we provide several examples, yielding concrete II$_1$ factors whose fusion algebras are given by the extended Hecke fusion algebra of Hecke pairs like $\SL(n,\Z) < \GL(n,\Q)$ or \linebreak $(R^* \ltimes R) < (\Q^* \ltimes \Q)$ where $\Z \subset R \subset \Q$ is a subring sitting strictly between $\Z$ and $\Q$.

\item The \emph{outer automorphism group} $\Out(M)$ can be explicitly computed for the generalized Bernoulli II$_1$ factors above, see Corollary \ref{cor.out}. In \ref{ex.out} this yields rather easy II$_1$ factors without outer automorphisms.

\item Every countable group arises as the outer automorphism group of a II$_1$ factor.
\end{itemize}

We introduce the necessary properties of actions and groups in the following two definitions.

\begin{definition} \label{def.conditions}
Let $\Gamma \actson I$ be an action of the group $\Gamma$ on the set $I$.
\begin{itemize}
\item We say that $J \subset I$ has \emph{infinite index} if $I \neq \bigcup_{i=1}^n g_i J$ for all $n \in \N$ and all $g_1,\ldots,g_n \in \Gamma$.
\end{itemize}
We consider the following properties of $\Gamma \actson I$.
\begin{itemize}
\item[(C1)] The set $I$ is infinite, the action $\Gamma \actson I$ is transitive and $\Stab i \actson I \setminus \{i\}$ has infinite orbits for one (equivalently all) $i \in I$.
\item[(C2)] The \emph{minimal condition on stabilizers :} there is no infinite sequence $(i_n)$ in $I$ such that $\Stab\{i_0,\ldots,i_n\}$ is strictly decreasing.
\item[(C3)] A \emph{faithfulness condition :} for every $g \in \Gamma$ with $g \neq e$, the subset $\Fix g \subset I$ has infinite index in the sense defined above.
\end{itemize}
Let $\Gamma$ be a group and $\Gamma_0 < \Gamma$ a subgroup.
\begin{itemize}
\item The \emph{commensurator of $\Gamma_0$ inside $\Gamma$} is defined as
$$\Comm_\Gamma(\Gamma_0) := \{g \in \Gamma \mid g \Gamma_0 g^{-1} \cap \Gamma_0 \;\;\text{has finite index in both}\;\; g \Gamma_0 g^{-1} \;\;\text{and}\;\; \Gamma_0 \} \; .$$
\item The subgroup $\Gamma_0 < \Gamma$ is called \emph{almost normal} if $\Comm_\Gamma(\Gamma_0) = \Gamma$. Under this condition, one also calls $\Gamma_0 < \Gamma$ a \emph{Hecke pair}.
\end{itemize}
\end{definition}

\begin{definition} \label{def.goodgood}
We say that $\Gamma \actson I$ is a \emph{good action of a good group} if $\Gamma$ is a group admitting an infinite almost normal subgroup with the relative property (T) and if the action $\Gamma \actson I$ satisfies conditions (C1), (C2) and (C3).
\end{definition}

We immediately illustrate that there are indeed plenty of examples and constructions for good actions of good groups.

\begin{examples} \label{ex.conditions}
In Definition \ref{def.conditions} above, it is of course condition (C2) that is the least intuitive. The following principles allow to construct many examples of actions satisfying (C2). As for all other results in this section, proofs are given in Section \ref{sec.proofs}.
\begin{itemize}
\item Let $V$ be a finite-dimensional vector space. Then, the actions $(\GL(V) \ltimes V) \actson V$ and $\PGL(V) \actson \rP(V)$ satisfy condition (C2).
\item Suppose that $\Gamma \actson I$ satisfies (C2) and take $I_1 \subset I$ as well as $\Gamma_1 < \Gamma$. If $I_1$ is globally $\Gamma_1$-invariant, then $\Gamma_1 \actson I_1$ satisfies (C2).
\item If both $\Gamma \actson I$ and $\Lambda \actson J$ satisfy (C2), the same is true for $(\Gamma \times \Lambda) \actson (I \times J)$.
\item Let $\Gamma$ be a group. Then, the left-right action $(\Gamma \times \Gamma) \actson \Gamma$ satisfies (C2) if and only if $\Gamma$ satisfies \emph{the minimal condition on centralizers :} there is no infinite sequence $(g_n)$ in $\Gamma$ such that $C_\Gamma(g_1,\ldots,g_n)$ is a strictly decreasing sequence of subgroups of $\Gamma$. The minimal condition on centralizers has been studied quite extensively in group theory, see e.g.\ \cite{bryant}. The following families of groups satisfy the minimal condition on centralizers~: linear groups, $C'(1/6)$-small cancelation groups, word hyperbolic groups.
\end{itemize}
Note also that for the left-right action $(\Gamma \times \Gamma) \actson \Gamma$, the following three properties are equivalent~: condition (C1), condition (C3) and the ICC property of $\Gamma$.

As a result, we list the following concrete examples of good actions of good groups.
\begin{itemize}
\item $\PSL(n,\Z) \actson \rP(\Q^n)$ and $\PSL(n,\Q) \actson \rP(\Q^n)$ for $n \geq 3$.
\item $\SL(n,\Z) \ltimes \Z^n$ acting on $\Z^n$ and $\SL(n,\Q) \ltimes \Q^n$ acting on $\Q^n$ for $n \geq 2$.
\item $(\PSL(n,\Z) \times \Gamma \times \Gamma) \actson (\rP(\Q^n) \times \Gamma)$ where $n \geq 3$ and $\Gamma$ is an arbitrary ICC group satisfying the minimal condition on centralizers.
\end{itemize}
\end{examples}

\begin{remark} \label{rem.weak-mixing-relative-T}
Let $\Gamma \actson I$ satisfy (C1), (C2) and (C3). Whenever $H < \Gamma$ is infinite and almost normal, the restricted action $H \actson I$ has infinite orbits and so, $H \actson (X_0,\mu_0)^I$ is weakly mixing. Indeed, once $H_0 := H \cap \Stab i_0$ has finite index in $H$ for some $i_0 \in I$, one constructs by induction finite index subgroups $H_n$ of $H$ and a strictly decreasing sequence $\Stab(i_0,\ldots,i_n)$ containing $H_n$. This contradicts condition (C2).

Also, conditions (C1) and (C3) imply immediately that the generalized Bernoulli action $\Gamma \actson (X_0,\mu_0)^I$ is essentially free and ergodic.
\end{remark}

\subsection{Computations of all finite index bimodules of certain II$_1$ factors}

As announced above, Theorem \ref{thm.main} allows to entirely determine all finite index $M$-$M$-bimodules for certain II$_1$ factors $M$.

\begin{corollary} \label{cor.example}
Let $\Gamma = \SL(2,\Q) \ltimes \Q^2$ act on $\Q^2$ by affine transformations. Let $\al \in \R \setminus \{0\}$. Consider the scalar $2$-cocycle
$\Omega_\al \in Z^2(\Q^2)$ defined by $\Omega_\al((x,y),(x',y')) = \exp(2\pi i \al (xy' - y x'))$. Extend $\Omega_\al$ to the whole of $\Gamma$ by $\SL(2,\Q)$-invariance.

Consider the II$_1$ factors
$$M(\al,X_0,\mu_0) := \rL^\infty\bigl((X_0,\mu_0)^{\Q^2}\bigr) \rtimes_{\Omega_\al} \Gamma \; .$$
Then, the following holds.
\begin{itemize}
\item If $M = M(\al,X_0,\mu_0)$ for an atomic $\mu_0$ with unequal weights, every finite index $M$-$M$-bimodule is a multiple of the trivial $M$-$M$-bimodule $\rL^2(M)$.
\item The II$_1$ factors $M(\al,X_0,\mu_0)$ and $M(\beta,Y_0,\eta_0)$ are commensurable if and only if $\al = \be$ and $(X_0,\mu_0) \cong (Y_0,\eta_0)$.
\end{itemize}
\end{corollary}

In order to describe in general the fusion algebra of the II$_1$ factor $\rL^\infty\bigl((X_0,\mu_0)^I\bigr) \rtimes \Gamma$ for a good action $\Gamma \actson I$, we introduce the notion of \emph{extended Hecke fusion algebra}.

Let $\Gamma < G$ be a \emph{Hecke pair}, i.e.\ $\Gamma$ is an almost normal subgroup of $G$. The usual \emph{Hecke fusion algebra} $\Hecke(\Gamma < G)$ is defined as follows
\begin{align}
& \Hecke(\Gamma < G) = \{ \xi : G \recht \N \mid \xi\;\;\text{is $\Gamma$-bi-invariant and supported on a finite subset of}\;\; \Gamma \backslash G / \Gamma \} \; , \label{eq.hecke} \\
& (\xi * \eta)(g) = \sum_{h \in \Gamma \backslash G} \xi(gh^{-1}) \; \eta(h) \quad\text{for all}\;\; \xi,\eta \in \Hecke(\Gamma < G) \; .\notag
\end{align}
We next define the \emph{extended Hecke fusion algebra} $\He(\Gamma < G)$ in such a way that there are fusion algebra homomorphisms
$$\Repfin(\Gamma) \recht \He(\Gamma < G) \recht \Hecke(\Gamma < G) \; ,$$
where $\Repfin(\Gamma)$ denotes the fusion algebra of \emph{finite dimensional unitary representations of $\Gamma$.}

As a set, $\He(\Gamma < G)$ is most conveniently defined as the set of unitary equivalence classes of finite dimensional representations of the full crossed product C$^*$-algebra $c_0(\Gamma \backslash G) \rtimesfull \Gamma$. But it is not clear to us, how to exploit this picture to write the fusion product on $\He(\Gamma < G)$. Therefore, note that
$$c_0(\Gamma \backslash G) \rtimesfull \Gamma \cong \bigoplus_{g \in \Gamma \backslash G / \Gamma} \; \M_{R(g)}(\C) \ot \cstarmax(\Gamma \cap g^{-1} \Gamma g) \quad\text{where}\;\; R(g) = [\Gamma : \Gamma \cap g^{-1} \Gamma g] \; .$$
Then, define $\He(\Gamma < G)$ as the set of functions $\xi : g \mapsto \xi_g$ on $G$ satisfying
\begin{itemize}
\item for all $g \in G$, $\xi_g$ is either $0$, either a (unitary equivalence class of a) finite dimensional unitary representation of the group $\Gamma_g := \Gamma \cap g^{-1} \Gamma g$,
\item for all $g \in G$ and $\gamma \in \Gamma$, we have $\xi_{\gamma g} \cong \xi_g$ and $\xi_{g \gamma} \cong \xi_g \circ \Ad \gamma$,
\item $\xi$ is supported on finitely many double cosets $\Gamma g \Gamma$.
\end{itemize}
The fusion product on $\He(\Gamma < G)$ is given by
\begin{align*}
(\xi * \eta)_g &= \bigoplus_{h \in \Gamma \backslash G} [\Gamma_g : \Gamma_g \cap \Gamma_h]^{-1} \; \Ind_{\Gamma_g \cap \Gamma_h}^{\Gamma_g} \bigl( \; (\xi_{gh^{-1}} \circ \Ad h)\; \otimes \; \eta_h \; \bigr) \\
& = \bigoplus_{h \in \Gamma \backslash G / \Gamma_g}  \; \Ind_{\Gamma_g \cap \Gamma_h}^{\Gamma_g} \bigl( \; (\xi_{gh^{-1}} \circ \Ad h)\; \otimes \; \eta_h \; \bigr) \; .
\end{align*}
Whenever $\xi \in \He(\Gamma < G)$, the function $g \mapsto \dim(\xi_g)$ belongs to $\Hecke(\Gamma < G)$. This yields the fusion algebra homomorphism $\He(\Gamma < G) \recht \Hecke(\Gamma < G)$.

On the other hand, when $\pi$ is a finite dimensional unitary representation of $\Gamma$, define $\xi_g = \pi$ for $g \in \Gamma$ and $\xi_g = 0$ elsewhere. Then also $\Repfin(\Gamma) \recht \He(\Gamma < G)$ is a fusion algebra homomorphism.

To prove the associativity of the fusion product on $\He(\Gamma < G)$, one has to do a painful exercise in order to obtain the symmetric expression
$$(\xi * \eta * \rho)_g = \bigoplus_{h,k \in \Gamma \backslash G} [\Gamma_g : \Gamma_g \cap \Gamma_h \cap \Gamma_k]^{-1} \; \Ind_{\Gamma_g \cap \Gamma_h \cap \Gamma_k}^{\Gamma_g} \bigl( \; (\xi_{gh^{-1}} \circ \Ad h) \; \otimes \; (\eta_{hk^{-1}} \circ \Ad k) \; \otimes \; \rho_k \; \bigr) \; .$$

\begin{theorem}\label{thm.fusion-algebra}
Let $\Gamma \actson I$ be a good action of a good group. Take $(X_0,\mu_0)$ atomic with unequal weights and set $M = \rL^\infty\bigl((X_0,\mu_0)^I\bigr) \rtimes \Gamma$. Define $G$ as the commensurator of $\Gamma$ inside $\Perm(I)$. By construction $\Gamma < G$ is a Hecke pair.

Then, the fusion algebra $\FAlg(M)$ of the II$_1$ factor $M$ is naturally isomorphic with the extended Hecke fusion algebra $\He(\Gamma < G)$.

The isomorphism $\FAlg(M) \cong \He(\Gamma < G)$ sends $_M H_M$ to $\xi$ in such a way that
$$\dim( _M H) = \sum_{g \in \Gamma \backslash G / \Gamma} [\Gamma : \Gamma \cap g \Gamma g^{-1}] \; \dim(\xi_g) \quad\text{and}\quad
\dim(H_M) = \sum_{g \in \Gamma \backslash G / \Gamma} [\Gamma : \Gamma \cap g^{-1} \Gamma g] \; \dim(\xi_g) \; .$$
\end{theorem}

\begin{example}\label{ex.hecke}
We have the following table of concrete computations of fusion algebras of II$_1$ factors. In the left column, we write good actions of good groups $\Gamma \actson I$ and in the right column we identify the fusion algebra $\FAlg(M)$ of the associated II$_1$ factor $M = \rL^\infty\bigl((X_0,\mu_0)^I \bigr) \rtimes \Gamma$ with the extended Hecke fusion algebras of a number of natural Hecke pairs.
As before, we systematically take an atomic base $(X_0,\mu_0)$ with unequal weights.
\begin{center}
\setlength{\extrarowheight}{2.5ex}
\begin{tabular}{l|l|c}
& \makebox[9.5cm][c]{$\Gamma \actson I$} & $\FAlg(M)$ \\ \hline
1. & $(\SL(n,\Z) \ltimes \Q^n) \actson \Q^n$ & $\He(\SL(n,\Z) < \GL(n,\Q))$ \\
2. & \parbox[t]{9.5cm}{$\Lambda < \PSL(n,\Q)$ a proper subgroup with the relative ICC property and take $(\Lambda \times \PSL(n,\Q)) \actson \PSL(n,\Q)$ by left-right action. Assume that $n \geq 3$.} & \parbox[t]{5cm}{\makebox[5cm][c]{$\He\bigl( \Lambda < \Comm_{G}(\Lambda) \bigr)$}\vspace{1ex}\\ where $G = \frac{\Z}{2\Z} \ltimes \PGL(n,\Q)$ \\ and $\frac{\Z}{2\Z}$ acts by $A \mapsto (A\transp)^{-1}$.} \\
3. & \parbox[t]{9.5cm}{Let $\Z \subset R \subset \Q$ be a subring strictly between $\Z$ and $\Q$. \\ Set $\Lambda = \SL(2,\Q) \ltimes \Q^2$. \\ Define $\Lambda_0 < \Lambda$ consisting of the elements $\bigl(\bigl(\begin{smallmatrix} q & 0 \\ 0 & q^{-1} \end{smallmatrix}\bigr) , \bigl(\begin{smallmatrix} x \\ y \end{smallmatrix}\bigr)\bigr)$ for $q \in R^*, x \in R, y \in \Q$. \\ Equip $\Lambda_0,\Lambda$ with the $2$-cocycle $\Omega_\al$, $\al \neq 0$ as in Cor.\ \ref{cor.example}.\\ Finally, let $(\Lambda_0 \times \Lambda) \actson \Lambda$ by left-right action.} & $\He\bigl((R^* \ltimes R) < (\Q^* \ltimes \Q)\bigr)$
\end{tabular}
\end{center}
In the final example, we define $M$ as a cocycle crossed product, see Cor.\ \ref{cor.example}. Note that a subring $R$ of $\Q$ is of the form $R = \Z[\cP^{-1}]$, where $\cP$ is a set of prime numbers.
\end{example}

\begin{remark}
For the following heuristic reason, it is interesting to have concrete examples of II$_1$ factors with fusion algebra $\He\bigl((R^* \ltimes R) < (\Q^* \ltimes \Q)\bigr)$. In general, the complexified fusion algebra $\FAlg_\C(M)$ of an arbitrary II$_1$ factor is equipped with the so-called \emph{modular automorphism group $(\si_t)_{t \in \R}$~:} whenever $_M H_M$ is an irreducible finite index $M$-$M$-bimodule, define
$$\si_t( _M H_M ) = \Bigl(\frac{\dim( _M H)}{\dim(H_M)} \Bigr)^{-it} \; _M H_M$$
and extend $\si_t$ uniquely to an automorphism of the complex $^*$-algebra $\FAlg_\C(M)$. Having examples where this modular automorphism group is non-trivial and entirely computed, provides the following link with quantum statistical mechanics, initiated by Bost and Connes in \cite{Bost-Connes}.

Under the isomorphism $\FAlg_\C(M) \cong \HeC(\Gamma < G)$ of Theorem \ref{thm.fusion-algebra}, the modular automorphism group $(\si_t)$ corresponds to the natural modular automorphism group of $\HeC(\Gamma < G)$ given by
$$\bigl(\si_t(\xi))_g = \Bigl(\frac{[\Gamma : \Gamma \cap g \Gamma g^{-1}]}{[\Gamma : \Gamma \cap g^{-1} \Gamma g]}\Bigr)^{-it} \, \xi_g \; .$$
The same formula defines the modular automorphism group on the usual complexified Hecke algebra $\Hecke_\C(\Gamma < G)$.
In the case of the Hecke pair $(1 \ltimes \Z) < (\Q^* \ltimes \Q)$, Bost and Connes classify in \cite{Bost-Connes} the KMS$_\beta$-states for the reduced C$^*$-algebra completion of $\Hecke_\C(\Gamma < G)$ equipped with the time evolution given by the modular automorphism group of $\Hecke_\C(\Gamma < G)$. It is now a natural problem to study KMS$_\beta$-states for the Hecke pair $(R^* \ltimes R) < (\Q^* \ltimes \Q)$, or even for the fusion algebra $\FAlg(M)$ provided by \ref{ex.hecke}.3.
\end{remark}

\subsection{Computations of the outer automorphism group of certain II$_1$ factors}

Since we were able to describe all finite index bimodules for the II$_1$ factors $M = \rL^\infty\bigl((X_0,\mu_0)^I\bigr) \rtimes \Gamma$, it is of course possible to describe all automorphisms as well. For the convenience of the reader, we state the result explicitly.

\begin{corollary} \label{cor.out}
Let $\Gamma \actson I$ be a good action of a good group. Set $M = \rL^\infty\bigl((X_0,\mu_0)^I\bigr) \rtimes \Gamma$. Then, the outer automorphism group of $M$ is given by
$$\Out(M) \cong \Aut(X_0,\mu_0) \times \Bigl(\Char \Gamma \rtimes \frac{G}{\Gamma} \Bigr) \quad\text{where $G$ equals the normalizer of $\Gamma$ inside $\Perm(I)$.}$$
The action $\om \cdot g$ of $g \in \frac{G}{\Gamma}$ on $\om \in \Char \Gamma$ is given by $\om \cdot g = \om \circ \Ad g$.
\end{corollary}

We illustrate the previous corollary by the following explicit computations.

\begin{examples} \label{ex.out}
\begin{enumerate}
\item Whenever $n \geq 3$ is odd and $(X_0,\mu_0)$ is an atomic probability space with unequal weights, the action $\PSL(n,\Z) \actson \rP(\Q^n)$ yields a II$_1$ factor $M$ with \emph{trivial outer automorphism group}, remembering $n$ and the base space $(X_0,\mu_0)$.
\item Let $\Lambda$ be an ICC group satisfying the minimal condition on centralizers. Assume that $\Lambda$ cannot be written as a non-trivial direct product. Consider the direct product of the action $\PSL(n,\Z) \actson \rP(\Q^n)$ (with $n \geq 3$, $n$ odd) and the left-right action of $\Lambda \times \Lambda$ on $\Lambda$. Again taking an atomic probability space with unequal weights, we obtain II$_1$ factors $M$ such that
$$\Out(M) \cong (\Char \Lambda \rtimes \Out\Lambda ) \times \Z/2\Z \; .$$
\end{enumerate}
\end{examples}

Playing with some modification of Example \ref{ex.out}.2 and using the main result of \cite{Bum-Wise}, we will prove the following result.

\begin{theorem} \label{thm.out-countable}
Every countable group arises as the outer automorphism group of a II$_1$ factor.
\end{theorem}

\subsection*{Organization of the article and the proofs}

In the next two sections, a lot of preparatory material is gathered. We first introduce in Section \ref{sec.intertwining} Popa's technique of \emph{intertwining-by-bimodules} and we prove some results that are needed throughout the main proofs of the article. Section \ref{sec.control} is still a technical preparatory section~: we prove a result that allows to control quasi-normalizers of subalgebras of crossed product von Neumann algebras. Again, these results are used several times in the main proofs of the article.

The core of the proof of Theorem \ref{thm.main} is given in Sections \ref{sec.partA} and \ref{sec.partB}. Use the notations of Theorem \ref{thm.main}. If $_N H_M$ is a finite index bimodule, it will be shown in Section \ref{sec.partA} that $H$ contains an $\rL^\infty(Y)$-$\rL^\infty(X)$-subbimodule $K$ satisfying $\dim(K_{\rL^\infty(X)}) < \infty$. This result is then combined in Section \ref{sec.partB} with Popa's cocycle superrigidity theorem (see \cite{P0}), to describe $_N H_M$ in terms of a commensurability of the actions $\Lambda \actson Y$ and $\Gamma \actson X$, as well as a finite dimensional projective representation of a finite index subgroup of $\Lambda$.

At the end of Section~\ref{sec.partB}, we call \emph{elementary bimodules} the ones that can be described in terms of a commensurability of the actions and a finite dimensional representation. Theorem \ref{thm.main} can then be rephrased as saying that every finite index $N$-$M$-bimodule is elementary. We determine the fusion rules between such elementary bimodules.

In the final Section \ref{sec.proofs}, we compile all the work of Sections \ref{sec.intertwining}~--~\ref{sec.partB} into proofs for the results announced above.

\section{Intertwining by bimodules} \label{sec.intertwining}

In \cite{P1,PBetti}, Sorin Popa has introduced a very powerful technique to obtain the unitary conjugacy of two von Neumann subalgebras of a tracial von Neumann algebra $(M,\tau)$. We make intensively use of this technique. In this section, we recall Popa's definition and prove several results that are needed later.

\begin{definition} \label{def.embeds}
Let $A,B \subset (M,\tau)$ be possibly non-unital embeddings. We say that
\begin{itemize}
\item $A \embed{M} B$ if $1_A \rL^2(M) 1_B$ contains a non-zero $A$-$B$-subbimodule $K$ satisfying $\dim(K_B) < \infty$.
\item $A \fembed{M} B$ if $Ap \embed{M} B$ for every non-zero projection $p \in 1_A M 1_A \cap A'$.
\end{itemize}
\end{definition}

The relevance of Definition \ref{def.embeds} lies in the following theorem due to Sorin Popa. Proofs can be found in Theorem 2.1 of \cite{P1} or Appendix~C of \cite{VBour}. In the list of equivalent conditions in Theorem \ref{thm.embed} below, condition~3 is, in a sense, the most useful, since it provides a powerful method to give proofs by contradiction.

\begin{theorem}[Popa, Thm.\ 2.1 in \cite{P1}] \label{thm.embed}
Let $A,B \subset (M,\tau)$ be possibly non-unital embeddings. Then, the following are equivalent.
\begin{enumerate}
\item $A \embed{M} B$.
\item There exists a, possibly non-unital, $^*$-homomorphism $\psi : A \recht B^n$ and a non-zero partial isometry $v \in \M_{1,n}(\C) \ot 1_A M 1_B$ satisfying $a v = v\psi(a)$ for all $a \in A$.
\item There does not exist a sequence of unitaries $(u_n)$ in $A$ satisfying $\|E_B(x^* u_n y)\|_2 \recht 0$ for all $x,y \in 1_A M 1_B$.
\end{enumerate}
\end{theorem}

\begin{remark}\label{rem.embed-a-lot}
We will use several times the following seemingly stronger version of Theorem \ref{thm.embed}. Suppose that $I$ is a countable set and that $A,B_i \subset (M,\tau)$ are possibly non-unital embeddings for all $i \in I$. If for all $i \in I$, we have $A \notembed{M} B_i$, there exists a sequence of unitaries $(u_n)$ in $A$ such that for all $i \in I$ and all $x,y \in 1_A M 1_{B_i}$, we have $\|E_{B_i}(x^* u_n y)\|_2 \recht 0$.

Indeed, let $(x_k)_{k \in \N}$ be a $\|\,\cdot\,\|_2$-dense sequence in the unit ball of $M$ and $I = \{i_k \mid k \in \N\}$. Fix $n \in \N$. View $A$ and $C_n := B_0 \oplus \cdots \oplus B_n$ as embedded in $M^n$. By our assumption and the second characterization in Theorem \ref{thm.embed}, $A \notembed{M} C_n$. So, we can take a unitary $u_n \in A$ such that $\|E_{B_k}(1_{B_k} x_i^* u_n x_j 1_{B_k})\|_2 < 1/n$ for all $0 \leq i,j,k \leq n$. The sequence $(u_n)$ satisfies the required properties.
\end{remark}

We leave the proof of the following elementary lemma as an exercise.

\begin{lemma} \label{lem.trivial}
Let $A,B \subset (M,\tau)$ be, possibly non-unital, embeddings. Let $q_0 \in A$, $q_1 \in 1_A M 1_A \cap A'$, $p_0 \in B$ and $p_1 \in 1_B M 1_B \cap B'$ be non-zero projections.
\begin{itemize}
\item If $q_0 A q_0 \embed{M} B$ or if $q_1 A \embed{M} B$, then $A \embed{M} B$.
\item If $A \embed{M} p_0 B p_0$ or if $A \embed{M} p_1 B$, then $A \embed{M} B$.
\item If $A \embed{M} B$ and if $D \subset A$ is a \emph{unital} von Neumann subalgebra, then $D \embed{M} B$.
\end{itemize}
\end{lemma}

\begin{lemma} \label{lem.commutant}
Let $A,B \subset M$ be, possibly non-unital, embeddings. If $A \embed{M} B$, then $1_B M 1_B \cap B' \embed{M} 1_A M 1_A \cap A'$.
\end{lemma}
\begin{proof} Let $A \embed{M} B$.
Take a projection $p \in B^n$, a non-zero partial isometry $v \in 1_A (\M_{1,n}(\C) \ot M)p$ and a unital $^*$-homomorphism $\psi : A \recht p B^n p$ satisfying $a v = v \psi(a)$ for all $a \in A$. Set $p_1 = v^* v \in p M^n p \cap \psi(A)'$ and $D = p_1 (p M^n p \cap \psi(A)')p_1$. Since $v D v^* \subset 1_A M 1_A \cap A'$ and $v^* v = 1_D$, we have $D \embed{M} 1_A M 1_A \cap A'$. By Lemma \ref{lem.trivial}, we have
$$p M^n p \cap \psi(A)' \embed{M} 1_A M 1_A \cap A' \; .$$
But, $p$ commutes with $1 \ot (1_B M 1_B \cap B')$ and $p (1 \ot (1_B M 1_B \cap B'))$ is a unital von Neumann subalgebra of $p M^n p \cap \psi(A)'$. It follows that
$$p (1 \ot (1_B M 1_B \cap B')) \embed{M} 1_A M 1_A \cap A' \; .$$
Yet another application of Lemma \ref{lem.trivial} yields the conclusion.
\end{proof}

\begin{remark}
The relation $\embed{M}$ between von Neumann subalgebras of $(M,\tau)$ is \emph{not transitive}. This is quite obvious~: let $(M,\tau)$ be a II$_1$ factor and $p \in M$ a non-trivial projection. Then, $M \embed{M} (pMp + \C (1-p))$ and $(pMp + \C (1-p)) \embed{M} \C 1$, but clearly $M \notembed{M} \C 1$.
Nevertheless, we have the following results.
\end{remark}

\begin{lemma} \label{lem.compose-full}
Let $A,B,D \subset (M,\tau)$ be possibly non-unital embeddings. If $A \embed{M} B$ and $B \fembed{M} D$, then $A \embed{M} D$.
\end{lemma}
\begin{proof}
Take, possibly non-unital, embeddings $\psi : A \recht B^n$ and $\vphi : B \recht D^m$ together with non-zero partial isometries $v \in \M_{1,n}(\C) \ot 1_A M 1_B$ and $w \in \M_{1,m}(\C) \ot 1_B M 1_D$ satisfying $a v = v \psi(a)$ for all $a \in A$ and $b w = w \vphi(b)$ for all $b \in B$. Because $B \fembed{M} D$, we can take $ww^* \in 1_B M 1_B \cap B'$ arbitrarily close to $1_B$. In particular, we can take $w$ in such a way that $v(1 \ot w) \neq 0$. Since $a v(1 \ot w) = v(1 \ot w) (\id \ot \vphi)\psi(a)$ for all $a \in A$, we are done.
\end{proof}

\begin{remark} \label{rem.composing}
Let $P,B \subset (M,\tau)$ and $A \subset B$ be possibly non-unital inclusions.

Suppose first that our aim is to prove that $P \embed{M} A$. Although the relation $\embed{M}$ is not transitive, we can nevertheless proceed in a two-step procedure. First prove that $P \embed{M} B$. Take a unital $^*$-homomorphism $\psi : P \recht p B^n p$ and a non-zero partial isometry $v \in \M_{1,n}(\C) \ot 1_P M 1_B$ satisfying $a v = v\psi(a)$ for all $a \in P$. Moreover, we can assume that $p$ equals the support projection of $E_B(v^* v)$. In a second step, prove that $\psi(P) \embed{B} A$. This yields a possibly non-unital $^*$-homomorphism $\vphi : P \recht A^m$ and a non-zero partial isometry $w \in p(\M_{n,m}(\C) \ot B1_A)$ satisfying $\psi(a) w = w \vphi(a)$ for all $a \in P$. We have then shown that $P \embed{M} A$, since $vw \neq 0$~: if $vw$ would be $0$, also $E_B(v^*v) w = E_B(v^*v w) = 0$, implying that $w = p w = 0$, a contradiction.

Secondly, we deduce from the previous paragraph the following precise statement~: if $P \embed{M} B$ and $P \notembed{M} A$, we can take a unital $^*$-homomorphism $\psi : P \recht p B^n p$ and a non-zero partial isometry $v \in \M_{1,n}(\C) \ot 1_P M 1_B$ satisfying $a v = v\psi(a)$ for all $a \in P$ and moreover satisfying $\psi(P) \notembed{B} A$.
\end{remark}

\subsection*{Intertwining by bimodules and inclusions of essentially finite index}

If $N \subset M$ is a subfactor of a II$_1$ factor, the \emph{Jones index} $[M:N]$ is defined as $[M:N]:= \dim(\rL^2(M)_N)$. We say that the subfactor $N \subset M$ is \emph{essentially of finite index}, if there exists a sequence of projections $p_n \in N' \cap M$ such that $p_n \recht 1$ and $N p_n \subset p_n M p_n$ has finite Jones index for all $n$. In Proposition \ref{prop.essential} in the Appendix, we define and characterize essentially finite index inclusions of arbitrary tracial von Neumann algebras.

We collect in this subsection several general results about the notions $\embed{M}$, $\fembed{M}$ and (essentially) finite index inclusions.

\begin{lemma} \label{lem.embed-finite-index}
Let $N,B \subset (M,\tau)$ be possibly non-unital embeddings. Let $A \subset N$ be a unital embedding that is essentially of finite index.
\begin{itemize}
\item If $A \embed{M} B$, then $N \embed{M} B$.
\item If $B \embed{M} N$, then $B \embed{M} A$.
\end{itemize}
\end{lemma}
\begin{proof}
Let $A \embed{M} B$. Take a, possibly non-unital, $^*$-homomorphism $\psi : A \recht B^n$ and a non-zero partial isometry $v \in \M_{1,n}(\C) \ot 1_A M 1_B$ satisfying $a v = v \psi(a)$ for all $a \in A$. Note that $vv^* \in 1_A M 1_A \cap A'$ and that $1_A = 1_N$. Let $p \in N \cap A'$ denote the support projection of $E_N(vv^*)$. Take $q \leq p$, a non-zero projection in $N \cap A'$ such that $\rL^2(Nq)$ is finitely generated as a right $A$-module. So, there exists a possibly non-unital $^*$-homomorphism $\vphi : N \recht A^m$ and a non-zero partial isometry $w \in \M_{1,m}(\C) \ot Nq$ satisfying $x w = w \vphi(x)$ for all $x \in N$.

We claim that $w (1 \ot v) \neq 0$. Once the claim is proven, the equality $x w (1 \ot v) = w (1 \ot v) (\id \ot \psi)\vphi(x)$ for all $x \in N$, implies that $N \embed{M} B$. Suppose that $w(1 \ot v) = 0$. Then, $w(1 \ot E_N(vv^*)) = E_N(w (1 \ot vv^*)) = 0$, implying that $w(1 \ot p) = 0$. Since $w \in \M_{1,m}(\C) \ot Nq$, this would imply that $w = 0$, contradiction.

Since $A \subset N$ has essentially finite index, we have $N \embed{N} A$ and so Remark \ref{rem.composing} yields the second statement.
\end{proof}

The most subtle abstract result on intertwining bimodules and (essentially) finite index inclusions that we need is Theorem \ref{thm.intertwine-finite-index} below. We first introduce the following notation.

\begin{notation} \label{not.embedmodule}
Let $(N,\tau)$ and $(M,\tau)$ be tracial von Neumann algebras and $A \subset N, B \subset M$ von Neumann subalgebras. Let $_N H_M$ be an $N$-$M$-bimodule.
\begin{itemize}
\item We set $A \embed{H} B$ if $H$ contains a non-zero $A$-$B$-subbimodule $K \subset H$ with $\dim(K_B) < \infty$.
\item We set $A \fembed{H} B$ if every non-zero $A$-$M$-subbimodule $K \subset H$ satisfies $A \embed{K} B$.
\end{itemize}
Note that if $\dim(H_M) < \infty$ and if we write $_N H_M \cong \, _{\psi(N)} p(\ell^2(\N) \ot \rL^2(M))_M$ for some $^*$-homo\-mor\-phism $\psi : N \recht p M^\infty p$ and a projection $p \in M^\infty$ satisfying $(\Tr \ot \tau)(p) < \infty$, then
\begin{itemize}
\item $A \embed{H} B$ if and only if $\psi(A) \embed{M} B$,
\item $A \fembed{H} B$ if and only if $\psi(A) \fembed{M} B$.
\end{itemize}
\end{notation}

\begin{theorem} \label{thm.intertwine-finite-index}
Let $(N,\tau)$ and $(M,\tau)$ be tracial von Neumann algebras with von Neumann subalgebras $A \subset N$ and $B \subset M$. Assume that
\begin{itemize}
\item every $A$-$A$-subbimodule $K \subset \rL^2(N)$ satisfying $\dim(K_A) < \infty$ is included in $\rL^2(A)$,
\item every $B$-$B$-subbimodule $K \subset \rL^2(M)$ satisfying $\dim(K_B) < \infty$ is included in $\rL^2(B)$.
\end{itemize}
Suppose that $_N H_M$ is a finite index $N$-$M$-bimodule such that, using Notation \ref{not.embedmodule},
$$A \fembed{H} B \quad\text{and}\quad A \fembedr{H} B \; .$$
Then there exists a projection $p \in B^\infty$ with $(\Tr \ot \tau)(p) < \infty$ and a $^*$-homomorphism $\vphi : N \recht p M^\infty p$ such that
\begin{align}
\bullet \quad & _N H_M \cong \; _{\vphi(N)} p (\ell^2(\N) \ot \rL^2(M))_M \; , \label{eq.isom} \\
\bullet \quad & \vphi(A) \subset p B^\infty p \;\;\text{and this inclusion has essentially finite index.}\notag
\end{align}
Moreover, through the isomorphism \eqref{eq.isom}, $p (\ell^2(\N) \ot \rL^2(B))$ is the smallest $A$-$B$-subbimodule of $H$ that contains every $A$-$B$-subbimodule $K$ satisfying $\dim(K_B) < \infty$ or $\dim( _A K) < \infty$.
\end{theorem}

\begin{proof}
Write
\begin{equation}\label{eq.iso}
_N H_M \cong \; _{\vphi(N)} p_0(\ell^2(\N) \ot \rL^2(M))_M
\end{equation}
for some finite index inclusion $\vphi : N \recht p_0 M^\infty p_0$ and a projection $p_0 \in M^\infty$ satisfying $(\Tr \ot \tau)(p_0) < \infty$.

{\it Claim.} There exists a sequence of non-zero central projections $z_n \in \cZ(A)$ summing to $1$ and a sequence of $A$-$B$-subbimodules $K^n \subset z_n H$ satisfying
\begin{itemize}
\item $\dim((K^n)_B) < \infty$,
\item every $A$-$B$-subbimodule $K \subset z_n H$ satisfying $\dim( _A K) < \infty$ is included in $K^n$,
\item there exists a $^*$-homomorphism $\psi_n : A \recht p_n B^{k_n} p_n$ and a partial isometry $v_n \in p_0(\M_{\infty,k_n}(\C) \ovt M)$ satisfying $\vphi(a) v_n = v_n \psi_n(a)$ for all $a \in A$ and $v_n v_n^* = \vphi(z_n)$. So, writing $q_n = v_n^* v_n$, we get the isomorphism
$$ _A (z_n H)_M \cong \, _{\psi_n(A)} q_n (\M_{k_n,1}(\C) \ot \rL^2(M))_M \; .$$
Through this isomorphism, $K^n$ corresponds to $q_n (\M_{k_n,1}(\C) \ot \rL^2(B))$. Moreover, $p_n$ is the support projection of $E_B(q_n)$.
\end{itemize}

{\it Proof of the claim.} It is sufficient to take an arbitrary non-zero central projection $z \in \cZ(A)$ and to prove the existence of a non-zero subprojection $z_0 \in \cZ(A)$ and an $A$-$B$-subbimodule $K^0 \subset z_0 H$ satisfying the three conditions above. Write, with the notations of \eqref{eq.iso}, $P = \vphi(A)' \cap p_0 M^\infty p_0$. Since $A' \cap N = \cZ(A)$, it follows from Lemma \ref{lem.index-relative-commutant} that $\vphi(\cZ(A)) \subset P$ has essentially finite index. Moreover, $\vphi(\cZ(A)) \subset \cZ(P)$. Making $z$ smaller, we can then assume that the inclusion $\vphi(\cZ(A) z) \subset P \vphi(z)$ is isomorphic with $1 \ot 1 \ot \cZ(A) z \subset \M_k(\C) \ot \C^l \ot \cZ(A) z$. We retain the existence of a finite number of projections $f_1,\ldots,f_m \in P \vphi(z)$ summing to $\vphi(z)$ and satisfying $f_i P f_i = f_i \vphi(\cZ(A))$ and $f_i \vphi(a) \neq 0$ whenever $a \in \cZ(A) z$ is non-zero.

Since $A \fembed{H} B$, we inductively construct
\begin{itemize}
\item a decreasing sequence $z_1 \geq \cdots \geq z_m$ of non-zero projections in $\cZ(A) z$,
\item for every $i=1,\ldots,m$, projections $p_i \in B^{n_i}$ and $^*$-homomorphisms $\psi_i : A \recht p_i B^{n_i} p_i$,
\item partial isometries $v_i \in \M_{\infty,n_i}(\C) \ovt M$ satisfying $\vphi(a) v_i = v_i \psi_i(a)$ for all $a \in A$ and $v_i v_i^* = f_i \vphi(z_i)$.
\end{itemize}
Set $z_0 = z_m$. We can cut down every $v_i$ with $\vphi(z_0)$ and then take the direct sum of all $v_i$ and $\psi_i$ for $i=1,\ldots,m$. We have found a projection $p \in B^n$, a $^*$-homomorphism $\psi : A \recht p B^n p$ and a partial isometry $v \in \M_{\infty,n}(\C) \ovt M$ satisfying $\vphi(a) v = v \psi(a)$ for all $a \in A$ and $v v^* = \vphi(z_0)$. Set $q = v^* v \in p M^n p \cap \psi(A)'$. Since the support projection of $E_B(q)$ belongs to $pB^n p \cap \psi(A)'$, we may assume that $p$ equals the support projection of $E_B(q)$. Using $v$, we get that
$$_A (z_0 H)_M \cong \; _{\psi(A)} q(\M_{n,1}(\C) \ot \rL^2(M))_M \; .$$
Through this isomorphism, we define the $A$-$B$-subbimodule $K^0 \subset z_0 H$ as $q(\M_{n,1}(\C) \ot \rL^2(B))$. Clearly, $\dim((K^0)_B) < \infty$. Suppose now that $K \subset z_0 H$ is an $A$-$B$-subbimodule satisfying \linebreak $\dim( _A K) < \infty$. View $K \subset q(\M_{n,1}(\C) \ot \rL^2(M))$. We can densely span $K$ by the components of vectors $\xi \in \M_{m,1}(\C) \ot q(\M_{n,1}(\C) \ot \rL^2(M))$ satisfying $\xi B \subset (\M_m(\C) \ot \psi(A))\xi \subset B^{mn} \xi$. By our assumptions, it follows that all the components of $\xi$ belong to $\rL^2(B)$ and hence $K \subset K^0$, proving the claim.

Let $z_n$ and $K^n$ be as in the claim above. By symmetry, there also exists a sequence of non-zero central projections $y_n \in \cZ(B)$ summing to $1$ and a sequence of $A$-$B$-subbimodules $L^n \subset H y_n$ satisfying
\begin{itemize}
\item $\dim ( _A (L^n)) < \infty$,
\item every $A$-$B$-subbimodule $L \subset H y_n$ satisfying $\dim(L_B) < \infty$ is included in $L^n$.
\end{itemize}
It follows that $z_n L^m \subset K^n y_m$ and $K^n y_m \subset z_n L^m$. So, $z_n L^m = K^n y_m$ for all $n,m$. In particular, $\dim( _A (K^n y_m)) < \infty$ for all $n,m$.

Fix $n$ and consider again $K^n$~: we have a projection $p \in B^k$, a $^*$-homomorphism $\psi : A \recht p B^k p$ and a partial isometry $v \in \M_{\infty,k}(\C) \ovt M$ satisfying $\vphi(a) v = v \psi(a)$ for all $a \in A$ and $v v^* = \vphi(z_n)$. Moreover $q = v^* v \in p M^k p \cap \psi(A)'$ and $p$ is the support projection of $E_B(q)$. But now we know that $\dim( _A (K^n y_m)) < \infty$ for all $m$, meaning that $\psi(A) y_m \subset p B^k p y_m$ is an inclusion of finite index. Hence, $\psi(A) \subset p B^k p$ is essentially of finite index. Combining this with the fact that $q$ commutes with $\psi(A)$, we conclude that $q$ belongs to the quasi-normalizer of $pB^k p$ inside $p M^k p$, which equals $p B^k p$. So, $q \in pB^k p$. Since $p$ is the support projection of $E_B(q)$, we get $p = q$.

Summing up all the partial isometries $v$ corresponding to the central projections $z_n$, we find the partial isometry $v \in M^\infty$ satisfying $vv^* = p_0$ (with $p_0$ as in \eqref{eq.iso}) and further satisfying $p : = v^* v \in B^\infty$ with $v^* \psi(A) v \subset p B^\infty p$ and the latter being an inclusion of essentially finite index.
\end{proof}

\section{Controlling quasi-normalizers and relative commutants} \label{sec.control}

In order to deduce unitary conjugacy $u A u^* \subset B$ of von Neumann subalgebras $A,B \subset (M,\tau)$, out of the weaker property $A \embed{M} B$, the main problem is to control the projection $v^* v$ where $v$ is given by \ref{thm.embed}.2. This projection $v^* v$ belongs to the \emph{relative commutant} of $\psi(A)$ inside $\psi(1) M^n \psi(1)$, but we have no a priori knowledge about the position of $\psi(A)$ inside $B^n$. In Section~3 of \cite{P1}, Popa proved a crucial result giving control on such relative commutants by using mixing properties. The main observation is contained in Lemma \ref{lem.easy-new} below. Since the exact form of the lemma as we need it here, is not available in the literature, we give a complete proof for the convenience of the reader.

In Lemma \ref{lem.control}, we then show how to use in a concrete setting, the basic principle provided by Lemma \ref{lem.easy-new}.

\begin{lemma}\label{lem.easy-new}
Let $B \subset (M,\tau)$ and $H \subset \rL^2(M)$ a $B$-$B$-subbimodule. Denote by $e_B$ the orthogonal projection of $\rL^2(M)$ onto $\rL^2(B)$. Assume that $u_n$ is a sequence of unitaries in $B$ such that
$$\|e_B (xu_n \xi)\|_2 \recht 0 \quad\text{for all}\;\; x \in M, \xi \in H^\perp \; .$$
Then, every $B$-$B$-subbimodule $K$ of $\rL^2(M)$ satisfying $\dim(K_B) < \infty$, is contained in $H$. In particular, the quasi-normalizer $\QN_M(B)\dpr$ is contained in $H \cap H^*$.
\end{lemma}
\begin{proof}
Let $\xi \in \M_{1,k}(\C) \ot \rL^2(M)$ and $\psi : B \recht B^k$ a possibly non-unital $^*$-homomorphism satisfying $b \xi = \xi \psi(b)$ for all $b \in B$. It is sufficient to prove that all entries of $\xi$ belong to $H$. Denote by $p_H$ the orthogonal projection of $\rL^2(M)$ onto $H$ and continue writing $e_B$ and $p_H$ for their respective componentwise extensions to matrices over $\rL^2(M)$.

We define $\eta = \xi - p_H(\xi)$ and we have to prove that $\eta = 0$.
Since $p_H$ commutes with the left and with the right action of $B$, we have $b \eta = \eta \psi(b)$ for all $b \in B$. Since all entries of $\eta$ belong to $H^\perp$, we get for all $x \in M$ that
$$\|e_B(x \eta)\|_2 = \|e_B(x \eta) \psi(u_n)\|_2 = \|e_B(x u_n \eta)\|_2 \recht 0 \; .$$
So, $e_B(x \eta) = 0$ for all $x \in M$, implying that $\eta = 0$.
\end{proof}

\begin{lemma} \label{lem.control}
Let $\Gamma \actson I$ be an action of a group $\Gamma$ on a set $I$ and consider $\Gamma \actson (X,\mu) = (X_0,\mu_0)^I$. Denote $M = \rL^\infty(X) \rtimes \Gamma$. Let $I_1 \subset I$ and set $\Norm I_1 = \{g \in \Gamma \mid g I_1 = I_1\}$.
Then, the following holds.
\begin{enumerate}
\item If $B \subset p \cL(\Stab I_1)^m p$ and if for all $i \in I \setminus I_1$,
$$B \notembed{\cL(\Stab I_1)} \; \cL(\Stab(I_1 \cup \{i\})) \; ,$$
then $\QN_{pM^m p}(B)\dpr \subset p (\rL^\infty(X_0^{I_1}) \rtimes \Norm I_1)^m p$.
\item If $B \subset p (\rL^\infty(X_0^{I_1}) \rtimes \Norm I_1)^m p$ and if for all $i \in I \setminus I_1$,
$$B \;\; \notembed{\rL^\infty(X_0^{I_1}) \rtimes \Norm I_1} \;\; \rL^\infty(X_0^{I_1}) \rtimes (\Norm(I_1) \cap \Stab i) \; ,$$
then, denoting $P = \rL^\infty(X) \rtimes \Norm I_1$, we have $\QN_{p P^m p}(B)\dpr \subset p (\rL^\infty(X_0^{I_1}) \rtimes \Norm I_1)^m p$.

Moreover, if $I \setminus I_1$ is infinite, the inclusion
$$\QN_{p M^m p}(B)\dpr \subset p M^m p$$
cannot have finite index.
\item If $B \subset p (\rL^\infty(X) \rtimes \Stab I_1)^m p$ and if for all $i \in I \setminus I_1$,
$$B \notembed{\rL^\infty(X) \rtimes \Stab I_1} \; \rL^\infty(X) \rtimes \Stab(I_1 \cup \{i\}) \; ,$$
then $\QN_{pM^m p}(B)\dpr \subset p (\rL^\infty(X) \rtimes \Norm I_1)^m p$.
\end{enumerate}
\end{lemma}

Throughout the proof of the lemma, we regard $A \subset A^m$ so that $a \in A$ also denotes $1 \ot a \in \M_n(\C) \ot A$. Similarly, for a conditional expectation $E_B : A \recht B$, the notation $E_B$ also denotes the amplified conditional expectation $A^m \recht B^m$.

\begin{proof}[Proof of 1]
Set $D = \rL^\infty(X_0^{I_1}) \rtimes \Norm I_1$.
The assumptions and Remark \ref{rem.embed-a-lot} yield a sequence of unitaries $u_n \in B$ such that for all $i \in I \setminus I_1$ and $g,h \in \Stab(I_1)$,
$$\|E_{\cL(\Stab(I_1 \cup \{i\}))}(u_g u_n u_h) \|_2 \recht 0 \; .$$
Set $\cJ = \{g \in \Gamma \mid g I_1 \subset I_1 \}$ and define $H$ as the closed linear span of the subspaces $L^\infty(X) u_g$, $g \in \cJ$, of $\rL^2(M)$. Observe that $H$ is an $\cL(\Stab I_1)$-$\cL(\Stab I_1)$-subbimodule of $\rL^2(M)$. We claim that
$$\|E_{\cL(\Stab I_1)}(au_g u_n b u_h)\|_2 \recht 0$$
for all $a, b \in \rL^\infty(X), g \in \Gamma, h \in \Gamma \setminus \cJ$. Indeed, because $h \in \Gamma \setminus \cJ$, take $i \in h I_1 \setminus I_1$. Write $u_n = \sum_{k \in \Stab I_1} u_n(k) u_k$ where the $u_n(k)$ are scalar matrices. Hence,
$$\|E_{\cL(\Stab I_1)}(a u_g u_n b u_h)\|_2^2 = \sum_{k \in \Stab I_1 \cap g^{-1} (\Stab I_1) h^{-1}} |\tau(a \si_{gk}(b))|^2 \, \|u_n(k)\|_2^2 \; .$$
If at the right-hand side we are not summing over the empty set, take $k_0 \in \Stab I_1$ such that
$$\Stab I_1 \cap g^{-1} (\Stab I_1) h^{-1} = k_0^{-1} (\Stab I_1 \cap h (\Stab I_1) h^{-1}) \subset k_0^{-1} \Stab(I_1 \cup \{i\}) \; .$$
It follows that
$$\|E_{\cL(\Stab I_1)}(a u_g u_n b u_h)\|_2 \leq \|a\|_2 \; \|b\|_2 \; \|E_{\cL(\Stab(I_1 \cup \{i\}))}(u_{k_0} u_n)\|_2 \recht 0 \; .$$
So, the claim is proven. Observing that $H \cap H^* = \rL^2(L^\infty(X) \rtimes \Norm I_1)$, Lemma \ref{lem.easy-new} implies that $\QN_{pM^m p}(B)\dpr \subset p (L^\infty(X) \rtimes \Norm I_1)^m p$.

To conclude the proof of 1, it remains to show that the quasi-normalizer of $B$ inside $p(L^\infty(X) \rtimes \Norm I_1)^m p$ is contained in $p (L^\infty(X_0^{I_1}) \rtimes \Norm I_1)^m p$. Because $B \subset p \cL(\Gamma)^m p$ and using Lemma \ref{lem.easy-new}, it is sufficient to prove that $\|E_{\cL(\Gamma)}(a u_n b) \|_2 \recht 0$ whenever $a \in \rL^\infty(X)$ and $b \in \rL^\infty(X) \ominus \rL^\infty(X_0^{I_1})$. We may assume that $a \in \rL^\infty(X_0^J)$ for some finite set $J \subset I$ and that $b = b_1 b_2$ with $b_1 \in \rL^\infty(X_0^i) \ominus \C 1$ and $b_2 \in \rL^\infty(X_0^{I\setminus \{i\}})$ for some $i \in I \setminus I_1$. Observe that
$$\|E_{\cL(\Gamma)}(a u_n b) \|_2^2 = \sum_{k \in \Stab I_1} |\tau(a \si_k(b))|^2 \; \|u_n(k)\|_2^2 \; .$$
Whenever $k i \not\in J$, we have $\tau(a \si_k(b)) = 0$. We can take a finite number of elements $k_1,\ldots,k_r$ \linebreak in $\Stab I_1$ such that $k \in \Stab I_1$ and $k i \in J$ implies that $k \in \bigcup_{s=1}^r k_s^{-1} \Stab(I_1 \cup \{i\})$. It follows that
$$\|E_{\cL(\Gamma)}(a u_n b) \|_2^2 \leq \|a\|_2^2 \, \|b\|_2^2 \; \sum_{s=1}^r \|E_{\cL(\Stab(I_1 \cup \{i\}))}(u_{k_s} u_n)\|_2^2 \recht 0 \; .$$
This ends the proof of point~1.

{\it Proof of 2.}\ We still denote $D = \rL^\infty(X_0^{I_1}) \rtimes \Norm I_1$ and $P = \rL^\infty(X) \rtimes \Norm I_1$. The assumptions and Remark \ref{rem.embed-a-lot} yield a sequence of unitaries $u_n$ in $B$ satisfying
$$\|E_{\rL^\infty(X_0^{I_1}) \rtimes (\Norm I_1 \cap \Stab i)}(x u_n y)\|_2 \recht 0 \quad\text{for all}\;\; x,y \in D \; .$$
Because of Lemma \ref{lem.easy-new}, it is sufficient to prove that $\|E_D(x u_n y)\|_2 \recht 0$ for all $x \in P, y \in P \ominus D$. We may assume that $x \in \rL^\infty(X), y \in \rL^\infty(X) \ominus \rL^\infty(X_0^{I_1})$ and specify even more to $x \in \rL^\infty(X_0^J)$ and $y = a b$ with $a \in \rL^\infty(X_0^{i}) \ominus \C1$ and $b \in \rL^\infty(X_0^{I \setminus \{i\}})$ for some $i \in I \setminus I_1$ and some finite subset $J \subset I$. Set $u_n = \sum_{k \in \Norm I_1} u_n(k) u_k$, where $u_n(k) \in \rL^\infty(X_0^{I_1})^m$. Observe that
$$\|E_D(x u_n y)\|_2^2 = \sum_{k \in \Norm I_1} \|E_{\rL^\infty(X_0^{I_1})}(x u_n(k) \si_k(y))\|_2^2 \; .$$
For all $k \in \Norm I_1$, one has $k i \not\in I_1$. So,
$$E_{\rL^\infty(X_0^{I_1})}(x u_n(k) \si_k(y)) = 0 \quad\text{whenever}\;\; k \in \Norm I_1 \;\;\text{and}\;\; k i \not\in J \; .$$
We then take $k_1, \ldots, k_r \in \Norm I_1$ such that $$k \in \Norm I_1 \;\;\text{and}\;\; k i \in J \quad\text{implies that}\quad k \in \bigcup_{s=1}^r k_s^{-1} (\Norm I_1 \cap \Stab i) \; .$$ We finally conclude that
$$\|E_D(x u_n y)\|_2^2 \leq \|x\|^2 \; \|y\|^2 \; \sum_{s = 1}^r \|E_{\rL^\infty(X_0^{I_1}) \rtimes (\Norm I_1 \cap \Stab i)}(u_{k_s} u_n)\|_2^2 \recht 0 \; .$$
The last statement of point~2 can be proven as follows. Since $$E_{pP^m p}(\QN_{pM^m p}(B)\dpr) = \QN_{pN^m p}(B)\dpr \subset p D^m p \; ,$$ we conclude that for all $a \in \rL^\infty(X_0^{I \setminus I_1})$ and all $x \in \QN_{pM^m p}(B)\dpr$,
$$\tau(a x) = \tau(a E_P(x)) = \tau(a) \, \tau(E_P(x)) = \tau(a) \, \tau(x) \; .$$
If $I \setminus I_1$ is infinite, we can take a unitary $u \in \rL^\infty(X_0^{I \setminus I_1})$ satisfying $\tau(u^s) = \delta_{s,0}$ for all $s \in \Z$. It follows that the subspaces $u^s \QN_{pM^m p}(B)\dpr$ of $M^m p$ are orthogonal. So $\rL^2(M^m p)$ has dimension $\infty$ as a right $\QN_{pM^m p}(B)\dpr$-module.

{\it Proof of 3.}\ The assumptions and Remark \ref{rem.embed-a-lot} yield a sequence of unitaries $u_n \in B$ satisfying
$$\|E_{\rL^\infty(X) \rtimes \Stab(I_1 \cup \{i\})}(x u_n y)\|_2 \recht 0 \quad\text{for all}\;\; x,y \in \rL^\infty(X) \rtimes \Stab I_1 \; .$$
Proceeding as in the first part of the proof of point~1, it is sufficient to prove that $$\|E_{\rL^\infty(X) \rtimes \Stab I_1}(u_g u_n u_h)\|_2 \recht 0$$ whenever $g,h \in \Gamma$ and $h I_1 \not\subset I_1$. Let $i \in h I_1 \setminus I_1$. We write $u_n = \sum_{k \in \Stab I_1} u_n(k) u_k$ with $u_n(k) \in \rL^\infty(X)^m$. Observe that
$$\|E_{\rL^\infty(X) \rtimes \Stab I_1}(u_g u_n u_h)\|_2^2 = \sum_{k \in \Stab I_1 \cap g^{-1} (\Stab I_1) h^{-1}} \|u_n(k)\|_2^2 \; .$$
Again, if at the right-hand side we are not summing over the empty set, take $k_0 \in \Stab I_1$ such that
$$\Stab I_1 \cap g^{-1} (\Stab I_1) h^{-1} = k_0^{-1} (\Stab I_1 \cap h (\Stab I_1) h^{-1}) \subset k_0^{-1} \Stab(I_1 \cup \{i\}) \; .$$
It follows that
$$\|E_{\rL^\infty(X) \rtimes \Stab I_1}(u_g u_n u_h)\|_2 \leq \|E_{\rL^\infty(X) \rtimes \Stab(I_1 \cup \{i\})}(u_{k_0}u_n)\|_2 \recht 0 \; ,$$
concluding the proof of point~3.
\end{proof}

\begin{lemma} \label{lem.notembedCartan}
Let $\Gamma \actson I$ satisfy conditions (C2) and (C3) in Definition \ref{def.conditions}. Set $M = \rL^\infty(X_0^I) \rtimes \Gamma$. Let $\Lambda \actson (Y,\eta)$ be any probability measure preserving action and set $N = \rL^\infty(Y) \rtimes \Lambda$.

Suppose that $_N H_M$ is a finite index bimodule. If $B \subset \rL^\infty(Y)$ is diffuse, $H$ does not contain a non-zero $B$-$\cL(\Gamma)$-subbimodule $K$ satisfying $\dim (K_{\cL(\Gamma)}) < \infty$.
\end{lemma}
\begin{proof}
Assume that we do have a $B$-$\cL(\Gamma)$-subbimodule $K$ satisfying $\dim (K_{\cL(\Gamma)}) < \infty$. First of all, take a finite index inclusion $\eta : N \recht q M^m q$ such that $_H H_M \cong \, _{\eta(N)} q(\M_{m,1}(\C) \ot \rL^2(M))_M$. The presence of the subbimodule $K$, combined with condition (C2) and the fact that the diffuse algebra $B$ cannot embed into a finite dimensional algebra, yield the following data~: a finite subset $I_0 \subset I$, a $^*$-homomorphism $\psi : B \recht p \cL(\Stab I_0)^m p$ and a non-zero partial isometry $v \in q M^m p$ satisfying $\eta(b) v = v \psi(b)$ for all $b \in B$ and such that, with $I_1 = \Fix(\Stab I_0)$ (and hence $\Stab I_0 = \Stab I_1$),
$$\psi(B) \notembed{\cL(\Stab(I_1))} \cL(\Stab(I_1 \cup \{i\})) \quad\text{whenever}\;\; i \in I \setminus I_1 \; .$$
Note that we do not exclude $I_0 = I_1 = \emptyset$ and $\Stab I_0 = \Gamma$. Remark \ref{rem.embed-a-lot} allows us to take a sequence of unitaries $u_n$ in $B$ such that
$$\|E_{\cL(\Stab(I_1 \cup \{i\}))}(x \psi(u_n) y)\|_2 \recht 0 \quad\text{for all}\;\; x,y \in \cL(\Stab I_1), i \in I \setminus I_1 \; .$$

Denote $B_1 = q M^m q \cap \eta(B)' \supset \eta(\rL^\infty(Y))$. Set $q_1 = vv^* \in B_1$ and $p_1 = v^* v$. By point 1 of Lemma \ref{lem.control}, we get that
$$v^* B_1 v \subset p_1 (\rL^\infty(X_0^{I_1}) \rtimes \Norm I_1)^m p_1 \; .$$
One also checks that
\begin{equation}\label{eq.onechecks}
\|E_{\rL^\infty(X_0^{I_1}) \rtimes (\Norm I_1 \cap \Stab i)}(x \psi(u_n) y)\|_2 \recht 0
\end{equation}
for all $x,y \in \rL^\infty(X_0^{I_1}) \rtimes \Norm I_1$ and all $i \in I \setminus I_1$.

Take partial isometries $w_1,\ldots, w_r \in B_1$ such that $w_s^* w_s \leq q_1$ for all $s$ and such that $\sum_{s=1}^r w_s w_s^*$ is a central projection $q_2 \in \cZ(B_1)$. Define
$$\vtil \in \M_{m,rm}(\C) \ot M \quad\text{as}\quad \vtil = \sum_{s=1}^r e_s \otimes w_s v \quad\text{and}\quad p_2 = \sum_{s=1}^r e_{ss} \otimes v^* w_s^* w_s v = \vtil^* \vtil \; .$$
Define
$$\psitil : \rL^\infty(Y) \recht p_2 (\rL^\infty(X_0^{I_1}) \rtimes \Norm I_1)^{rm} p_2 : \psitil(a) = \vtil^* \eta(a) \vtil \; .$$
The sequence of unitaries $\psitil(u_n)$ still satisfies \eqref{eq.onechecks}, so that we can apply point 2 of Lemma \ref{lem.control} to the subalgebra
$$\psitil(\rL^\infty(Y)) \subset p_2 (\rL^\infty(X_0^{I_1}) \rtimes \Norm I_1)^{rm} p_2 \; .$$
It follows that the quasi-normalizer of $\psitil(\rL^\infty(Y))$ inside $p_2 M^{rm} p_2$ does not have finite index. But this quasi-normalizer contains $\vtil^* \eta(N) \vtil$ and so as well the von Neumann algebra generated by $\vtil^* \eta(N) \vtil$. This leads to a contradiction since $\eta(N) \subset q M^m q$ has finite index.
\end{proof}

\section{For good actions of good groups, finite index bimodules \\ automatically preserve the Cartan subalgebras} \label{sec.partA}

The proof of the following theorem occupies this whole section. It consists of several steps, with Step~\ref{step-4} below as the final one.

\begin{theorem} \label{thm.preserve-Cartan}
Let $\Gamma \actson I$ and $\Lambda \actson J$ be good actions of good groups (Def.\ \ref{def.goodgood}). Consider $\Gamma \actson (X,\mu) = (X_0,\mu_0)^I$ and $\Lambda \actson (Y,\eta) = (Y_0,\eta_0)^J$. Set $M = \rL^\infty(X) \rtimes \Gamma$ and $N = \rL^\infty(Y) \rtimes \Lambda$.

Every finite index $N$-$M$-bimodule $_N H_M$ preserves the Cartan subalgebras, in the sense that there exists an $\rL^\infty(Y)$-$\rL^\infty(X)$-subbimodule $K \subset H$ satisfying $\dim(K_{\rL^\infty(X)}) < \infty$.
\end{theorem}

Fix the groups $\Gamma,\Lambda$ and their actions as in the formulation of Theorem \ref{thm.preserve-Cartan}. Let $_N H_M$ be a finite index bimodule.

Take a finite index inclusion $\psi : N \recht pM^n p$ such that $_N H_M \cong \, _{\psi(N)} p(\M_{n,1}(\C) \ot \rL^2(M))_M$. Write $P = \cL(\Lambda)$. Take an infinite, almost normal subgroup $G < \Lambda$ with the relative property~(T) and set $Q = \cL(G)$. A combination of Remark \ref{rem.weak-mixing-relative-T} and point 1 of Lemma \ref{lem.control} implies that $N \cap Q' \subset P$.

The proof of Theorem \ref{thm.preserve-Cartan} is organized as follows.
\begin{itemize}
\item Step \ref{step-1}~: we prove that $Q \embed{H} \cL(\Gamma)$. To prove this Step \ref{step-1}, we need to combine a version of Popa's theorem 4.1 in \cite{P1} (our Lemma \ref{lem.intertwine-rel-T} below) with the techniques of Section \ref{sec.control}.
\item Step \ref{step-2}~: we prove that $\cL(\Lambda) \embed{H} \cL(\Gamma)$. In fact, we obtain a better result, so that we can essentially assume that $\psi(\cL(\Lambda)) \subset \cL(\Gamma)^n$.
\item Step \ref{step-3}~: assuming now that $\psi(\cL(\Lambda)) \subset \cL(\Gamma)^n$, we prove that for $j \in J$ and $i \in I$, one has $\psi(\cL(\Stab j)) \embed{\cL(\Gamma)} \cL(\Stab i)$.
\item Step \ref{step-4}~: we finally prove that $\rL^\infty(Y) \embed{H} \rL^\infty(X)$.
\end{itemize}

Throughout this section, we denote by $(u_g)_{g \in \Gamma}$ the canonical unitaries in $\rL^\infty(X) \rtimes \Gamma$ and by $(\nu_s)_{s \in \Lambda}$ the canonical unitaries in $\rL^\infty(Y) \rtimes \Lambda$.

\begin{step}[\bf Intertwining $Q$ inside $\cL(\Gamma)$] \label{step-1}
Let $p_1 \in pM^n p \cap \psi(P)'$ be a non-zero projection.
Then, $\psi(Q)p_1 \embed{M} \cL(\Gamma)$.
\end{step}
\begin{proof}[Proof of Step~\ref{step-1}]
By condition (C2) and Remark \ref{rem.composing}, we can take a finite subset $I_0 \subset I$ (which might be empty) such that writing $I_1 = \Fix(\Stab I_0)$ and $I_2 = I \setminus  I_1$, we have
\begin{equation}\label{eq.hulp1}
\psi(P)p_1 \embed{M} \rL^\infty(X) \rtimes \Stab I_0 \quad\text{and}\quad \psi(P)p_1 \notembed{M} \rL^\infty(X) \rtimes \Stab(I_0 \cup \{i\}) \;\; \text{whenever}\;\; i \in I_2 \; .
\end{equation}
Note that $I_0 = \emptyset$ yields $I_2 = I$.

Suppose that $q \in \rL^\infty(X_0^{I_1})$ is a non-zero projection such that $q \rL^\infty(X_0^{I_1})$ is diffuse and $\psi(P)p_1 \embed{M} q (\rL^\infty(X) \rtimes \Stab I_0)$. Lemma \ref{lem.commutant} implies that $q \rL^\infty(X_0^{I_1}) \embed{M} p_1(\psi(P)' \cap pM^n p)p_1$. Lemma \ref{lem.trivial} yields $q \rL^\infty(X_0^{I_1}) \embed{M} \psi(P)' \cap pM^n p$. By Lemma \ref{lem.index-relative-commutant}, the inclusion $\psi(N \cap P') \subset pM^n p \cap \psi(P)'$ has finite index and we know that $N \cap P' = \cZ(P)$. But then Lemma \ref{lem.embed-finite-index} gives $q \rL^\infty(X_0^{I_1}) \embed{M} \psi(P)$, which is a contradiction with Lemma \ref{lem.notembedCartan}.

So, $I_1$ is finite. Moreover, in the case $I_1 \neq \emptyset$, we denote by $q \in \rL^\infty(X_0^{I_1})$ the projection on the atomic part of $\rL^\infty(X_0^{I_1})$ and conclude that
$$\psi(P)p_1 \embed{M} q (\rL^\infty(X) \rtimes \Stab I_0) \; .$$
The right hand side contains $q (\rL^\infty(X_0^{I_2}) \rtimes \Stab I_0)$ as a subalgebra of essentially finite index. By Lemmas \ref{lem.trivial} and \ref{lem.embed-finite-index}, we conclude that
\begin{equation}\label{eq.noghulp}
\psi(P)p_1 \embed{M} \rL^\infty(X_0^{I_2}) \rtimes \Stab I_0 \; .
\end{equation}

Write $X_1 = X_0^{I_2}$, $\Gamma_1 = \Stab I_0$ and $M_1 = \rL^\infty(X_1) \rtimes \Gamma_1$. Combining \eqref{eq.noghulp}, \eqref{eq.hulp1} and Remark \ref{rem.composing}, we can take a projection $q \in M_1^n$, a non-zero partial isometry $v \in \M_{1,n}(\C) \ot p_1 M$ and a unital $^*$-homomorphism $\eta : P \recht q M_1^n q$ satisfying $\psi(a) v = v \eta(a)$ for all $a \in P$ and satisfying
$$\eta(P) \notembed{M_1} \rL^\infty(X_0^{I_2}) \rtimes \Stab_{\Gamma_1} (i) \quad\text{whenever}\;\; i \in I_2 \; .$$
We may assume that the support of $E_{M_1}(v^* v)$ equals $q$.
Lemma \ref{lem.intertwine-rel-T} below implies that $\eta(Q) \embed{M_1} \cL(\Gamma_1)$. Again using Remark \ref{rem.composing}, we conclude that $\psi(Q)p_1 \embed{M} \cL(\Gamma_1) \subset \cL(\Gamma)$.
\end{proof}

We needed above the following result, due to Sorin Popa. Although the proof is very close to the proof of Theorem 4.1 in \cite{P1}, our Lemma \ref{lem.intertwine-rel-T} is not a direct consequence of Popa's theorem, so that we present a self-contained proof for the convenience of the reader.

The formulation of the lemma makes use of the relative property (T) for an inclusion of tracial von Neumann algebras. This notion has been introduced by Popa in 4.2 of \cite{PBetti} (see also B.2 in \cite{VBour}). For our purposes, it is in fact sufficient to know that for an inclusion of groups $\Lambda < \Gamma$, the inclusion $\cL(\Lambda ) \subset \cL(\Gamma)$ has the relative property (T), if and only if the group pair $\Lambda < \Gamma$ has the relative property (T).

\begin{lemma} \label{lem.intertwine-rel-T}
Let $\Gamma \actson I$ be a group acting on a set. Set $(X,\mu) = (X_0,\mu_0)^I$ and $M = \rL^\infty(X) \rtimes \Gamma$. Suppose that $Q \subset P \subset p M^n p$ satisfies the following properties.
\begin{itemize}
\item $Q \subset P$ is quasi-regular and has the relative property (T).
\item For all $i \in I$, we have $P \notembed{M} \rL^\infty(X) \rtimes \Stab i$.
\end{itemize}
Then, $Q \embed{M} \cL(\Gamma)$.
\end{lemma}

In the proof of the lemma, we make use of the following terminology~: if $A, B \subset M$ are possibly non-unital embeddings of von Neumann algebras, we say that an element $a \in 1_A M 1_B$ is \emph{$A$-$B$-finite} if there exists $x_1,\ldots,x_n,y_1,\ldots,y_m \in 1_A M 1_B$ satisfying
$$A a \subset \sum_{k=1}^n x_k B \quad\text{and}\quad a B \subset \sum_{k=1}^m A y_k \; .$$

\begin{proof}
The crucial ingredient is \emph{Popa's $s$-malleability} satisfied by generalized Bernoulli actions with diffuse base space. Define $Y_0 = X_0 \times [0,1]$ and $(Y,\mu) = (Y_0,\mu_0)^I$. Set $A = \rL^\infty(Y)$ and $B = A \ovt A = \rL^\infty(Y \times Y)$, both equipped with the generalized Bernoulli action. The action $\Gamma \actson A$ is $s$-malleable, meaning that the von Neumann algebra $B$ admits trace preserving automorphisms $(\al_t)_{t \in \R}$ and $\be$ satisfying the following conditions (see 1.6.1 in \cite{P1} and Section~3 in \cite{VBour}).
\begin{itemize}
\item $(\al_t)$ is a continuous one-parameter group of automorphisms and $\beta$ is a period 2 automorphism.
\item $\be \al_t = \al_{-t} \be$ for all $t \in \R$.
\item $\al_1(a \ot 1) = 1 \ot a$ and $\be(a \ot 1) = a \ot 1$ for all $a \in A$.
\item $\al_t$ and $\be$ commute with the action $\Gamma \actson B$.
\end{itemize}
We define $N = A \rtimes \Gamma$ and $\Ntil = B \rtimes \Gamma$. Extend $\al_t$ and $\be$ to trace preserving automorphisms of $\Ntil$ acting identically on $\cL(\Gamma)$. We view $M \subset N \subset \Ntil$ through the identification
$$N = (A \ovt 1) \rtimes \Gamma \subset (A \ovt A) \rtimes \Gamma \; .$$
We also define $N_1 := \al_1(N) = (1 \ovt A) \rtimes \Gamma$. Note that an argument similar to Lemma \ref{lem.control} yields $p \Ntil^n p \cap P' \subset p N^n p$.

We claim the existence of a non-zero, $Q$-$\al_1(Q)$-finite element $a \in p \Ntil^n \al_1(p)$. To prove this claim, define for every $t \in \R$, the $P$-$P$-bimodule on the Hilbert space $H_t = p (\M_n(\C) \ot \rL^2(\Ntil))\al_t(p)$ by the formulas
$$x \cdot \xi = x \xi \quad\text{and}\quad \xi \cdot x = \xi \al_t(x) \quad\text{for all}\;\; x \in P \; .$$
For $t = 0$, the vector $p$ is $P$-central. The relative property~(T) of $Q \subset P$ yields a $t = 2^{-k} > 0$ such that $H_t$ admits a non-zero $Q$-central vector. Taking its polar decomposition, we get a non-zero element $a \in p \Ntil^n \al_t(p)$ satisfying
$x a = a \al_t(x)$ for all $x \in Q$. In particular, $a$ is $Q$-$\al_t(Q)$-finite. In order to arrive at the claim above, it remains to show the following~: if for some $t > 0$ there exists a non-zero, $Q$-$\al_t(Q)$-finite element $a \in p \Ntil^n \al_t(p)$, then the same is true for $2t$. So, start with $t$ and $a$. Clearly, $\al_t(\be(a^*))$ is $Q$-$\al_t(Q)$-finite as well, while $\al_t(a)$ is $\al_t(Q)$-$\al_{2t}(Q)$-finite. As a consequence, $\al_t(\be(a)^* b a)$ is $Q$-$\al_{2t}(Q)$-finite for all $b \in \QN_P(Q)$. Suppose that $\be(a)^* b a = 0$ for all $b \in \QN_P(Q)$. Then the same holds true for all $b \in P$, since $Q \subset P$ is quasi-regular. Denote by $q$ the supremum of all the range projections of the elements $ba$, $b \in P$. By our assumption, $q \be(a) = 0$. On the other hand, $q \in P' \cap p \Ntil^n p$. So, $q \in p N^n p$ and hence, $\be(q) = q$. But then, the equality $q \be(a) = 0$ implies $q a = 0$, a contradiction. We have shown the claim above.

Since there exists a non-zero $Q$-$\al_1(Q)$-finite element in $p \Ntil^n p$, we can take a non-zero element $a \in p(\M_{n,mn}(\C) \ot \Ntil)(1 \ot p)$ and a, possibly non-unital, $^*$-homomorphism $\psi : Q \recht Q^m$ satisfying $x a = a (\id \ot \al_1)\psi(x)$ for all $x \in Q$. Suppose now that $Q \notembed{M} \cL(\Gamma)$. Theorem \ref{thm.embed} allows us to take a sequence of unitaries $(u_n)$ in $Q$ such that $\|E_{\cL(\Gamma)}(x u_n y)\|_2 \recht 0$ for all $x,y \in M^n$. Recall that $Y_0 = X_0 \times [0,1]$. Define $$D = (1 \ovt \rL^\infty(([0,1] \times X_0 \times [0,1])^I)) \rtimes \Gamma$$ and note that $N_1 \subset D$. It follows that $\|E_D(x u_n y)\|_2 \recht 0$ for all $x,y \in \Ntil^n$, since it suffices to check this limit for $x,y$ being the product of an element in $D$ and an element in $M$.

It then follows that
$$\|E_D(a^* a) \|_2 = \|E_D(a^* a) (\id \ot \al_1)\psi(u_n)\|_2 = \|E_D(a^* u_n a)\|_2 \recht 0 \; .$$
We arrive at the contradiction $a = 0$.
\end{proof}

In the next step, we prove that in fact, the whole of $\cL(\Lambda)$ embeds into $\cL(\Gamma)$. We even prove that the embedding can be taken in such a way that $\cL(\Lambda) \subset \cL(\Gamma)$ has essentially finite index.

\begin{step}[\bf Intertwining the group algebras $\cL(\Lambda)$ and $\cL(\Gamma)$] \label{step-2}
There exists a partial isometry $v \in \M_{n,\infty}(\C) \ovt M$ satisfying
\begin{itemize}
\item $vv^* = p$ and $q:=v^* v \in \cL(\Gamma)^\infty$,
\item $v^* \psi(\cL(\Lambda)) v \subset q \cL(\Gamma)^\infty q$ and this is an inclusion of essentially finite index.
\end{itemize}
\end{step}
\begin{proof}[Proof of Step~\ref{step-2}]
Recall that we denoted $P = \cL(\Lambda)$ with its quasi-regular subalgebra $Q \subset P$.
We prove below that $\psi(P) \fembed{M} \cL(\Gamma)$. In terms of the bimodule $_N H_M$, this means that $\cL(\Lambda) \fembed{H} \cL(\Gamma)$. By symmetry, we then have $\cL(\Lambda) \fembedr{H} \cL(\Gamma)$. As in the first point of Lemma \ref{lem.control}, we get that all the conditions of Theorem \ref{thm.intertwine-finite-index} are fulfilled and we obtain the statement of Step~\ref{step-2}.

It remains to prove that $\psi(P) \fembed{M} \cL(\Gamma)$. Take a non-zero projection $p_0 \in p M^n p \cap \psi(P)'$. We shall prove that $\psi(P)p_0 \embed{M} \cL(\Gamma)$. By Step~\ref{step-1}, $\psi(Q) p_0 \embed{M} \cL(\Gamma)$. By condition (C2), we can take a, possibly empty, finite subset $I_0 \subset I$ such that $\psi(Q) p_0 \embed{M} \cL(\Stab I_0)$ and such that writing $I_1 = \Fix(\Stab I_0)$, we have
\begin{equation}\label{eq.hulp2}
\psi(Q)p_0 \notembed{M} \cL(\Stab (I_0 \cup \{i\})) \quad\text{whenever}\;\; i \in I \setminus I_1 \; .
\end{equation}
We first claim that $I_1$ is finite and that in case $I_1 \neq \emptyset$, $X_0$ is atomic. Indeed, if the claim would be false, take a projection $q \in \rL^\infty(X_0^{I_1})$ such that $q \rL^\infty(X_0^{I_1})$ is diffuse. We have $\psi(Q) p_0 \embed{M} q \cL(\Stab I_0)$ and a combination of Lemmas \ref{lem.trivial} and \ref{lem.commutant} implies that $q \rL^\infty(X_0^{I_1}) \embed{M} p M^n p \cap \psi(Q)'$. The right hand side contains $\psi(N \cap Q')$ as a finite index subalgebra and $N \cap Q' \subset P$. It follows that $q \rL^\infty(X_0^{I_1}) \embed{M} \psi(P)$. This is a contradiction with Lemma \ref{lem.notembedCartan}. So, $I_1$ is finite.

Denote by $P_1$ the quasi-normalizer of $\psi(Q)p_0$ inside $p_0 M^n p_0$. Note that $\psi(P)p_0 \subset P_1$ and so it is sufficient to prove that $P_1 \embed{M} \cL(\Gamma)$. By \eqref{eq.hulp2} and Remark \ref{rem.composing}, we take a $^*$-homomorphism $\eta : Q \recht q \cL(\Stab I_0)^n q$ and a non-zero partial isometry $v \in p_0 M^n q$ satisfying $\psi(a) v = v \eta(a)$ for all $a \in Q$ and satisfying
$$\eta(Q) \notembed{\cL(\Stab I_0)} \; \cL(\Stab (I_0 \cup \{i\})) \quad\text{whenever}\;\; i \in I \setminus I_1 \; .$$
By point 1 of Lemma \ref{lem.control}, we have
$$v^* P_1 v \subset (\rL^\infty(X_0^{I_1}) \rtimes \Norm I_1)^n \; ,$$
where $\Norm I_1 < \Gamma$ denotes the subgroup of elements that globally preserve $I_1$. We already know that $I_1$ is finite and that in the case $I_1 \neq \emptyset$, the space $(X_0,\mu_0)$ is atomic. So, $\cL(\Stab I_0) \subset \rL^\infty(X_0^{I_1}) \rtimes \Norm I_1$ has essentially finite index. By Lemma \ref{lem.embed-finite-index}, we get $P_1 \embed{M} \cL(\Stab I_0) \subset \cL(\Gamma)$.
\end{proof}

Fix $j_0 \in J$ and $i_0 \in I$. Set $\Lambda_0 = \Stab j_0$ and $\Gamma_0 = \Stab i_0$. Note that condition (C1) says that $\Lambda_0 \actson J \setminus \{j_0\}$ and $\Gamma_0 \actson I \setminus \{i_0\}$ act with infinite orbits.

Because of Step \ref{step-2}, we have $$_N H_M \cong \, _{\psi(N)} p (\ell^2(\N) \ot \rL^2(M))_M$$ where $p \in \cL(\Gamma)^\infty$ is such that $\psi(\cL(\Lambda)) \subset p \cL(\Gamma)^\infty p$ and such that this last inclusion has essentially finite index.

\begin{step}[\bf Intertwining the stabilizers $\cL(\Lambda_0)$ and $\cL(\Gamma_0)$] \label{step-3}
There exists a partial isometry $v \in \cL(\Gamma)^\infty$ satisfying
\begin{itemize}
\item $vv^* = p$ and $q := v^* v \in \cL(\Gamma_0)^\infty$,
\item $v^* \psi(\cL(\Lambda_0)) v \subset q \cL(\Gamma_0)^\infty q$ and this last inclusion has essentially finite index.
\end{itemize}
\end{step}

\begin{proof}[Proof of Step~\ref{step-3}]
We first claim that $E_{\cL(\Gamma)}(\psi(a)) = \psi(E_{\cL(\Lambda)}(a))$ for all $a \in N$. In fact, since $\psi(\cL(\Lambda)) \subset \cL(\Gamma)^\infty$, it is sufficient to take $a \in \rL^\infty(Y) \ominus \C1$ and prove that $E_{\cL(\Gamma)}(\psi(a)) = 0$. Since $\Lambda \actson (Y,\eta)$ is weakly mixing, take a sequence $s_n \in \Lambda$ such that $\si_{s_n}(a) \recht 0$ weakly. It follows that $\|E_{\cL(\Lambda)}(\nu_s b \si_{s_n}(a))\|_2 \recht 0$ for all $s \in \Lambda$ and $b \in \rL^\infty(Y)$. Hence,
\begin{equation} \label{eq.ttt}
\|E_{\cL(\Lambda)}(x \si_{s_n}(a))\|_2 \recht 0 \quad\text{for all}\;\; x \in N \; .
\end{equation}
Since
$$\|E_{\cL(\Gamma)}(\psi(a))\|_2 = \|\psi(\nu_{s_n}) E_{\cL(\Gamma)}(\psi(a)) \psi(\nu_{s_n})^*\|_2 = \|E_{\cL(\Gamma)}(\psi(\si_{s_n}(a)))\|_2$$
for all $n$, it suffices to prove that $\|E_{\cL(\Gamma)}(\psi(\si_{s_n}(a)))\|_2 \recht 0$ in order to obtain our claim. Choose $\eps > 0$. Since $\psi(\cL(\Lambda)) \subset p \cL(\Gamma)^\infty p$ is essentially of finite index, Proposition \ref{prop.essential} implies that we can take $y_1,\ldots,y_m \in p \cL(\Gamma)^\infty p$ such that
$$\|E_{\cL(\Gamma)}(x) \|_2 \leq \eps \|x\| + \sum_{k=1}^m \|y_k\| \, \|E_{\psi(\cL(\Lambda))}(y_k^* x)\|_2 \quad\text{for all} \; \; x \in p M^\infty p \; .$$
Moreover, if $x \in N$, we have
$$\|E_{\psi(\cL(\Lambda))}(y_k^* \psi(x))\|_2 = \|E_{\cL(\Lambda)}( \psi^{-1}(E_{\psi(N)}(y_k^*)) x ) \|_2 \; .$$
Using \eqref{eq.ttt}, we get that $\|E_{\cL(\Gamma)}(\psi(\si_{s_n}(a)))\|_2 \leq 2 \|a\| \eps$ for $n$ large enough. This proves the claim above.

We claim next that $\psi(\cL(\Lambda_0)) \fembed{\cL(\Gamma)} \cL(\Gamma_0)$. As in the beginning of the proof of Step~\ref{step-2}, a combination of Lemma \ref{lem.control} and Theorem \ref{thm.intertwine-finite-index} yields the conclusion of Step~\ref{step-3}.

The prove the claim, let $p_0 \in p \cL(\Gamma)^\infty p \cap \psi(\cL(\Lambda_0))'$ be a non-zero projection. We have to prove that $\psi(\cL(\Lambda_0))p_0 \embed{\cL(\Gamma)} \cL(\Gamma_0)$. Assume the contrary. Since $\Gamma \actson I$ is transitive, all the $\cL(\Stab i)$ are unitarily conjugate inside $\cL(\Gamma)$. So, $\psi(\cL(\Lambda_0))p_0 \notembed{\cL(\Gamma)} \cL(\Stab i)$ for all $i \in I$. As in the proof of point 1 of Lemma \ref{lem.control}, we get a sequence of unitaries $u_n$ in $\cL(\Lambda_0)$ such that
$$\|E_{\cL(\Gamma)}(x \psi(u_n) p_0 y)\|_2 \recht 0 \quad\text{for all}\;\; x,y \in p (M \ominus \cL(\Gamma))^\infty p \; .$$
Take an invertible element $a \in \rL^\infty(Y_0^{\{j_0\}})$ with $\tau(a) = 0$. The claim in the beginning of the proof says that $E_{\cL(\Gamma)}(\psi(a)) = 0$. So, $\|E_{\cL(\Gamma)}(\psi(a) \psi(u_n) p_0 \psi(a)^*)\|_2 \recht 0$. On the other hand, $\psi(a)$ and $\psi(u_n)$ commute, with $\psi(u_n)$ being unitary in $p \cL(\Gamma)^\infty p$. Hence, $\psi(a) p_0 \psi(a)^* = 0$. It follows that $p_0 = 0$, a contradiction.
\end{proof}

\begin{step}[\bf Intertwining the Cartan subalgebras] \label{step-4}
We have $\psi(\rL^\infty(Y)) \embed{M} \rL^\infty(X)$.
\end{step}

\begin{proof}[Proof of Step~\ref{step-4}]
We are by now in the following situation~: $_N H_M \cong \, _{\psi(N)} p (\ell^2(\N) \ot \rL^2(M))_M$ where $p \in \cL(\Gamma_0)^\infty$ is such that
\begin{itemize}
\item $\psi(\cL(\Lambda)) \subset p \cL(\Gamma)^\infty p$ and this inclusion has essentially finite index.
\item $\psi(\cL(\Lambda_0)) \subset p \cL(\Gamma_0)^\infty p$ and this inclusion has essentially finite index.
\end{itemize}
Although we do not need it in the proof, we make the following clarifying remark~: Lemma \ref{lem.control} implies that a $^*$-homomorphism $\psi$ satisfying all the conditions above, is uniquely determined up to the obvious replacement of $\psi$ by $v \psi( \cdot ) v^*$ for some partial isometry $v \in \cL(\Gamma_0)^\infty$ satisfying $v^* v = p$.

We first claim that $E_{\cL(\Gamma_0)}(\psi(x)) = \psi(E_{\cL(\Lambda_0)}(x))$ for all $x \in \cL(\Lambda)$. Let $s \in \Lambda - \Lambda_0$. It suffices to prove that $E_{\cL(\Gamma_0)}(\psi(\nu_s)) = 0$. Since $\Lambda_0 \actson J \setminus \{j_0\}$ has infinite orbits, we can take a sequence $s_n \in \Lambda_0$ such that
\begin{equation}\label{eq.abc}
\|E_{\cL(\Lambda_0)}(x \nu_{s_n s})\|_2 \recht 0 \quad\text{for all}\;\; x \in \cL(\Lambda) \; .
\end{equation}
Let $\eps > 0$. Since $\psi(\cL(\Lambda_0)) \subset p \cL(\Gamma_0)^\infty p$ is essentially of finite index, Proposition \ref{prop.essential} yields $y_1,\ldots,y_m \in \cL(\Gamma_0)$ such that
$$\|E_{\cL(\Gamma_0)}(x)\|_2 \leq \eps \|x\| + \sum_{k=1}^m \|y_k\| \, \|E_{\psi(\cL(\Lambda_0))}(y_k^* x)\|_2 \quad\text{for all}\;\; x \in p \cL(\Gamma_0)^\infty p \; .$$
It follows that
$$\|E_{\cL(\Gamma_0)}(\psi(\nu_s))\|_2 = \|E_{\cL(\Gamma_0)}(\psi(\nu_{s_n s}))\|_2 \leq \eps + \sum_{k=1}^m \|y_k\| \, \|E_{\cL(\Lambda_0)}(\psi^{-1}(E_{\psi(\cL(\Lambda))}(y_k^*)) \nu_{s_n s}) \|_2 \; .$$
Then \eqref{eq.abc} implies that $E_{\cL(\Gamma_0)}(\psi(\nu_s)) = 0$, proving our first claim.

We next claim that $\psi(\rL^\infty(Y)) \subset p (\rL^\infty(X) \rtimes \Gamma_0)^\infty p$. The proof of this second claim is identical to the proof of Lemma 6.10 in \cite{PV}. For the convenience of the reader, we repeat it here in a way adapted to our notations. When $\cF$ is a subset of the II$_1$ factor $M$, we denote
\begin{align*}
[\cF] & := \{x \in M \mid \, \exists x_n \in \lspan \cF \;\;\text{such that}\;\; \|x_n\| \;\;\text{remains bounded and}\;\; \|x-x_n\|_2 \recht 0 \} \\
[\cF]^\infty & := \{x \in M^\infty \mid \, \text{every component}\;\; x_{ij} \;\;\text{of $x$ belongs to}\;\; [\cF] \}
\end{align*}

Combining Lemmas \ref{lem.control} and \ref{lem.embed-finite-index}, the relative commutant of $\psi(\cL(\Lambda_0))$ inside $p M^\infty p$ is contained in $p(\rL^\infty(X_0^{i_0}) \rtimes \Gamma_0)^\infty p$. This implies that
\begin{equation}\label{eq.a}
\psi(\rL^\infty(Y_0^{j_0})) \subset [\rL^\infty(X_0^{i_0}) \cL(\Gamma_0)]^\infty \; .
\end{equation}
By our first claim above,
\begin{equation}\label{eq.b}
\psi(\nu_s) \subset [u_g \mid g \in \Gamma - \Gamma_0]^\infty \quad\text{for all $s \in \Lambda - \Lambda_0$}\; .
\end{equation}
Combining \eqref{eq.a} and \eqref{eq.b} and the transitivity of $\Lambda \actson J$, it follows that
\begin{equation}\label{eq.c}
\psi(\rL^\infty(Y_0^{j})) \subset [\rL^\infty(X_0^{I \setminus \{i_0\}}) \cL(\Gamma)]^\infty \quad\text{for all}\;\; j \in J \setminus \{j_0\} \; .
\end{equation}
Take $j \in J \setminus \{j_0\}$ and $a \in \rL^\infty(Y_0^{j})$. We prove now that in fact $\psi(a) \in (\rL^\infty(X) \rtimes \Gamma_0)^\infty$. Denote by $E : (\rL^\infty(X) \rtimes \Gamma)^\infty \recht (\rL^\infty(X) \rtimes \Gamma_0)^\infty$ the natural conditional expectation and set $b = \psi(a) - E(\psi(a))$. We prove that $b = 0$. To do so, take $x \in \psi(\rL^\infty(Y_0^{j_0}))$ with $\tau(x) = 0$ and $x$ invertible. Since $x$ commutes with $\psi(a)$ and $x$ belongs to $(\rL^\infty(X) \rtimes \Gamma_0)^\infty$ by \eqref{eq.a}, we also get that $xb = bx$. Further, \eqref{eq.c} says that $\psi(a) \in [\rL^\infty(X_0^{I \setminus \{i_0\}}) \cL(\Gamma)]^\infty$, implying that
\begin{equation}\label{eq.d}
b \in [\rL^\infty(X_0^{I \setminus \{i_0\}}) u_g \mid g \in \Gamma - \Gamma_0 ]^\infty \; .
\end{equation}
By \eqref{eq.a} and the choice of $x$, we know that $x \in [(\rL^\infty(X_0^{i_0}) \ominus \C1) \cL(\Gamma_0)]^\infty$. Combining this with \eqref{eq.d}, we get that
\begin{align*}
b x & \in [\rL^\infty(X_0^{I \setminus \{i_0\}}) \, \cL(\Gamma)]^\infty \; , \\
x b & \in [(\rL^\infty(X_0^{i_0}) \ominus \C1) \, \rL^\infty(X_0^{I \setminus \{i_0\}}) \, \cL(\Gamma)]^\infty \; .
\end{align*}
It follows that $bx$ and $xb$ are orthogonal. Since $bx = xb$, we conclude that $xb = 0$. But $x$ was invertible, so that $b = 0$, proving the inclusion
$$\psi(\rL^\infty(Y_0^{j})) \subset p(\rL^\infty(X) \rtimes \Gamma_0)^\infty p \quad\text{for all}\;\; j \in J \setminus \{j_0\} \; .$$
By \eqref{eq.a}, the same holds for $j = j_0$. The proof of the second claim $\psi(\rL^\infty(Y)) \subset p (\rL^\infty(X) \rtimes \Gamma_0)^\infty p$ is done.

We now end the proof of Step~\ref{step-4} and hence of Theorem \ref{thm.preserve-Cartan}. Suppose that $\psi(\rL^\infty(Y)) \notembed{M} \rL^\infty(X)$. By the second claim above, we know that $\psi(\rL^\infty(Y)) \embed{M} \rL^\infty(X) \rtimes \Stab i_0$. Condition (C2) then yields a non-empty finite subset $I_0 \subset I$ satisfying
\begin{itemize}
\item $\psi(\rL^\infty(Y)) \embed{M} \rL^\infty(X) \rtimes \Stab I_0$,
\item $\Stab I_0$ is non-trivial (because we supposed that $\psi(\rL^\infty(Y)) \notembed{M} \rL^\infty(X)$),
\item $\psi(\rL^\infty(Y)) \notembed{M} \rL^\infty(X) \rtimes \Stab (I_0 \cup \{i\})$ whenever $i \in I \setminus I_1$ and $I_1 = \Fix(\Stab I_0)$.
\end{itemize}
Point 3 of Lemma \ref{lem.control}, implies that $\psi(N) \embed{M} \rL^\infty(X) \rtimes \Norm I_1$. We have reached a contradiction with $\psi(N) \subset pM^\infty p$ having finite index, once we show that $\Norm I_1 < \Gamma$ has infinite index. Suppose that $\Norm I_1 < \Gamma$ has finite index. Then, $\Gamma = \bigcup_{k=1}^m g_k \Norm I_1$. Since $\Gamma \actson I$ is transitive, $I = \bigcup_{k=1}^m g_k I_1$. This means that $I_1 \subset I$ has finite index (in the sense of Def.\ \ref{def.conditions}). On the other hand, taking $g \in \Stab I_0$ and $g \neq e$, the inclusion $I_1 \subset \Fix g$ together with condition (C3), imply that $I_1 \subset I$ has infinite index. We have reached the desired contradiction.
\end{proof}

\section{Cartan preserving bimodules and cocycle superrigidity} \label{sec.partB}

In this section, we introduce the family of \emph{elementary finite index bimodules} between group measure space II$_1$ factors (see Notation \ref{not.elementary} and Definition \ref{def.elementary}). It is shown in Theorem \ref{thm.cartan-preserving-bimodules} that for \emph{cocycle superrigid actions} (with countable as well as compact target groups), every finite index bimodule containing a finite index bimodule between the Cartan subalgebras, must be elementary. In Subsection \ref{subsec.fusion-elementary}, we describe the fusion rules between elementary bimodules.

Also, for suitable generalized Bernoulli actions, the elementary bimodules can be described entirely in group theoretic terms. This is done in Proposition \ref{prop.commensuration-bernoulli}.

In Theorem \ref{thm.preserve-Cartan} it was shown that for suitable generalized Bernoulli action II$_1$ factors, every finite index bimodule contains a finite index bimodule between the Cartan subalgebras. So, coupling Theorem \ref{thm.preserve-Cartan} with the results of this section, we will arrive at a proof of Theorems \ref{thm.main} and \ref{thm.fusion-algebra}~: all finite index bimodules between good generalized Bernoulli II$_1$ factors are elementary and the fusion algebra $\FAlg(M)$ of such a II$_1$ factor $M$ can be described as an extended Hecke fusion algebra.

\subsection{Reduction to elementary bimodules} \label{subsec.reduction-elementary}

\begin{terminology}
If $\Gamma \overset{\si}{\actson} \rL^\infty(X)$ and $\Lambda \overset{\si}{\actson} \rL^\infty(Y)$, we say that a $^*$-isomorphism $\Delta : \rL^\infty(X) \recht \rL^\infty(Y)$ is a $\delta$-conjugation if $\delta : \Gamma \recht \Lambda$ is a group homomorphism and $\Delta(\si_g(a)) = \si_{\delta(g)}(\Delta(a))$ for all $a \in \rL^\infty(X)$ and all $g \in \Gamma$.
\end{terminology}

\begin{notation} \label{not.elementary}
Let $\Gamma \actson (X,\mu)$ and $\Lambda \actson (Y,\eta)$ be ergodic, essentially free, probability measure preserving actions. Let $\Omega \in Z^2(\Gamma,S^1)$ and $\omega \in Z^2(\Lambda,S^1)$ be scalar $2$-cocycles.

We define the following finite index bimodules.
\begin{itemize}
\item Let $\pi : \Gamma \recht \cU(n)$ be a finite dimensional projective representation with scalar $2$-cocycle $\Omega_\pi$ defined by $\pi(g)\pi(h) = \Omega_\pi(g,h) \pi(gh)$. Denote by
$$\Hrep(\pi,\Gamma)$$ the $\bigl(\rL^\infty(X) \rtimes_{\Omega_\pi \Omega} \Gamma\bigr) \; - \; \bigl(\rL^\infty(X) \rtimes_\Omega \Gamma \bigr) \; -$ bimodule defined through the $^*$-homomorphism
\begin{align*}
& \psi : \rL^\infty(X) \rtimes_{\Omega_\pi \Omega} \Gamma \recht \M_n(\C) \ot \bigl(\rL^\infty(X) \rtimes_\Omega \Gamma\bigr) : \\ & \psi(a) = 1 \ot a \;\; , \;\; \psi(u_g) = \pi(g) \ot u_g \quad\text{for all}\;\; a \in \rL^\infty(X), g \in \Gamma \; .
\end{align*}
\item Let $\delta : \Gamma \recht \Lambda$ be a group isomorphism  and $\Delta : \rL^\infty(X) \recht \rL^\infty(Y)$ a $^*$-isomorphism such that $\Delta$ is a $\delta$-conjugation and such that $\omega(\delta(g),\delta(h)) = \Omega(g,h)$ for all $g,h \in \Gamma$. Denote by
$$\Hiso(\Gamma,\Delta,\Lambda)$$
the $\bigl(\rL^\infty(X) \rtimes_\Omega \Gamma \bigr) \; - \; \bigl(\rL^\infty(Y) \rtimes_\omega \Lambda \bigr) \; -$ bimodule defined through the $^*$-isomorphism
$$\psi : \rL^\infty(X) \rtimes_\Omega \Gamma \recht \rL^\infty(Y) \rtimes_\omega \Lambda : \psi(a) = a \;\; , \;\; \psi(u_g) = u_{\delta(g)} \quad\text{for all}\;\; a \in \rL^\infty(X), g \in \Gamma \; .$$
\item Let $\Gamma_1 < \Gamma$ be a finite index subgroup. Denote by
$$\Hincl(\Gamma_1,\Gamma)$$
the $\bigl(\rL^\infty(X) \rtimes_\Omega \Gamma_1 \bigr) \; - \; \bigl(\rL^\infty(X) \rtimes_{\Omega} \Gamma \bigr) \; -$ bimodule given by the obvious embedding of the first crossed product into the second one. Define
$$\Hred(\Gamma,\Gamma_1)$$
as the contragredient of $\Hincl(\Gamma_1,\Gamma)$.
\end{itemize}
\end{notation}

We recall the notion of a $1$-cocycle for a group action. Suppose that $\Gamma \actson (X,\mu)$ is a probability measure preserving action. A \emph{$1$-cocycle} $\rho$ for the action $\Gamma \actson (X,\mu)$ with values in the Polish group $K$, is a measurable map
$$\rho : X \times \Gamma \recht K \quad\text{satisfying}\quad \rho(x,gh) = \rho(x,g) \rho(x \cdot g, h) \quad\text{almost everywhere} \; .$$
The $1$-cocycles $\rho, \omega : X \times \Gamma \recht K$ are called \emph{cohomologous} if there exists a measurable map $\vphi : X \recht K$ satisfying
$$\om(x,g) = \vphi(x) \, \rho(x,g) \, \vphi(x \cdot g)^{-1} \quad\text{almost everywhere}\; .$$
We identify the set of \emph{homomorphisms} from $\Gamma$ to $K$ with the set of $1$-cocycles $X \times \Gamma \recht K$ \emph{not depending on the space variable $X$.}

\begin{property}[Cocycle superrigidity] \label{property}
We deal with ergodic, essentially free, probability measure preserving actions $\Gamma \actson (X,\mu)$ satisfying the following cocycle superrigidity property~: if $\Gamma_1 < \Gamma$ is a finite index subgroup, $K$ a countable or a compact second countable group and if
$$\rho : X \times \Gamma_1 \recht K$$
is a $1$-cocycle, then $\rho$ is cohomologous to a homomorphism $\Gamma_1 \recht K$.
\end{property}

By Corollary 5.4 in the article \cite{P0} of Popa, we have the following~: if the group $\Gamma$ admits a almost normal subgroup $H < \Gamma$ with the relative property (T) and if $\Gamma \actson I$ is such that $H$ acts with infinite orbits, then the generalized Bernoulli actions $\Gamma \actson (X_0,\mu_0)^I$ satisfy the cocycle superrigidity property \ref{property}.

\begin{theorem}\label{thm.cartan-preserving-bimodules}
Let $\Gamma \actson (X,\mu)$ and $\Lambda \actson (Y,\eta)$ be ergodic, essentially free, probability measure preserving actions. Let $\Omega \in Z^2(\Gamma,S^1)$ and $\omega \in Z^2(\Lambda,S^1)$ be scalar $2$-cocycles. We make the following assumptions.
\begin{itemize}
\item $\Gamma \actson (X,\mu)$ and $\Lambda \actson (Y,\eta)$ satisfy the cocycle superrigidity property \ref{property}.
\item $\Gamma$, resp.\ $\Lambda$, have no finite normal subgroups and their actions on $(X,\mu)$, resp.\ $(Y,\eta)$, are weakly mixing.
\end{itemize}

Let $H$ be an irreducible $\bigl( \rL^\infty(X) \rtimes_\Omega \Gamma \bigr) \; - \; \bigl( \rL^\infty(Y) \rtimes_\omega \Lambda \bigr) \; -$ bimodule of finite index satisfying $\rL^\infty(X) \embed{H} \rL^\infty(Y)$. Then there exist
\begin{itemize}
\item finite index subgroups $\Gamma_1 < \Gamma$ and $\Lambda_1 < \Lambda$,
\item a finite dimensional projective representation $\pi : \Gamma_1 \recht \cU(n)$,
\item a $^*$-isomorphism $\Delta : \rL^\infty(X) \recht \rL^\infty(Y)$ and a group isomorphism $\delta : \Gamma_1 \recht \Lambda_1$ such that $\Delta$ is a $\delta$-conjugation,
\end{itemize}
satisfying $\Omega(g,h) = \Omega_\pi(g,h) \; \omega(\delta(g),\delta(h))$ for all $g,h \in \Gamma_1$, as well as the bimodule isomorphism
$$H \cong \Hred(\Gamma,\Gamma_1) \underset{\rL^\infty(X) \rtimes_\Omega \Gamma_1}{\ot} \Hrep(\pi,\Gamma_1) \underset{\rL^\infty(X) \rtimes_{\Omega_1} \Gamma_1}{\ot} \Hiso(\Gamma_1,\Delta,\Lambda_1) \underset{\rL^\infty(X) \rtimes_{\omega} \Lambda_1}{\ot}  \Hincl(\Lambda_1,\Lambda) \; ,$$
where $\Omega_1 = \Omega_\pi^{-1} \Omega$ on $\Gamma_1$.
\end{theorem}

\begin{proof}
Set $A = \rL^\infty(X)$ and $B = \rL^\infty(Y)$. Set $N = \rL^\infty(Y) \rtimes_\omega \Lambda$.

By Lemma \ref{lem.cartan} below, take a projection $p \in \D_k \ot B$ and an irreducible finite index inclusion $\psi : A \rtimes_\Omega \Gamma \recht p N^k p$ defining the bimodule $H$ and satisfying
$$\psi(A) \subset (\D_k \ot B)p \quad\text{and}\quad \cD := p N^k p \cap \psi(A)' \;\;\text{of type I$_n$.}$$
Note that $\psi(A) \subset \cZ(\cD)$ and that this inclusion has finite index. Moreover, $(\Ad \psi(u_g))_{g \in \Gamma}$ extends the given ergodic action $\Gamma \actson A$ to an ergodic action on $\cZ(\cD)$, giving enough \lq uniformity\rq\ to the inclusion $\psi(A) \subset \cZ(\cD)$ to obtain an action $\Gamma \actson \{1,\ldots,r\} \times X$ and a $^*$-isomorphism $\theta : \rL^\infty(\{1,\ldots,r\} \times X) \recht\cZ(\cD)$ satisfying
\begin{align*}
& (i,x) \cdot g = (\ldots, x \cdot g) \quad\text{for all $(i,x) \in \{1,\ldots,r\} \times X$ and $g \in \Gamma$,} \\ & \theta(1 \ot a) = \psi(a) \quad\text{for all}\;\; a \in \rL^\infty(X) \; , \\
& \theta(a( \cdot g)) = \psi(u_g) \theta(a) \psi(u_g)^* \quad\text{for all $a \in \rL^\infty(\{1,\ldots,r\} \times X), g \in \Gamma$}.
\end{align*}
Write the permutation group $\pS_r$ as acting on the right on $\{1,\ldots,r\}$. We get a $1$-cocycle $\rho : X \times \Gamma \recht \pS_r$ such that $(i,x) \cdot g = (i \cdot \rho(x,g) , x \cdot g)$. By cocycle superrigidity, we may assume from the beginning that $(i,x) \cdot g = (i \cdot g, x \cdot g)$ for some action $\Gamma \actson \{1,\ldots,r\}$.

Define $\Gamma_1 = \Stab 1$ for the action $\Gamma \actson \{1,\ldots,r\}$ and set $p_1 = \theta(\delta_1 \ot 1)$. Since $(\D_k \ot B)p \subset \cD$ and $B \subset N$ is maximal abelian, it follows that $\cZ(\cD) \subset (\D_k \ot B)p$. The restriction of $\Omega$ to $\Gamma_1$ is still denoted as $\Omega$. We define
$$\psi_1 : A \rtimes_{\Omega} \Gamma_1 \recht p_1 N^k p_1 : \psi_1(x) = \psi(x) p_1 \; .$$
Writing $\cD_1 := p_1 N^k p_1 \cap \psi_1(A)' = \cD p_1$, the algebra $\cD_1$ is still of type I$_n$ and we have by construction
$$\psi_1(A) = \cZ(\cD_1) \subset (\D_k \ot B)p_1 \subset \cD_1 \; .$$
Denote by $H(\psi_1)$ the $\bigl(A \rtimes_{\Omega} \Gamma_1\bigr) \; - \; N \; -$ bimodule defined by $\psi_1$, we also have by construction
$$H \cong \Hred(\Gamma,\Gamma_1) \underset{A \rtimes_{\Omega} \Gamma_1}{\ot} H(\psi_1) \; .$$
Since $\cD_1$ is of type I$_n$ with $\cZ(\cD_1) = \psi_1(A)$, take a $^*$-isomorphism $\theta : \M_n(\C) \ot A \recht \cD_1$ satisfying $\theta(1 \ot a) = \psi_1(a)$ for all $a \in A$. Note that $(\Ad \psi_1(u_g))_{g \in \Gamma_1}$ defines an action of $\Gamma_1$ on $\cD_1$, extending the given action $(\si_g)_{g \in \Gamma_1}$ of $\Gamma_1$ on $A$. So, we find for every $g \in \Gamma_1$ a unitary $U_g \in \M_n(\C) \ot A$ such that
$$\theta(U_g (\id \ot \si_g)(a) U_g^*) = \psi_1(u_g) \theta(a) \psi_1(u_g)^* \quad\text{for all}\;\; a \in \M_n(\C) \ot A, g \in \Gamma_1 \; .$$
View $U_g$ as a measurable map from $X$ to $\cU(n)$. Composing with the quotient map $\cU(n) \recht \PcU(n)$, we
define $$\rho : X \times \Gamma_1 \recht \PcU(n) : \rho(x,g) = U_g(x) \; .$$
Then, $\rho$ is a $1$-cocycle. Cocycle superrigidity for $\Gamma_1 \actson X$ implies that we may assume that $U_g = \pi(g) \ot 1$ for some projective representation $\pi : \Gamma_1 \recht \cU(n)$. Define $\Omega_1 \in Z^2(\Gamma_1,S^1)$ by the formula $\Omega_1 = \Omega_{\pi}^{-1} \Omega$ and define for $g \in \Gamma_1$ the unitary $\nu_g \in p_1 N^k p_1$ as $\nu_g = \theta(\pi(g)^* \ot 1) \psi(u_g)$. The unitaries $(\nu_g)_{g \in \Gamma_1}$ satisfy the following properties.
\begin{itemize}
\item $\nu_g \nu_h = \Omega_1(g,h) \nu_{gh}$ for all $g,h \in \Gamma_1$.
\item $\nu_g \psi_1(a) \nu_g^* = \psi_1(\si_g(a))$ for all $g \in \Gamma_1$, $a \in A$.
\item $\nu_g$ and $\psi_1(A)$ commute with $\theta(\M_n(\C) \ot 1)$ for all $g \in \Gamma_1$.
\end{itemize}
Let $\tau_1$ be the normalized trace on $p_1 N^k p_1$ and denote by $E_{\psi_1(A)}$ the trace preserving conditional expectation $p_1 N^k p_1 \recht \psi_1(A)$. Since $\Gamma_1 \actson X$ is essentially free, the formula $\nu_g \psi_1(a) = \psi_1(\si_g(a)) \nu_g$ implies that $E_{\psi_1(A)}(\nu_g) = 0$ for all $g \neq e$. So, $\tau_1(\psi_1(a) \nu_g) = 0$ for $g \neq e$. Hence, the map $a u_g \mapsto \psi_1(a) \nu_g$ extends to an embedding $A \rtimes_{\Omega_1} \Gamma_1 \recht p_1 N^k p_1$.

Let $e_{11}$ be the obvious minimal projection in $\M_n(\C)$. Then, $\theta(e_{11} \ot 1) \in \cD_1$ is an abelian projection, while $(\D_k \ot B)p_1 \subset \cD_1$ is a maximal abelian subalgebra. We can take a partial isometry $v \in \cD_1$ satisfying $v^* v = \theta(e_{11} \ot 1)$, $p_2 := v v^*$ belongs to $\D_k \ot B$ and $v \theta(e_{11} \ot A) v^* = (\D_k \ot B) p_2$.
Define
$$\psi_2 : A \rtimes_{\Omega_1} \Gamma_1 \recht p_2 N^k p_2 : \psi_2(a u_g) = v \psi_1(a) \nu_g v^* \; .$$
By construction, we have
\begin{align}
& H(\psi_1) \cong \Hrep(\pi,\Gamma_1) \underset{A \rtimes_{\Omega_1} \Gamma_1}{\ot} H(\psi_2) \; , \label{eq.tusseniso}\\
& \psi_2(A) = (\D_k \ot B) p_2 \; . \notag
\end{align}
Set $\Ytil = \frac{\Z}{k\Z} \times Y$ and $\Lambdatil = \frac{\Z}{k\Z} \times \Lambda$ acting in the obvious way on $\Ytil$. Regard $\M_k(\C) \ot N = \rL^\infty(\Ytil) \rtimes_\omega \Lambdatil$, with the subalgebra $\D_k \ot B$ corresponding to $\rL^\infty(\Ytil)$. So, we view $p_2 \in \rL^\infty(\Ytil)$.

The proof of Proposition 5.11 in \cite{P0} (using of course once more cocycle superrigidity and using the absence of finite normal subgroups in $\Gamma$), yields the following data~:
\begin{itemize}
\item a finite index, $\Gamma_1$-invariant subalgebra $A_0 \subset A$,
\item an injective group homomorphism $\delta : \Gamma_1 \recht \Lambdatil$ and a map $\pi_0 : \Gamma_1 \recht S^1$,
\item a non-zero partial isometry $v \in p_2(\rL^\infty(\Ytil) \rtimes \Lambdatil)$ with $e=vv^*$ and $q = v^* v$,
\end{itemize}
satisfying
\begin{itemize}
\item $e$ commutes with $\psi_2(A_0 \rtimes_{\Omega_1} \Gamma_1)$ and $e \psi_2(a) e = \psi_2(E_{A_0}(a)) e$ for all $a \in A$,
\item $q$ belongs to $\rL^\infty(\Ytil)$ and is $\delta(\Gamma_1)$-invariant,
\item the map $\al : A_0 \rtimes_{\Omega_1} \Gamma_1 \recht \rL^\infty(\Ytil)q \rtimes_\omega \delta(\Gamma_1) : \al(x) = v^* \psi_2(x) v$ is a $^*$-isomorphism satisfying $\al(A_0) = \rL^\infty(\Ytil)q$ and $\al(u_g) = \pi_0(g) u_{\delta(g)}$.
\end{itemize}
Since $\psi_2(A_0 \rtimes_{\Omega_1} \Gamma_1)$ has finite index in $p_2 N^k p_2$, it follows that $\rL^\infty(\Ytil)q \rtimes_\omega \delta(\Gamma_1)$ has finite index in $q( \rL^\infty(\Ytil) \rtimes_\omega \Lambdatil) q$. As a consequence, $\delta(\Gamma_1) < \Lambdatil$ has finite index. But then, $\delta(\Gamma_1) \cap (\{e \} \times \Lambda)$ has still finite index in $\{e \} \times \Lambda$. Since $q$ is $\delta(\Gamma_1)$-invariant, it follows that $q \in \rL^\infty(\frac{\Z}{k\Z})$. The isomorphism $\al$ implements a conjugacy between the actions $\Gamma_1 \actson A_0$ and $\delta(\Gamma_1) \actson \rL^\infty(\Ytil)q$. The first one is weakly mixing and the second one has $\rL^\infty(\frac{\Z}{k\Z}) q$ as a finite-dimensional invariant subalgebra. We conclude that $q$ must be a minimal projection in $\rL^\infty(\frac{\Z}{k\Z})$. But then we may assume that $\delta$ takes values in $\Lambda$ and identify $\rL^\infty(\Ytil)q \rtimes_\omega \delta(\Gamma_1)$ with $\rL^\infty(Y) \rtimes_\omega \delta(\Gamma_1)$.

Since we assumed that the restriction of $\Lambda \actson Y$ to a finite index subgroup of $\Lambda$ is cocycle superrigid, the conjugacy $\al$ implies that $\Gamma \actson A_0$ is cocycle superrigid. As in the beginning of the proof, this implies that $\Gamma \actson A$ is conjugate to the diagonal action of $\Gamma$ on $A_0$ and an action of $\Gamma$ on a finite set. Since $\Gamma \actson A$ is weakly mixing, it follows that $A_0 = A$. But then, the projection $e$ commutes with $\psi_2(A \rtimes_{\Omega_1} \Gamma_1)$, so that $e = p_2$.

We have altogether shown that the non-normalized trace of $p_2 \in N^k$ equals $1$ and that there exists a partial isometry $v \in p_2 (\M_{k,1}(\C) \ot N)$ as well as an injective group homomorphism $\delta : \Gamma_1 \recht \Lambda$ with image of finite index, and a map $\pi_0 : \Gamma_1 \recht S^1$ such that $p_2 = vv^*$, $v^* v = 1_N$ and
$$v^* \psi_2(A) v = B \quad , \quad v^* \psi_2(u_g) v = \pi_0(g) u_{\delta(g)} \; .$$
If we now replace in \eqref{eq.tusseniso} the projective representation $\pi$ by the new projective representation $g \mapsto \overline{\pi_0(g)} \pi(g)$ and if we change $\Omega_1$ accordingly, we finally arrive at the desired conclusion
$$H \cong \Hred(\Gamma,\Gamma_1) \underset{\rL^\infty(X) \rtimes_\Omega \Gamma_1}{\ot} \Hrep(\pi,\Gamma_1) \underset{\rL^\infty(X) \rtimes_{\Omega_1} \Gamma_1}{\ot} \Hiso(\Gamma_1,\Delta,\Lambda_1) \underset{\rL^\infty(X) \rtimes_{\omega} \Lambda_1}{\ot}  \Hincl(\Lambda_1,\Lambda) \; .$$
\end{proof}

The following is a bimodule version of Theorem A.1 in \cite{PBetti} (cf.\ also Lemma 7.1 in \cite{PBetti}).

\begin{lemma}\label{lem.cartan}
Let $A \subset (M,\tau)$ and $B \subset (N,\tau)$ be II$_1$ factors with Cartan subalgebras $A,B$. Suppose that $_M H_N$ is an irreducible finite index bimodule satisfying $A \embed{H} B$. Denoting by $\D_n \subset \M_n(\C)$ the subalgebra of diagonal matrices, there exists $n$, a projection $p \in D_n \ot B$ and a finite index inclusion $\psi : A \recht p N^n p$ satisfying
\begin{itemize}
\item $_M H_N \cong \; _{\psi(M)} p(\M_{n,1}(\C) \ot \rL^2(N))_N \; ,$
\item $\psi(A) \subset (\D_n \ot B)p \; ,$
\item $p N^n p \cap \psi(A)'$ is a von Neumann algebra of type I$_k$ for some $k \in \N$.
\end{itemize}
\end{lemma}
\begin{proof}
Take an irreducible finite index inclusion $\eta : M \recht q N^m q$ such that $$_M H_N \cong \; _{\eta(M)} q(\M_{m,1}(\C) \ot \rL^2(N))_N \; .$$
By Lemma \ref{lem.index-relative-commutant} and the fact that $M \cap A' = A$, we get that $\eta(A) \subset q N^m q \cap \eta(A)'$ has finite index. Hence, $q N^m q \cap \eta(A)'$ is of finite type I. Moreover, whenever $u \in \cN_M(A)$, the unitary $\eta(u)$ normalizes $q N^m q \cap \eta(A)'$ and all these normalizing unitaries together act ergodically on $q N^m q \cap \eta(A)'$ by the irreducibility $q N^m q \cap \eta(M)' = \C 1$. The existence of a trace preserving ergodic action implies that $q N^m q \cap \eta(A)'$ is of type I$_k$ for some $k \in \N$.

Observe that every, possibly non-unital, $^*$-homomorphism $A \recht \M_r(\C) \ot B$ can be intertwined into a $^*$-homomorphism $A \recht \D_r \ot B$. We know that $\eta(A) \embed{N} B$ and so, using the previous observation, we find a non-zero partial isometry $v \in q (\M_{m,r}(\C) \ot N)$ and a, possibly non-unital, $^*$-homomorphism $\theta : A \recht \D_r \ot B$ satisfying $\eta(a) v = v \theta(a)$ for all $a \in A$. Cutting down $v$ on the left by an abelian projection in $q N^m q \cap \eta(A)'$ and on the right by one of the minimal projections in $\D_r$, we may assume that $r=1$ and that $vv^*$ is an abelian projection in $q N^m q \cap \eta(A)'$. Set $p_1 = \theta(1)$, which is a non-zero projection in $B$. It follows that $v^* v$ is an abelian projection in $p_1 N p_1 \cap \theta(A)'$. Moreover, $B p_1 \subset p_1 N p_1 \cap \theta(A)'$ is a maximally abelian subalgebra. By a folklore result (use e.g.\ Section 6.4 in \cite{KadRing}), we can take a partial isometry $w \in p_1 N p_1 \cap \theta(A)'$ satisfying $ww^* = v^* v$ and $w^* (p_1 N p_1 \cap \theta(A)') w \subset B p_1$. Replacing $v$ by $vw$, we have found a non-zero partial isometry $v \in q(\M_{m,1}(\C) \ot N)$ satisfying $p_1 = v^* v \in B$, $vv^* \in q N^m q \cap \eta(A)'$ and $v^* (q N^m q \cap \eta(A)') v \subset B p_1$.

Since $qN^m q \cap \eta(A)'$ is of type I$_k$ with abelian projection $vv^*$, denote the central support of $vv^*$ by $z$ and take $k$ partial isometries $w_1,\ldots,w_k \in qN^m q \cap \eta(A)'$ satisfying $w_i^* w_i = vv^*$ and $\sum_{i=1}^k w_i w_i^* = z$. Set $Z = \cZ(qN^m q \cap \eta(A)')$ and note that $\Ad \psi(\cN_M(A))$ defines an ergodic action on $Z$. So, we can take partial isometries $u_1,\ldots,u_n$ in $qN^m q$ having initial and final support in $Z$ and satisfying
$$u_j^* u_j \leq z \; , \quad \sum_{j=1}^n u_j u_j^* = q \quad\text{and}\quad u_j^* Z u_j = u_j^* u_j Z \; .$$
Define $u \in \M_{m,nk}(\C) \ot N$ with columns given by $u_j w_i v \in \M_{m,1}(\C) \ot N$. Then $uu^* = q$ and $u^* Z u \subset \D_{nk} \ot B$. Defining $\psi(x) = u^* \eta(x)u$ and observing that $\eta(A) \subset Z$, we have reached the conclusion of the lemma.
\end{proof}

\subsection{Fusion rules between elementary bimodules} \label{subsec.fusion-elementary}

In this subsection, the Connes tensor product of bimodules is just denoted by juxtaposition. So, $H K$ means $H \underset{N}{\ot} K$. It will always be clear from the context over which von Neumann algebras the bimodules are considered.

We denote by $\Aut(X,\mu)$ the Polish group of probability space isomorphisms modulo equality almost everywhere. Since we write in this article groups as acting on the right on $X$, we also let $\Delta \in \Aut(X,\mu)$ act on the right on $x$ and write $x \cdot \Delta$. For every $\Delta \in \Aut(X,\mu)$, define $(\Delta_* f)(x) = f(x \cdot \Delta^{-1})$ and note that $\Delta_* \in \Aut(\rL^\infty(X,\mu))$. As such, the group $\Aut(X,\mu)$ is isomorphic with the group of trace preserving automorphisms of $\rL^\infty(X,\mu)$.

\begin{definition} \label{def.commensuration}
Let $\Gamma \actson (X,\mu)$ be an essentially free, probability measure preserving action.

An element $\Delta$ of $\Aut(X,\mu)$ is called a \emph{commensuration of $\Gamma \actson (X,\mu)$} if $\Delta$ belongs to the commensurator of $\Gamma$ viewed as a subgroup of $\Aut(X,\mu)$.

Whenever $\Delta$ is a commensuration of $\Gamma \actson (X,\mu)$, we define the finite index subgroups of $\Gamma$
\begin{equation}\label{eq.leftright-group}
\lef{\Delta} := \Gamma \cap \Delta \Gamma \Delta^{-1} \quad\text{and}\quad \rig{\Delta} := \Gamma \cap \Delta^{-1} \Gamma \Delta \; .
\end{equation}
Then, $\delta := \Ad \Delta^{-1} : \lef{\Delta} \recht \rig{\Delta}$ is a group isomorphism. Moreover, with the above notations, $\Delta_*$ is a $\delta$-conjugation.

More generally, a \emph{commensuration of $\Gamma \actson (X,\mu)$ and $\Lambda \actson (Y,\eta)$} is a probability space isomorphism $\Delta : (X,\mu) \recht (Y,\eta)$ such that $\Gamma \cap \Delta^{-1} \Lambda \Delta$ and $\Lambda \cap \Delta \Gamma \Delta^{-1}$ have finite index in $\Gamma$, resp.\ $\Lambda$.
\end{definition}

\begin{definition}[Elementary bimodules] \label{def.elementary}
Let $\Gamma \actson (X,\mu)$ be an essentially free, probability measure preserving, weakly mixing action. Let $\Omega \in Z^2(\Gamma,S^1)$ be a scalar $2$-cocycle. Set $M = \rL^\infty(X) \rtimes_\Omega \Gamma$. Using the notations of \ref{not.elementary}, \ref{def.commensuration} and motivated by Theorem \ref{thm.cartan-preserving-bimodules}, we introduce the following finite index $M$-$M$-bimodules.

Suppose that $\Delta$ is a commensuration of $\Gamma \actson (X,\mu)$ with $\rig{\Delta}$, $\lef{\Delta}$ defined by \eqref{eq.leftright-group}. Let $\pi : \rig{\Delta} \recht \cU(n)$ be a projective representation satisfying $\Omega_\pi \; \Omega = \Omega \circ \Ad \Delta$ on $\rig{\Delta}$. Define
$$H(\Delta,\pi) := \Hred(\Gamma,\lef{\Delta}) \; \Hiso(\lef{\Delta},\Delta_*,\rig{\Delta}) \; \Hrep(\rig{\Delta},\pi) \; \Hincl(\rig{\Delta},\Gamma) \; .$$
We call the $M$-$M$-bimodules of the form $H(\Delta,\pi)$, the \emph{elementary $M$-$M$-bimodules}.
\end{definition}

We now write down the fusion rules between the elementary $M$-$M$-bimodules. In order to do so correctly, we need to take care of the $2$-cocycles and define the induction of a projective representation.

\begin{definition} \label{def.induce}
Let $\Gamma$ be a group with subgroup $\Gamma_1 < \Gamma$. Let $\pi : \Gamma_1 \recht \cU(K)$ be a projective representation with scalar $2$-cocycle $\Omega_\pi$. Suppose that $\Omega \in Z^2(\Gamma,S^1)$ is a $2$-cocyle that extends $\Omega_\pi$. We then define the \emph{induced projective representation} $\pi_1 = \Ind_{\Gamma_1}^\Gamma \pi$ \emph{along the cocycle $\Omega$} on the Hilbert space
\begin{align*}
K_1 := \{\xi : \Gamma \recht K \mid \; & \xi(hg) = \overline{\Omega(h,g)} \pi(h) \xi(g) \;\;\text{for all}\;\; h \in \Gamma_1, g \in \Gamma \\ & \text{and}\;\; g \mapsto \|\xi(g)\| \;\; \text{belongs to}\;\; \ell^2(\Gamma_1 \backslash \Gamma) \}
\end{align*}
by the formula $(\pi_1(g) \xi)(h) = \Omega(h,g) \xi(hg)$. Note that $\Omega_{\pi_1} = \Omega$.
\end{definition}

Whenever $\Omega \in Z^2(\Gamma,S^1)$ and $g \in \Gamma$, define $$\vphi_g(h) = \Omega(ghg^{-1},g) \overline{\Omega(g,h)} \; .$$ Also write for every function $\vphi : \Gamma \recht S^1$ its coboundary $(\partial \vphi)(g,h) = \vphi(g) \vphi(h) \overline{\vphi(gh)}$. One then has
$$\Omega \circ \Ad g = (\partial \vphi_g) \Omega \quad\text{for all}\;\; g \in \Gamma \; .$$

\begin{theorem} \label{thm.fusion-elementary}
Let $\Gamma \actson (X,\mu)$ be an essentially free, probability measure preserving, weakly mixing action. Let $\Omega \in Z^2(\Gamma,S^1)$ be a scalar $2$-cocycle. Set $M = \rL^\infty(X) \rtimes_\Omega \Gamma$. Define $\cF \subset \FAlg(M)$ as the fusion subalgebra generated by the elementary $M$-$M$-bimodules in the sense of Definition \ref{def.elementary}.

Let $\Delta$ and $\Deltatil$ be commensurations of $\Gamma \actson (X,\mu)$.
\begin{enumerate}
\item If $\pi, \pitil$ are projective representations of $\rig{\Delta}$ satisfying $\Omega_{\pi} = \Omega_{\pitil} = (\Omega \circ \Ad \Delta) \; \overline{\Omega}$ on $\rig{\Delta}$, then
$$H(\Delta,\pi \oplus \pitil) \cong H(\Delta,\pi) \oplus H(\Delta,\pitil) \; .$$
\item If $\pi,\pitil$ are projective representations of $\rig{\Delta}, \rig{\Deltatil}$ satisfying $\Omega_{\pi} \; \Omega = \Omega \circ \Ad \Delta$ and similarly for $\Omega_{\pitil}$, then the following fusion rule holds.
$$H(\Delta,\pi) \; H(\Deltatil,\pitil) \cong \bigoplus_{g \, \in \; \raisebox{-0.6ex}[0ex][0ex]{\scriptsize $\rig{\Delta}$} \!\! \backslash \Gamma / \!\! \raisebox{-0.6ex}[0ex][0ex]{\scriptsize $\lef{\Deltatil}$}} H(\Delta g \Deltatil, \pi_g) \; ,$$
where $\pi_g$ is the projective representation of $\rig{\Delta g \Deltatil}$ defined as the induction
\begin{align*}
\pi_g = & \Ind_{\rig{\Delta g \Deltatil} \cap \rig{\Deltatil}}^{\rig{\Delta g \Deltatil}} \bigl( (\vphi_g \circ \Ad \Deltatil) (\pi \circ \Ad g \Deltatil)  \ot \pitil\Bigr) \\ &\text{along the $2$-cocycle}\;\; (\Omega \circ \Ad(\Delta g \Deltatil)) \; \overline{\Omega} \;\;\text{on}\;\; \rig{\Delta g \Deltatil} \; .
\end{align*}
\item Let $\pi$ be a projective representations of $\rig{\Delta}$ satisfying $\Omega_{\pi} \; \Omega = \Omega \circ \Ad \Delta$ on $\rig{\Delta}$. Then $H(\Delta,\pi)$ is an \emph{irreducible} bimodule if and only if $\pi$ is irreducible. Moreover, for irreducible $\pi,\pitil$, we have
$$H(\Delta,\pi) \cong H(\Deltatil,\pitil)$$
if and only if there exist $g,h \in \Gamma$ such that
$$\Deltatil = g \Delta h \quad\text{and}\quad \pitil \;\;\text{is unitarily equivalent with}\;\; (\vphi_g \circ \Ad(\Delta h)) \, (\pi \circ \Ad h) \, \vphi_h \; .$$
\end{enumerate}
\end{theorem}
\begin{proof}
Point 1 is obvious.

Set $A = \rL^\infty(X)$. Let $\Delta$ be a commensuration of $\Gamma \actson (X,\mu)$. Set $\Gamma_1 = \lef{\Delta}$ and $\delta = \Ad \Delta^{-1}$ the isomorphism of $\Gamma_1$ onto $\rig{\Delta}$. Let $\pi : \delta(\Gamma_1) \recht \cU(n)$ be a projective representation such that $\Omega_\pi \; \Omega = \Omega \circ \delta$. Let $k$ be the index of $\Gamma_1$ in $\Gamma$. Then,
$$H(\Delta,\pi) \cong H(\psi)$$
where $\psi : A \rtimes_\Omega \Gamma \recht \M_n(\C) \ot \M_k(\C) \ot (A \rtimes_\Omega \Gamma)$ is defined as follows. Choose coset representatives $\Gamma = \bigsqcup_{i=1}^k \Gamma_1 g_i$ and define the action $\Gamma \actson \{1,\ldots,k\}$ and the $1$-cocycle $\eta : \{1,\ldots,k\} \times \Gamma \recht \Gamma_1$ by the formula $g_i s = \eta(i,s) g_{i \cdot s}$ for all $i \in \{1,\ldots,k\}$ and $s \in \Gamma$. We then set
\begin{alignat*}{2}
\psi(a) &= \sum_{i=1}^k 1 \ot e_{ii} \ot \Delta_*(\si_{g_i}(a)) & & \qquad\text{for all}\;\; a \in A \; , \\
\psi(u_s) &= \sum_{i=1}^k \Omega(g_i,s) \overline{\Omega(\eta(i,s),g_{i \cdot s})} \; \bigl( \pi(\delta(\eta(i,s))) \ot e_{i,i \cdot s} \ot u_{\delta(\eta(i,s))} \bigr) & & \qquad\text{for all}\;\; s \in \Gamma \; .
\end{alignat*}
Based on such a concrete formula, verification of 2 is elementary, though a bit tedious.

We finally prove 3. If $\pi$ is reducible, it follows immediately from 1 that $H(\Delta,\pi)$ is reducible. So, suppose that $\pi$ is irreducible and realize $H(\Delta,\pi) \cong H(\psi)$ as above.
We first prove that the relative commutant of $\psi(A)$ inside $\M_n(\C) \ot \M_k(\C) \ot (A \rtimes_\Omega \Gamma)$ is given by $\M_n(\C) \ot \D_k \ot A$. In order to obtain this result, it suffices to take an arbitrary element $x \in \M_k(\C) \ot (A \rtimes_\Omega \Gamma)$ commuting with all the operators
$$\sum_{i=1}^k e_{ii} \ot \Delta_*(\si_{g_i}(a)) \;\; , \;\; a \in A \; ,$$
and to prove that $x \in \D_k \ot A$. Consider $x$ as a matrix $(x_{ij})$ with entries in $A \rtimes_\Omega \Gamma$ and decompose every entry as $x_{ij} = \sum_{s \in \Gamma} x_{ij}^s u_s$ with $x_{ij}^s \in A$. It follows that
$$x_{ij}^s \; \si_s(\Delta_*(\si_{g_j}(a))) = \Delta_*(\si_{g_i}(a)) \; x^s_{ij} \quad\text{for all}\;\; a \in A, s \in \Gamma, i,j \in \{1,\ldots,k\} \; .$$
So, if $x^s_{ij} \neq 0$, the automorphisms $\si_s \circ \Delta_* \circ \si_{g_j}$ and $\Delta_* \circ \si_{g_i}$ coincide on some non-negligible part of $X$. By Lemma \ref{lem.maximal-com} below, it follows that $g_i g_j^{-1} \in \Gamma_1$. This implies that $i = j$. Moreover, if $x^s_{ii} \neq 0$, it follows that $\si_s$ is the identity on a non-negligible part of $X$ and so $s=e$. Altogether it follows that $x \in \D_k \ot A$.

The unitaries $(\psi(u_s))_{s \in \Gamma}$ normalize $\M_n(\C) \ot \D_k \ot A$ and define on this last von Neumann algebra an action that we denote by $(\rho_s)$. The automorphism $\theta \in \Aut(\M_n(\C) \ot \D_k \ot A)$ given by
$$\theta(e_{ij} \ot e_k \ot a) = e_{ij} \ot e_k \ot \Delta(\si_{g_k}(a)) \; ,$$
conjugates the action $(\rho_s)$ with the diagonal action of $\Gamma \actson A$ and the action $(\gamma_s)$ of $\Gamma$ on $\M_n(\C) \ot \D_k$ given by
$$\gamma_{s^{-1}}(w \ot e_k) = \pi(\eta(k,s))^* w \pi(\eta(k,s)) \ot e_{k \cdot s} \; .$$
Since $\Gamma \actson A$ is weakly mixing, the irreducibility of $H(\psi)$ follows if we prove the ergodicity of $(\gamma_s)$. The latter follows straightforwardly from the assumed irreducibility of $\pi$.

The statement about the isomorphism between two irreducible elementary bimodules can be proven in a way that is very similar to the proof of the irreducibility of $H(\psi)$.
\end{proof}

We will use Theorem \ref{thm.fusion-elementary} to compute explicitly the fusion algebra of certain II$_1$ factors. So, we need to compute the commensurator of $\Gamma$ inside $\Aut(X,\mu)$ in certain cases. This is done for certain generalized Bernoulli actions in Proposition \ref{prop.commensuration-bernoulli} below.

We make use of the obvious embeddings
$$\pi_i : \rL^\infty(X_0,\mu_0) \recht \rL^\infty((X_0,\mu_0)^I) \quad\text{for all}\;\; i \in I \; .$$

\begin{proposition} \label{prop.commensuration-bernoulli}
Let $\Gamma \actson I$ and $\Lambda \actson J$ be actions such that $\Stab i \actson I \setminus \{i\}$ and $\Stab j \actson J \setminus \{j\}$ act with infinite orbits for all $i \in I$ and $j \in J$. Let $(X_0,\mu_0)$ and $(Y_0,\eta_0)$ be standard probability spaces and consider the generalized Bernoulli actions
$$\Gamma \actson (X,\mu) := (X_0,\mu_0)^I \quad\text{and}\quad \Lambda \actson (Y_0,\eta_0)^J \; .$$
Suppose that the bijection $\eta : I \recht J$ is a commensuration of $\Gamma \actson I$ and $\Lambda \actson J$, i.e.\ $\eta(g \cdot i) = \delta(g) \cdot \eta(i)$ for all $i \in I$ and $g \in \Gamma_1$, where $\delta$ is an isomorphism between the finite index subgroups $\Gamma_1,\Lambda_1$ of $\Gamma,\Lambda$.

Let for every orbit $\ibar \in \Gamma_1 \backslash I$, be given a trace preserving $^*$-isomorphism $\al_{\ibar} : \rL^\infty(X_0,\mu_0) \recht \rL^\infty(Y_0,\eta_0)$. Define the isomorphism $\Delta : (X,\mu) \recht (Y,\eta)$ such that
$$\Delta_* \circ \pi_i = \pi_{\eta(i)} \circ \al_{\ibar} \quad\text{for all}\;\; i \in I \; .$$
\begin{itemize}
\item $\Delta$ is a commensuration of $\Gamma \actson (X_0,\mu_0)^I$ and $\Lambda \actson (Y_0,\eta_0)^J$.
\item Every commensuration of $\Gamma \actson (X_0,\mu_0)^I$ and $\Lambda \actson (Y_0,\eta_0)^J$ arises in this way.
\end{itemize}
\end{proposition}
\begin{proof}
It is obvious that the proposed formula for $\Delta$ defines a commensuration.

Suppose conversely that $\Delta$ is a commensuration of $\Gamma \actson (X_0,\mu_0)^I$ and $\Lambda \actson (Y_0,\eta_0)^J$. So, let $\Delta(x \cdot g) = \Delta(x) \cdot \delta(g)$ for almost all $x \in X, g \in \Gamma_1$ and let $\delta : \Gamma_1 \recht \Lambda_1$ be an isomorphism between finite index subgroups $\Gamma_1,\Lambda_1$ of $\Gamma,\Lambda$.

We first claim that there exists a bijection $\eta : I \recht J$ such that $\Delta_*(\rL^\infty(X_0^i)) = \rL^\infty(Y_0^{\eta(i)})$ for all $i \in I$. Because of the symmetry between $\Delta$ and $\Delta^{-1}$, it is sufficient to prove that for every $i \in I$, there exists $j \in J$ satisfying $\rL^\infty(Y_0^{j}) \subset \Delta_*(\rL^\infty(X_0^i))$.

Choose $i \in I$. Set $\Gamma_2 = \Gamma_1 \cap \Stab i$, which is a finite index subgroup of $\Stab i$. By our assumptions, $\rL^\infty(X_0^i) = \rL^\infty(X)^{\Gamma_2}$ and so $\Delta_*(\rL^\infty(X_0^i)) = \rL^\infty(Y)^{\delta(\Gamma_2)}$. If for all $j \in J$ the group $\delta(\Gamma_2) \cap \Stab j$ would have infinite index in $\delta(\Gamma_2)$, the group $\delta(\Gamma_2)$ would act with infinite orbits on $J$ and so, the fixed point algebra $\rL^\infty(Y)^{\delta(\Gamma_2)}$ would be trivial, a contradiction. So, take $j \in J$ and a finite index subgroup $\Gamma_3 < \Gamma_2$ satisfying $\delta(\Gamma_3) = \delta(\Gamma_2) \cap \Stab j$. Then, $\Gamma_3 < \Stab i$ has finite index, implying that $\rL^\infty(X_0^i) = \rL^\infty(X)^{\Gamma_3}$ and hence, $\Delta_*(\rL^\infty(X_0^i)) = \rL^\infty(Y)^{\delta(\Gamma_3)}$. Therefore,
$$\rL^\infty(Y_0^{j}) = \rL^\infty(Y)^{\Stab j} \subset \rL^\infty(Y)^{\delta(\Gamma_3)} = \Delta_*(\rL^\infty(X_0^i)) \; .$$
This proves the claim above.

It is then clear that $\eta(g \cdot i) = \delta(g) \cdot \eta(i)$ for all $i \in I$ and $g \in \Gamma_1$. One defines for every $i \in I$, the $^*$-isomorphism $\al_i : \rL^\infty(X_0,\mu_0) \recht \rL^\infty(Y_0,\eta_0)$ such that $\Delta_* \circ \pi_i = \pi_{\eta(i)} \circ \al_i$. It is easily checked that $\al_i$ only depends on the $\Gamma_1$-orbit of $i$. So, we are done.
\end{proof}

For weakly mixing actions $\Gamma \actson (X,\mu)$, the subgroup $\lef{\Delta} = \Gamma \cap \Delta \Gamma \Delta^{-1}$ for a given commensuration $\Delta$ of $\Gamma \actson (X,\mu)$ can be characterized by a weaker condition. That is done in the following lemma that we have used in the proof of Theorem \ref{thm.fusion-elementary}.

\begin{lemma} \label{lem.maximal-com}
Let $\Gamma \actson (X,\mu)$ be an essentially free, probability measure preserving, weakly mixing action. Then the commensurator of $\Gamma$ inside $\Aut(X,\mu)$ acts essentially freely on $(X,\mu)$.
\end{lemma}
\begin{proof}
Let $\Delta$ be a commensuration of $\Gamma \actson (X,\mu)$ and suppose that $x \cdot \Delta = x$ for all $x \in \cU$ and $\cU$ non-negligible. We have to prove that $x \cdot \Delta = x$ almost everywhere.

If $g \in \Gamma \cap \Delta^{-1} \Gamma \Delta$ and $\cU \cap \cU \cdot g^{-1}$ is non-negligible, we find
$$x \cdot \Delta g \Delta^{-1} = x \cdot g \Delta^{-1} = x \cdot g$$
for all $x \in \cU \cap \cU \cdot g^{-1}$, so that essential freeness of $\Gamma \actson (X,\mu)$ implies that $\Delta$ and $g$ commute.

Let now $g \in \rig{\Delta} = \Gamma \cap \Delta^{-1} \Gamma \Delta$ be arbitrary. Since $\rig{\Delta} < \Gamma$ has finite index, the action of $\rig{\Delta}$ on $(X,\mu)$ is still weakly mixing. So, we can take $g_1 \in \rig{\Delta}$ such that both $\cU \cap (\cU \cdot g^{-1}) \cdot g_1^{-1}$ and $\cU \cap \cU \cdot g_1^{-1}$ are non-negligible. By the previous paragraph, $\Delta$ commutes with $g_1 g$ and with $g_1$. So, $\Delta$ commutes with $g$ for all $g \in \rig{\Delta}$.

But then, for all $x \in \cU$ and all $g \in \rig{\Delta}$,
$$(x \cdot g) \cdot \Delta = x \cdot (g \Delta) = x \cdot (\Delta g) = x \cdot g \; .$$
Since $\rig{\Delta}$ acts ergodically on $(X,\mu)$, it follows that $x \cdot \Delta = x$ for almost all $x \in X$.
\end{proof}

\section{Proofs of the results announced in Section \ref{sec.main-results}} \label{sec.proofs}

\begin{proof}[About Examples \ref{ex.conditions}]
The non-trivial points to verify in \ref{ex.conditions} are the following.
\begin{itemize}
\item All the linear groups/linear actions satisfy the minimal condition on centralizers/stabilizers because in a finite dimensional vector space, there cannot be an infinite strictly decreasing sequence of vector subspaces.
\item Groups defined by a (possibly infinite) presentation satisfying the $C'(1/6)$-small cancelation condition satisfy the minimal condition on centralizers because the centralizer of any non-trivial element is cyclic, see \cite{truff}. For more or less analogous reasons, word hyperbolic groups satisfy the minimal condition on centralizers, see Example 3.2.4 in \cite{duncan}.
\item Note that $\Z^n < \SL(n,\Q) \ltimes \Q^n$ and $\PSL(n,\Z) < \PSL(n,\Q)$ are almost normal subgroups with the relative property (T).
\end{itemize}
\end{proof}

In order to treat systematically the concrete computations of fusion algebras in \ref{cor.example} and \ref{ex.hecke}, we start with the following lemma computing some commensurators of subgroups.

\begin{lemma} \label{lem.compute-comm}
\begin{itemize}
\item Let $\Gamma < \GL(n,\Q)$ be a subgroup with the following property~: if $\Gamma_0 < \Gamma$ is a finite index subgroup and if $V \subset \Q^n$ is a non-zero globally $\Gamma$-invariant vector subspace of $\Q^n$, then $V = \Q^n$.

Then, the commensurator of $\Gamma \ltimes \Q^n$ inside $\Perm(\Q^n)$ equals $\Comm_{\GL(n,\Q)}(\Gamma) \ltimes \Q^n$.

\item Let $\Lambda$ be a group that cannot be written as a non-trivial direct product and that has no non-trivial finite index subgroups. Let $\Lambda_0 < \Lambda$ be a proper subgroup with the relative ICC property.

Then, the commensurator of the left-right action of $\Lambda_0 \times \Lambda$ on $\Lambda$ is given by the permutations $g \mapsto \al(g) g_0$ for some $\al \in \Comm_{\Aut(\Lambda)}(\Ad \Lambda_0)$ and $g_0 \in \Lambda$.
\end{itemize}
\end{lemma}
\begin{proof}
To prove the first item, write, maybe confusingly, $\Gamma \ltimes \Q^n = \{(v,g) \mid v \in \Q^n, g \in \Gamma\}$ acting on $\Q^n$ by $(v,g) \cdot w = v + g w$. Let $\eta$ be a permutation of $\Q^n$ in the commensurator of $\Gamma \ltimes \Q^n$.
Composing with a translation, we may assume that $\eta(0) = 0$. Since $\Q^n$ has no non-trivial finite index subgroups, we find finite index subgroups $\Gamma_0,\Gamma_1$ of $\Gamma$ and a group isomorphism $\delta : \Gamma_0 \ltimes \Q^n \recht \Gamma_1 \ltimes \Q^n$ satisfying $\eta(v + g w) = \delta(v,g) \cdot \eta(w)$ for all $g \in \Gamma_0$ and $v,w \in \Q^n$. In particular, $\delta(\Gamma_0) = \Gamma_1$ and the lemma is proven once we have shown that $\delta(\Q^n) = \Q^n$. By symmetry it suffices to show that $\delta(\Q^n) \subset \Q^n$.

Write $\delta(v,1) = (\pi(v),\rho(v))$ for all $v \in \Q^n$. Set $V = \{v \in \Q^n \mid \rho(v) = 1 \}$. We have to prove that $V = \Q^n$.
Assume that $\rho(v) \neq 1$ for some $v \in \Q^n$. Since $\delta(\Q^n)$ is a normal subgroup of $\Gamma_1 \ltimes \Q^n$, it follows that
$$((1 - \rho(v))w, 1) = (w,1) \, \delta(v,1) \, (-w,1) \, \delta(v,1)^{-1} \quad\text{belongs to}\;\; \delta(\Q^n)$$
for all $v,w \in \Q^n$. Since we assumed that $\rho(v) \neq 1$ for at least one $v \in \Q^n$, it follows that $V \neq \{0\}$. So, $V$ is a non-trivial globally $\Gamma$-invariant subgroup of $\Q^n$. We finally prove that $V$ is in fact a vector subspace of $\Q^n$. Then, our assumptions imply that $V = \Q^n$, ending the proof of the first item.

Since $((1-\rho(v))w,1) \in \delta(\Q^n)$, it follows that $((1-\rho(v))w,1)$ and $\delta(v',1)$ commute for all $v,v',w \in \Q^n$. Writing this out yields $(1-\rho(v'))(1-\rho(v)) = 0$ for all $v,v' \in \Q^n$. But then, $\gamma(v) = \rho(v) - 1$ defines an additive group homomorphism $\gamma : \Q^n \recht \M_n(\Q)$. Such a homomorphism is automatically linear and so, $V$ is a vector subspace of $\Q^n$.

To prove the second item, let $\eta \in \Perm(\Lambda)$ be in the commensurator of $\Lambda_0 \times \Lambda$. We may assume that $\eta(e) = e$. We have to prove that $\eta$ is an automorphism of $\Lambda$. By our assumptions, we find finite index subgroups $\Lambda_1,\Lambda_2 < \Lambda_0$ and an isomorphism $\delta : \Lambda_1 \times \Lambda \recht \Lambda_2 \times \Lambda$ satisfying $\eta((g,h) \cdot i) = \delta(g,h) \cdot \eta(i)$. In particular, $\delta(\diag(\Lambda_1)) = \diag(\Lambda_2)$, where $\diag(\Lambda_i)$ denotes the diagonal subgroup of $\Lambda_i \times \Lambda$.

We claim that $\delta = \al \times \al$ for some automorphism $\al \in \Aut(\Lambda)$ satisfying $\al(\Lambda_1) = \Lambda_2$. Once the claim is proven, the second item of the lemma follows immediately. To prove the claim, it suffices to prove that $\delta(\Lambda_1 \times \{e\}) = \Lambda_2 \times \{e\}$. Indeed, taking centralizers, it then follows that $\delta(\{e\} \times \Lambda) = \{e \} \times \Lambda$, yielding the automorphism $\al \in \Aut(\Lambda)$. But then, $\delta = \al \times \al$, because $\delta$ preserves the diagonal subgroups.

In order to finally prove that $\delta(\Lambda_1 \times \{e\}) = \Lambda_2 \times \{e\}$, it suffices by symmetry to prove the inclusion
$\delta(\Lambda_1 \times \{e\}) \subset \Lambda_2 \times \{e\}$. Denote by $\Gamma_1$, resp.\ $\Gamma$ the image of $\delta(\Lambda_1 \times \{e\})$, resp.\ $\delta(\{e\} \times \Lambda)$ under the projection map $\Lambda_2 \times \Lambda \recht \Lambda$. We have written $\Lambda$ as the product of two commuting subgroups $\Gamma_1$ and $\Gamma$. Since $\Lambda$ has trivial center and cannot be written as a non-trivial direct product, one of the groups $\Gamma_1$, $\Gamma$ is trivial. If $\Gamma_1$ is trivial, we are done. So, suppose that $\Gamma$ is trivial. This means that $\delta(\{e\} \times \Lambda) \subset \Lambda_1 \times \{e\}$. Again taking centralizers, we find a subgroup $\Lambda_3 < \Lambda_1$ such that $\delta(\Lambda_3 \times \{e\}) = \{e\} \times \Lambda$. Since $\delta(\{e\} \times \Lambda_3) \subset \Lambda_1 \times \{e\}$, it
follows that $\delta(\diag \Lambda_3)$ projects surjectively onto $\Lambda$. A fortiori, $\diag \Lambda_2 = \delta(\diag \Lambda_1)$ projects surjectively onto $\Lambda$. This is a contradiction, since $\Lambda_0$ is a proper subgroup of $\Lambda$.
\end{proof}

\begin{proof}[Proof of Corollary \ref{cor.example}]
As was probably noted first in \cite{vNW}, the group $\Gamma := \SL(2,\Q) \ltimes \Q^2$ does not admit non-trivial finite dimensional unitary representations. In particular, $\Gamma$ has no non-trivial finite index subgroups.
Then, Corollary \ref{cor.example} follows from Theorem \ref{thm.cartan-preserving-bimodules} and Proposition \ref{prop.commensuration-bernoulli}, once we have shown the following~: if $\eta$ is a permutation of $\Q^2$ that normalizes $\Gamma$ and satisfies $\Omega_\al \sim \Omega_\al \circ \Ad \eta$ as scalar $2$-cocycles on $\Gamma$, then $\eta$ belongs to $\Gamma$.

Lemma \ref{lem.compute-comm} says that $\eta = (g_0,v_0) \in \GL(2,\Q) \ltimes \Q^2$. A small computation yields $\Omega_\al \circ \Ad \eta \sim \Omega_{(\det g_0) \al}$. So, $\Omega_\al \sim \Omega_{(\det g_0) \al}$ as scalar $2$-cocycles on $\Gamma$ and in particular as scalar $2$-cocycles on $\Q^2$. It is well known, and even directly computable, that this implies $\al = (\det g_0) \al$. Since $\al \neq 0$, we conclude that $\det g_0 = 1$ and so $\eta \in \Gamma$.
\end{proof}

\begin{proof}[Proof of Theorem \ref{thm.fusion-algebra}]
By Theorem \ref{thm.cartan-preserving-bimodules}, $\FAlg(M)$ equals the fusion algebra of elementary $M$-$M$-bimodules in the sense of Definition \ref{def.elementary}. Theorem \ref{thm.fusion-elementary} says that the fusion algebra of elementary $M$-$M$-bimodules is exactly given as the extended Hecke fusion algebra $\He(\Gamma < G)$, where $G$ denotes the commensurator of $\Gamma$ inside $\Aut(X,\mu)$. By Proposition \ref{prop.commensuration-bernoulli} and because $(X_0,\mu_0)$ is assumed to be atomic with unequal weights, the latter is isomorphic with the commensurator of $\Gamma$ inside $\Perm(I)$.
\end{proof}

\begin{proof}[Proof of Example \ref{ex.hecke}]
In all three examples, we use the following principle~: let $\Gamma_1 < G_1$ be a Hecke pair and $\pi : G_1 \recht G$ a surjective homomorphism satisfying $\Ker \pi \subset \Gamma_1$. Set $\Gamma := \pi(\Gamma)$ and note that $\Gamma < G$ is again a Hecke pair. Assume moreover that every finite dimensional unitary representation of $\Gamma_1$ is trivial on $\Ker \pi$. Then, $\He(\Gamma_1 < G_1) \cong \He(\Gamma < G)$.

{\it Example \ref{ex.hecke}.1.}\ By Theorem \ref{thm.fusion-algebra} and Lemma \ref{lem.compute-comm}, $\FAlg(M) \cong \He\bigl((\SL(n,\Z) \ltimes \Q^n) < (\GL(n,\Q) \ltimes \Q^n)\bigr)$. We claim that the latter is isomorphic with $\He(\SL(n,\Z) < \GL(n,\Q))$. Because of the principle above, it is sufficient to show that every finite dimensional unitary representation $\pi$ of $\SL(n,\Z) \ltimes \Q^n$ is trivial on $\Q^n$. The restriction of $\pi$ to $\Q^n$ is the direct sum of group characters of $\Q^n$ belonging to a finite subset $S \subset \widehat{\Q^n}$. We have to prove that $S = \{1\}$. Since $S$ is finite, every $\om \in S$ is invariant under a finite index subgroup $\Gamma < \SL(n,\Z)$, meaning that $\om((1-g)x) = 1$ for all $g \in \Gamma$ and $x \in \Q^n$. All sums of elements of the form $(1-g)x$, for $g \in \Gamma$ and $x \in \Q^n$, form a vector subspace of $\Q^n$. If $\om \neq 1$, this vector subspace is not the whole of $\Q^n$ and we find a non-zero $y \in \Q^n$ such that $g\transp y = y$ for all $g \in \Gamma$. This contradicts the fact that $\Gamma$ has finite index in $\SL(n,\Z)$.

{\it Example \ref{ex.hecke}.2.}\ By \cite{SW}, $\Aut(\PSL(n,\Q)) = \Z/2\Z \ltimes \PGL(n,\Q)$, where $\Z/2\Z$ acts by the order $2$ automorphism $\delta(A) = (A\transp)^{-1}$. The result is then a combination of Theorem \ref{thm.fusion-algebra}, Lemma \ref{lem.compute-comm} and the above mentioned principle. On the way, one uses once more that $\PSL(n,\Q)$ has no non-trivial finite dimensional unitary representations.

{\it Example \ref{ex.hecke}.3.}\ Set $\Gamma := \Lambda_0 \times \Lambda$. By Lemma \ref{lem.compute-comm}, the inclusion $\Gamma < \Comm_{\Perm(\Lambda)}(\Gamma)$ is isomorphic with $\Gamma < G$, where
$$G = \{(\rho, \rho \circ \Ad g) \mid \rho \in \Comm_{\Aut(\Lambda)}(\Lambda_0), g \in \Lambda \} \; .$$
But, $\Aut(\Lambda) = \GL(2,\Q) \ltimes \Q^2$. Moreover, for all $\rho \in \GL(2,\Q) \ltimes \Q^2$, we have $\Omega_\al \circ \rho = \Omega_{(\det \rho) \al}$ and $\Lambda$ has no non-trivial finite dimensional unitary representations. A combination of Thms.\ \ref{thm.cartan-preserving-bimodules}, \ref{thm.fusion-elementary}, Prop.\ \ref{prop.commensuration-bernoulli} and the above principle, implies that $\FAlg(M) \cong \He(\Lambda_0 < \Comm_\Lambda(\Lambda_0))$.

A small computation shows that $\Comm_\Lambda(\Lambda_0)$ consists of the elements $\bigl(\bigl(\begin{smallmatrix} q & 0 \\ 0 & q^{-1} \end{smallmatrix}\bigr) , \bigl(\begin{smallmatrix} x \\ y \end{smallmatrix}\bigr)\bigr)$ with $q \in \Q^*$ and $x,y \in \Q$. Note that
$\bigl(\bigl(\begin{smallmatrix} 0 & 1 \\ -1 & 0 \end{smallmatrix}\bigr) , \bigl(\begin{smallmatrix} x \\ y \end{smallmatrix}\bigr)\bigr)$ is excluded from this commensurator because $R < \Q$ has infinite index as an additive subgroup. Finally, consider the quotient homomorphism
$$\Comm_\Lambda(\Lambda_0) \recht \Q^* \ltimes \Q : \bigl(\bigl(\begin{smallmatrix} q & 0 \\ 0 & q^{-1} \end{smallmatrix}\bigr) , \bigl(\begin{smallmatrix} x \\ y \end{smallmatrix}\bigr)\bigr) \mapsto (q,x) \; .$$
We shall apply the principle above to conclude that $\FAlg(M) = \He\bigl((R^* \ltimes R) < (\Q^* \ltimes \Q) \bigr)$. In order to do so, we have to show that every finite dimensional unitary representation of $\Lambda_0$ factorizes through $R^* \ltimes R$. It suffices to show that every finite dimensional unitary representation $\pi$ of $R^* \ltimes \Q$ is trivial on $\Q$. The restriction of such a $\pi$ to $\Q$ is a finite direct sum of group characters $\om \in S \subset \Qh$. It follows that the finite set $S$ is globally invariant under $R^*$. But $R^*$ acts freely on $\Qh - \{1\}$, implying that $S = \{1\}$. So, $\pi$ is trivial on $\Q$.
\end{proof}

\begin{proof}[Proof of Corollary \ref{cor.out}]
Theorem \ref{thm.cartan-preserving-bimodules} implies that, up to inner automorphisms, every automorphism of $M$ is given by a character of $\Gamma$ and an element in the normalizer of $\Gamma$ inside $\Aut(X,\mu)$. This normalizer is determined in Proposition \ref{prop.commensuration-bernoulli}, yielding the result in Corollary \ref{cor.out}.
\end{proof}

\begin{proof}[About Examples \ref{ex.out}]
In the first example, set $\Gamma := \PSL(n,\Z)$. Note that for $n$ odd, $\PSL(n,\Z) = \SL(n,\Z)$. By Example 2.6.1 in \cite{Borel}, $\Out(\Gamma)$ has two elements, the non-trivial one being given by $\al : A \mapsto (A\transp)^{-1}$. Since there is no permutation $\eta$ of $\rP(\Q^n)$ satisfying $\eta(A v) = \al(A) \eta(v)$ for all $A \in \Gamma$, $v \in \rP(\Q^n)$, we conclude that the normalizer of $\Gamma$ inside $\Perm(\rP(\Q^n))$ equals $\Gamma$. Because $\Gamma$ has no non-trivial characters, the conclusion follows from Corollary \ref{cor.out}.

As in the second item of Lemma \ref{lem.compute-comm} and using the example in the previous paragraph, the normalizer of $\Gamma := \PSL(n,\Z) \times \Lambda \times \Lambda$ inside $\Perm(\rP(\Q^n) \times \Lambda)$ is generated by $\Gamma$, $\{\id\} \times \Aut(\Lambda)$ and the permutation $(v,g) \mapsto (v, g^{-1})$. The conclusion follows again from Corollary \ref{cor.out}.
\end{proof}

\begin{proof}[Proof of Theorem \ref{thm.out-countable}]
Let $Q$ be a given countable group. Bumagin and Wise construct in \cite{Bum-Wise} a countable group $\Lambda$ with the following properties.
\begin{itemize}
\item $\Out(\Lambda) \cong Q$.
\item $\Lambda$ is a subgroup of a $C'(1/6)$-small cancelation group and $\Lambda$ is not virtually cyclic. In particular, $\Lambda$ is ICC and satisfies the minimal condition on centralizers (see \ref{ex.conditions}). Also, the centralizer in $\Lambda$ of any non-cyclic subgroup, is trivial.
\item Slightly modifying the construction of \cite{Bum-Wise}, by adding relations that make it impossible to have non-trivial abelian quotients, we may also assume that $\Char \Lambda = \{1\}$.
\end{itemize}
Let a finite group $H$ act by permutations of a finite set $J$, in such a way that $\Char H = \{1\}$ and that the normalizer of $H$ inside $\Perm(J)$ equals $H$. A concrete example is provided in Lemma 7.8 in \cite{PV} as the linear action of $H = \GL(3,F_2)$ on $J = F_2^3 \setminus \{0\}$, where $F_2$ denotes the field with two elements. We then consider the action $\Gamma \actson I$, defined as the direct product of the actions  $$\PSL(3,\Z) \actson \rP(\Q^3) \quad\text{and}\quad ((\Lambda^J \rtimes H) \times \Lambda) \; \actson \; \Lambda^J \; .$$
In this expression, $\Lambda^J \rtimes H$ acts on the left on $\Lambda^J$, while $\Lambda$ acts diagonally on the right.

In order to show that $\Gamma \actson I$ is a good action of a good group, we have to prove that the action $((\Lambda^J \rtimes H) \times \Lambda) \actson \Lambda^J$ satisfies conditions (C1), (C2) and (C3) in Definition \ref{def.conditions}. Conditions (C1) and (C2) are immediately given by the ICC property and the minimal condition on centralizers for $\Lambda$. Condition (C3) is checked as follows~: if $\si \in H$ is different from the identity, $\Fix((g,\si),h)$ has infinite index because the diagonal subgroup $\diag \Lambda < \Lambda \times \Lambda$ has infinite index. When $\si = e$, but $(g,h) \neq e$, we again have $\Fix((g,e),h)$ of infinite index, because $\Lambda$ is an ICC group.

Define $M$ as the generalized Bernoulli II$_1$ factor associated with $\Gamma \actson I$ and an atomic base space with unequal weights. Corollary \ref{cor.out} yields $\Out(M) \cong G/\Gamma$, where $G$ denotes the normalizer of $\Gamma$ inside $\Perm(I)$. In order to determine this normalizer, first make the following easy observation. Let $(\Gamma_i)_{i \in \cF}$ and $(\Lambda_j)_{j \in \cK}$ be finite families of ICC groups with the property that they do not contain a non-trivial direct product as a finite index subgroup. Then, for every injective homomorphism $\theta : \bigoplus_{i \in \cF} \Gamma_i \recht \bigoplus_{j \in \cK} \Lambda_j$ with finite index image, there exists a bijection $\si : \cF \recht \cK$ satisfying $\theta(\Gamma_i) \subset \Lambda_{\si(i)}$. With a reasoning similar to Lemma \ref{lem.compute-comm}, we deduce that the normalizer $G$ of $\Gamma$ inside $\Perm(I)$, is generated by $\Gamma$ and $\Aut(\Lambda)$. Here, $\Aut(\Lambda)$ is viewed as acting diagonally on $\Lambda^J$. So, $\Out(M) \cong G/\Gamma \cong \Out(\Lambda) \cong Q$ and we are done.
\end{proof}

\section*{Appendix A~: Inclusions of (essentially) finite index and approximations of conditional expectations}\renewcommand{\thesection}{A}\setcounter{definition}{0}

Let $A \subset (M,\tau)$ be an inclusion of tracial von Neumann algebras. \emph{Jones' basic construction} is defined as the von Neumann algebra $\langle M,e_A \rangle$ acting on $\rL^2(M,\tau)$ generated by $M$ and the \emph{Jones projection} $e_A$ defined by $e_A x = E_A(x)$ for all $x \in M$. Here we view $A \subset M \subset \rL^2(M,\tau)$ and $E_A$ is the unique $\tau$-preserving conditional expectation. We make the following well known observations.
\begin{itemize}
\item The von Neumann algebra $\langle M,e_A \rangle$ equals the commutant of the right action of $A$ on $\rL^2(M,\tau)$.
\item The projection $e_A$ commutes with $A$ inside $\langle M,e_A \rangle$ and further we have $e_A x e_A = E_A(x) e_A$ for all $x \in M$. It follows that the linear span of the elements $x e_A y$ with $x,y \in M$ is a dense $^*$-subalgebra of $\langle M,e_A \rangle$.
\end{itemize}
The basic construction $\langle M,e_A \rangle$ comes equipped with a semifinite faithful normal trace $\Tr$ characterized by the formula
$$\Tr (x e_A y) = \tau(xy) \quad\text{for all}\;\; x,y \in M \; .$$
The \emph{Jones index} of the inclusion $A \subset (M,\tau)$ satisfies $[M:A] = \Tr(1)$. The value of $[M:A]$ depends on the choice of tracial state $\tau$. In this article, there will be in all circumstances a natural choice of tracial state, either given by an ambient II$_1$ factor, either as the natural tracial state on $\cL(\Gamma)$. So, when we speak about a \emph{finite index inclusion}, it is always with respect to the naturally present state. In Definition \ref{prop.essential} we will moreover see that this kind of subtlety is not really crucial.

As right $A$-modules, $\rL^2(M)_A \cong p (\ell^2(\N) \ot \rL^2(A))_A$ for some projection $p \in \B(\ell^2(\N)) \ovt A$. One verifies that $[M:A] = (\Tr \ot \tau)(p)$. A canonical index for $A \subset M$ would rather be given by the $\cZ(A)$-valued trace of the right $A$-module $\rL^2(M)_A$.

For completeness, we give a proof of the following elementary lemma.
\begin{lemma} \label{lem.finite-index}
Let $A \subset (M,\tau)$.
\begin{itemize}
\item Suppose that $\rL^2(M)$ is generated as a right $A$-module by $n$ vectors $\xi_1,\ldots,\xi_n \in \rL^2(M)$, meaning that $\rL^2(M)$ is the closure of $\xi_1 A + \cdots + \xi_n A$. Then, $[M:A] \leq n$.
\item If $[M: A] < \infty$ and $\eps > 0$, there exists a central projection $z \in \cZ(A)$ satisfying $\tau(z) > 1- \eps$ such that $\rL^2(Mz)$ is finitely generated as a right $A$-module (and in the sense of the previous item).
\end{itemize}
\end{lemma}
\begin{proof}
Let $\xi \in \rL^2(M)$ and denote by $p \in \langle M,e_A \rangle$ the orthogonal projection onto the closure of $\xi A$. The densely defined operator $M \subset \rL^2(M) \recht \rL^2(M) : x \mapsto \xi E_A(x)$ is closable and the polar decomposition of its closure yields a partial isometry $v \in \langle M,e_A \rangle$ satisfying $vv^* = p$ and $v^*v \leq e_A$. It follows that $\Tr(p) = \Tr(vv^*) = \Tr(v^* v) \leq \Tr(e_A) = 1$. If now $\rL^2(M)$ is generated by $\xi_1,\ldots,\xi_n$, we find in this way projections $p_1,\ldots,p_n$ in $\langle M,e_A \rangle$ satisfying $\Tr(p_i) \leq 1$ for all $i$ and $1 = p_1 \vee \cdots \vee p_n$. So, $\Tr(1) \leq n$ and hence by definition $[M:A] < \infty$.

Suppose now that $[M:A] < \infty$. Denote by $J : \rL^2(M) \recht \rL^2(M)$ the anti-unitary given by $J x = x^*$ for all $x \in M$. We know that $\Tr$ defines a finite faithful normal trace on $\langle M,e_A \rangle$, which is hence a finite von Neumann algebra. Moreover, the center of $\langle M,e_A \rangle$ is given by $J \cZ(A) J$ and $e_A$ is a projection with central support equal to $1$ in $\langle M,e_A \rangle$. Given $\eps > 0$, it follows that we can take a projection $z \in \cZ(A)$ and a finite number of partial isometries $v_1,\ldots,v_n \in \langle M,e_A \rangle$ satisfying $\tau(z) > 1 - \eps$ and $J z J = \sum_{i=1}^n v_i e_A v_i^*$. Viewing $1 \in M$ as a vector in $\rL^2(M)$ and $v_i$ as an operator on $\rL^2(M)$, define $\xi_i = v_i(1)$. We find that $\rL^2(Mz) = JzJ \rL^2(M)$ is generated by $\xi_1,\ldots,\xi_n$ as a right $A$-module.
\end{proof}

We have two reasons to introduce a wider notion of \lq finite index inclusion\rq. This other notion has two advantages~: it is independent of the choice of traces involved and arbitrary direct sums of finite index inclusions remain, what we will call, \emph{essentially of finite index}. The following proposition has nothing new in it~: indeed the construction of the $y_i$ in point~3 below, repeats the construction of a \emph{Pimsner-Popa basis} of a finite index inclusion, see Proposition 1.3 in \cite{PimPo}.

\begin{defprop} \label{prop.essential}
Let $A \subset (M,\tau)$. We say that $A \subset M$ \emph{is essentially of finite index} if one of the following equivalent conditions hold.
\begin{enumerate}
\item For every $\eps > 0$, there exists a projection $p \in M \cap A'$ such that $\tau(p) > 1 - \eps$ and $[pMp : Ap] < \infty$ (w.r.t.\ the trace $\tau(p)^{-1} \tau(\cdot)$ on $pMp$).
\item The trace $\Tr$ on $\langle M,e_A \rangle$ is semifinite on $M' \cap \langle M,e_A \rangle$.
\item For every $\eps > 0$, there exists a projection $p \in M \cap A'$ with $\tau(p) > 1 - \eps$ and elements $y_1,\ldots,y_n \in Mp$ satisfying
$$x p = \sum_{i=1}^n y_i E_A(y_i^* x) \quad\text{for all}\;\; x \in M \; .$$
\end{enumerate}
\end{defprop}
\begin{proof}[Proof of the equivalence of the three conditions]
Denote as above by $J : \rL^2(M) \recht \rL^2(M)$ the anti-unitary given by $J x = x^*$ for all $x \in M$. Note that $M' \cap \langle M,e_A \rangle = J(M \cap A') J$.

Suppose that 1 holds and choose $\eps > 0$. Take $p \in M \cap A'$ with $\tau(p) > 1 - \eps$ and $[pMp : Ap] < \infty$. It follows from Lemma \ref{lem.finite-index} that, after making $p$ slightly smaller but keeping $\tau(p) > 1- \eps$, we have $\rL^2(Mp)$ finitely generated as a right $A$-module. Since $\rL^2(Mp) = JpJ \rL^2(M)$ an argument identical to the first part of the proof of Lemma \ref{lem.finite-index} shows that $\Tr(JpJ) < \infty$. So, we have proven 2.

Suppose that 2 holds and choose $\eps > 0$. Take a projection $p \in M \cap A'$ with $\tau(p) > 1 - \eps$ and $\Tr(JpJ) < \infty$. Then, the formula $x \mapsto \Tr(x JpJ)$ defines a finite trace on $M$. Cutting $p$ with a projection in $\cZ(M)$, but keeping $\tau(p) > 1 - \eps$, we may suppose that $\Tr(x JpJ) \leq \lambda \tau(x)$ for all $x \in M^+$ and some $\lambda > 0$. Finally, cutting $JpJ$ with a projection in $\cZ(\langle M,e_A \rangle) = J \cZ(A) J$ and keeping $\tau(p) > 1-\eps$, we may assume the existence of partial isometries $v_1,\ldots,v_n \in \langle M,e_A \rangle$ satisfying $JpJ = \sum_{i=1}^n v_i e_A v_i^*$. As in the proof of Lemma \ref{lem.finite-index}, we consider $v_i$ as on operator on $\rL^2(M)$ and define $\xi_i \in \rL^2(M)$ such that $\xi_i = v_i(1)$.

We claim that in fact $\xi_i \in Mp$. First note that $\langle a , \xi_i \rangle = \Tr(e_A a^* v_i)$ for all $a \in M$. To prove that $\xi_i \in M$, it is sufficient to check that $(a,b) \mapsto \langle ab^* ,\xi_i \rangle$ is a bounded sesquilinear form on $\rL^2(M)$. This is the case because of
\begin{align*}
|\Tr(e_A b a^* v_i )|^2 & = |\Tr(e_A b a^* v_i e_A)|^2  \leq \Tr(e_A bb^* e_A) \, \Tr(a^* v_i e_A v_i^* a) \\ & \leq \|b\|_2^2 \, \Tr(a^* JpJ a) \leq \lambda \, \|a\|_2^2 \, \|b\|_2^2 \; .
\end{align*}
Since $v_i = JpJ v_i$ and hence $JpJ \xi_i = \xi_i p$, we conclude that $\xi_i \in Mp$. Write $y_i = \xi_i$. One checks that $v_i e_A = y_i e_A$, so that we have shown that
$$JpJ = \sum_{i=1}^n y_i e_A y_i^* \; .$$
This is exactly 3.

If 3 holds, we clearly find for every $\eps > 0$ a projection $p \in M \cap A'$ such that $\tau(p) > 1 - \eps$ and such that $\rL^2(Mp)$ is finitely generated as a right $A$-module. Point 1 then follows from Lemma \ref{lem.finite-index}.
\end{proof}

Also the following lemma is well known, but we include a proof for the convenience of the reader.

\begin{lemma} \label{lem.index-relative-commutant}
Let $A \subset B \subset (M,\tau)$.
\begin{itemize}
\item $[M \cap A' : B \cap A'] \leq [M:B]$.
\item If $B \subset M$ is essentially of finite index, also $B \cap A' \subset M \cap A'$ is essentially of finite index.
\end{itemize}
\end{lemma}

\begin{proof}
Observe that $E_B(x) = E_{B \cap A'}(x)$ whenever $x \in M \cap A'$.
So, the map $\psi(x e_{B \cap A'} y) = x e_B y$ for $x,y \in M \cap A'$ extends to a, possibly non-unital, $\Tr$-preserving embedding $\psi : \langle M \cap A',e_{B \cap A'}\rangle \recht \langle M,e_B \rangle$. It follows that
$$[M \cap A' : B \cap A'] = \Tr_{\langle M \cap A',e_{B \cap A'}\rangle}(1) = \Tr_{\langle M,e_B \rangle}(\psi(1)) \leq \Tr_{\langle M,e_B \rangle}(1) = [M:B] \; .$$
This proves the first point of the lemma. The second point follows from the first point and the observation that $M \cap B' \subset (M \cap A') \cap (B \cap A')'$.
\end{proof}

\end{document}